\documentclass[pdflatex,sn-mathphys-num]{sn-jnl}

\usepackage{orcidlink}

\usepackage{multirow}
\usepackage{listings}%
\usepackage[title]{appendix}%
\usepackage{textcomp}%
\usepackage{manyfoot}%
\usepackage{booktabs}%
\usepackage{rotating}
\usepackage{array}
\usepackage{enumitem}
\usepackage{soul}  
\usepackage{amsthm}

\usepackage{longtable}
\usepackage{placeins} 

\newcolumntype{L}[1]{>{\raggedright\let\newline\\\arraybackslash\hspace{0pt}}m{#1}}
\newcolumntype{C}[1]{>{\centering\let\newline\\\arraybackslash\hspace{0pt}}m{#1}}
\newcolumntype{R}[1]{>{\raggedleft\let\newline\\\arraybackslash\hspace{0pt}}m{#1}}

\usepackage{graphicx}%
\usepackage{tikz}
\usepackage{pgf}
\usetikzlibrary{matrix,calc,shapes.geometric,arrows,arrows.meta,positioning,
                fit,backgrounds,decorations.markings}

\tikzset{
    node/.style = {rectangle, align=center, minimum size=.5cm},
    box/.style = {rectangle, align=center, minimum size=.8cm, thick, draw=black, text=black, fill=white},
    tnode/.style = {rectangle, align=center, inner sep=0cm, outer sep=0cm, minimum size=0cm},
    mirrorbox/.style = {rectangle, align=center, minimum size=.8cm, thick, draw=black, text=white, fill=black},
    smallnode/.style = {circle, draw, fill=black, align=center, minimum size=.1cm, 
                        inner sep=0cm, outer sep=0.05cm},
    vsmallnode/.style = {circle, draw, fill=black, align=center, minimum size=.07cm, 
                        inner sep=0cm, outer sep=0.05cm},
    leaf/.style = {rectangle, draw, fill=black, align=center, minimum size=.1cm, 
                        inner sep=0cm, outer sep=0.05cm},
    empty/.style = {smallnode, draw, fill=white},
    phantom/.style = {smallnode, draw=white, fill=white, minimum size=0cm},
    dot/.style = {circle, anchor=center, align=center, fill=black,
                outer sep=0, inner sep=.3ex},
    point/.style = {circle, anchor=center, align=center,
                outer sep=0, inner sep=0},
    arrowhead/.style = {isosceles triangle, draw, fill=black, inner sep=0, outer sep=.05cm, 
                minimum size=.13cm},
    labelleft/.style = {left = .1 of #1.south west, outer sep=0, inner sep=0, anchor=base east},
    labelright/.style = {right = .1 of #1.south east, outer sep=0, inner sep=0, anchor=base west},
    labelabove/.style = {above right = .05 and -.1 of #1.north, outer sep=0, inner sep=0, anchor=base west},
    labelaboveright/.style = {above right = .05 and .1 of #1.north, outer sep=0, inner sep=0, anchor=base west},
    labelaboveleft/.style = {above left = .05 and .1 of #1.north, outer sep=0, inner sep=0, anchor=base west},
    labelbelow/.style = {below right = .3 and -.1 of #1.south, outer sep=0, inner sep=0, anchor=base west},
    strand/.style={thick, looseness=1.2},
    arrow/.style = {-Latex},
    dottededge/.style = {dotted, thick},
    dirarrow/.style={-Latex, dashed, thin, blue, looseness=1.2},
    midarrow/.style={postaction={decorate,decoration={
                markings, mark=at position .7 with {\arrow{Latex[fill=white, scale=1.5]}}}}},
    tree/.style = {matrix, row sep={1cm,between origins},
                column sep={1cm,between origins}, matrix anchor=#1.center,
                every node/.style={smallnode}},
    tree_mid/.style = {matrix, row sep={.7cm,between origins},
                column sep={.7cm,between origins}, matrix anchor=#1.center,
                every node/.style={smallnode}},
    tree_small/.style = {matrix, row sep={.5cm,between origins},
                column sep={.5cm,between origins}, matrix anchor=#1.center,
                every node/.style={vsmallnode}}
}

\usepackage{amsmath,amssymb,amsfonts}%
\usepackage{amsthm}%
\usepackage{mathrsfs}%
\usepackage{algorithm}%
\usepackage{algorithmicx}%
\usepackage{algpseudocode}%
\usepackage{mathtools}
\usepackage{dsfont} 
\usepackage{bm}
\usepackage{xfrac}

\algblock[class]{class}{ClassEnd}
\algnotext[class]{ClassEnd}

\theoremstyle{plain}%
\newtheorem{theorem}{Theorem}%
\newtheorem{conjecture}{Conjecture}

\newtheorem{lemma}{Lemma}

\theoremstyle{definition}%
\newtheorem{definition}{Definition}%

\theoremstyle{remark}%

\newcommand{\qq}[1]{``#1''}

\newcommand{\subfigref}[2]{\hyperref[#1]{Figure~\ref*{#1}#2}}
\newcommand{\tables}[1]{\hyperref[#1]{Tables~\ref*{#1}}}
\newcommand{\tangle}{\mathcal{T}}
\newcommand{\algebraic}{\tangle_\mathbb{A}}

\newcommand{\rational}{\tangle_\mathbb{Q}}
\newcommand{\integral}{\tangle_\mathbb{Z}}
\newcommand{\basic}{\tangle_\mathds{1}} 

\newcommand{\defeq}{\mathrel{\overset{\makebox[0pt]{\mbox{\tiny def.}}}{\;=\;}}}
\newcommand{\reidone}{\mathrel{\overset{\makebox[0pt]{\mbox{\tiny I}}}{\;=\;}}}
\newcommand{\reidtwo}{\mathrel{\overset{\makebox[0pt]{\mbox{\tiny II}}}{\;=\;}}}
\newcommand{\reidthree}{\mathrel{\overset{\makebox[0pt]{\mbox{\tiny III}}}{\;=\;}}}
\newcommand{\flype}{\mathrel{\overset{\makebox[0pt]{\mbox{\tiny F}}}{\;=\;}}}
\newcommand{\twist}{\mathrel{\overset{\makebox[0pt]{\mbox{\tiny T}}}{\;=\;}}}
\newcommand{\ring}{\mathrel{\overset{\makebox[0pt]{\mbox{$\circ$}}}{\;=\;}}}

\newcommand{\arra}[1]{\xrightarrow{\alpha_{#1}}}
\newcommand{\arrb}[1]{\xrightarrow{\beta_{#1}}}

\newlength\rightshift
\begin{document}
\title[Classification of algebraic tangles]{Classification of algebraic tangles}

\author*[1]{\fnm{Bartosz A.} \sur{Gren}\,\orcidlink{0000-0001-9971-6807}} \email{b.gren@cent.uw.edu.pl}
\author[1]{\fnm{Joanna I.} \sur{Sulkowska}\,\orcidlink{0000-0003-2452-0724}} \email{j.sulkowska@cent.uw.edu.pl}
\author[2,3]{\fnm{Bo\v{s}tjan} \sur{Gabrov\v{s}ek}\,\orcidlink{0000-0002-8272-5392}} \email{bostjan.gabrovsek@pef.uni-lj.si}

\affil[1]{\orgname{University of Warsaw, Centre of New Technology}, \orgaddress{S. Banacha 2C, \city{02-097 Warsaw}, \country{Poland}}}
\affil[2]{\orgname{University of Ljubjana, Faculty of Education, Kardeljeva plo\v{s}\v{c}ad 16}, \orgaddress{\city{1000 Ljubjana}, \country{Slovenia}}}
\affil[3]{\orgname{Rudolfovo – Science and Technology Centre}, \orgaddress{Podbreznik 15, \city{8000 Novo mesto}, \country{Slovenia}}}

\abstract{
We study algebraic tangles as fundamental components in knot theory, developing a systematic approach to classify and tabulate prime tangles using a novel canonical representation. The canonical representation enables us to distinguish mutant tangles, which fills the gaps in previous classifications. Moreover, we increase the classification of prime tangles up to 14 crossings and analyze tangle symmetry groups. We provide a database of our results: \url{https://tangleinfo.cent.uw.edu.pl}.
}

\keywords{Tangles, Algebraic Tangles, Rational Tangles, Mutant knots, Knot theory, Canonical form}

\maketitle
\section{Introduction}

The formal study of tangles was introduced by John Conway in the 1970s \cite{conway1970enumeration}. Conway’s work revolutionized the way mathematicians approached knot theory, providing new methods for decomposing and analyzing knots. Conway showed that a knot or link can be decomposed into tangles, allowing us to represent knots and links in a concise and systematic way. 
A tangle consists of two arcs and a finite number of circles embedded in a three-dimensional ball such that the endpoints of the arcs lie on the boundary of the ball. 

A particularly important subclass of tangles are rational tangles,
which, due to their simplicity, have been extensively studied in the literature, as the continued fraction associated with a rational tangle is a perfect invariant. 
One notable application of studying tangles is in the development of knot tables, since tangles offer a very compact and efficient way to encode a knot.

However, creating knot tables through the use of tangles involves the classification of algebraic tangles~\cite{conway1970enumeration}, which are much more difficult to study, and the theory of algebraic tangles remains undeveloped -- a gap that this paper aims to address. 
Knot tangle decompositions were also studied in the context of arborescent knots by Bonahon and Siebenmann \cite{bonahon2010new}.

The concept of knot classification follows a straightforward approach: first, construct an initial set of knot diagrams that comprehensively includes all possible knots; then, systematically eliminate duplicate representations to obtain a distinct classification.
The second step presents a computational bottleneck; therefore, it is advantageous to start with the smallest possible initial set. A naive approach to generate an initial set of knots with up to $N$ crossings involves a generation of all possible 4-valent planar graphs (with at most $N$ vertices) and then a replacement of the vertices with crossings.
A more effective approach is to generate 4-valent polyhedral graphs and replace the vertices with algebraic tangles -- using this method, Conway managed to tabulate knots with up to 11 crossings and links with up to 10 crossings by hand \cite{conway1970enumeration}.
Due to the lack of a tabulation of algebraic tangles, later classifications were created using the naive approach, supplemented by tangle properties, which led to the construction of knot tables up to 13 crossings (in 1983), 16 crossings (in 1998), and 19 crossings (in 2012) \cite{dowker1983classification, hoste1998first, burton2020next}. With a tabulation of algebraic tangles, knot and link tables could be easily extended.

{Beyond their role in knot classification, tangles have found diverse applications in knot theory; for instance, decomposing a knot into tangles is essential for faster computations of the computationally intensive Khovanov homology~\cite{bar2007fast}, a theory that provides a powerful categorification of the Jones polynomial and has advanced the study of knot invariants, 4-dimensional topology, and quantum algebra.}

Our paper is highly motivated by the need to create tables of various entangled structures that appear in the study of biological structures. 
Knot tables initiated the study of entanglement in proteins \cite{mansfield1994there,taylor2000deeply,sulkowska2012conservation}; later, proteins with lasso motifs were also studied \cite{niemyska2016complex,dabrowski2016lassoprot,gren2021lasso}.
Moriuchi’s work on classifying $\Theta$-curves and handcuff links up to seven crossings~\cite{moriuchi2008enumeration} was essential for the study of $\Theta$-curves in proteins~\cite{dabrowski2024theta, bruno2024knots, sulkowska2020folding}, and we wish to extend these studies to analyse protein as bonded knots \cite{gabrovvsek2021invariant, adams2020knot,dabrowski2019knotprot}, bonded links \cite{dabrowski2016linkprot,dabrowski2017topological}, and bonded knotoids \cite{goundaroulis2020knotoids, gabrovvsek2023invariants, gugumcu2022invariants}.

{
Beyond topological classification, tangle theory also plays a crucial role in understanding biological processes. For example, the study of topoisomerase enzymes~\cite{cameron2023coinhibition}, which manage the topology of DNA, illustrates how tangles are manipulated during cellular processes. Tangles are also used to understand the mechanisms of DNA recombination~\cite{Ernst_Sumners_1990, sumners2011dna}.
}

{This paper is structured as follows: In \autoref{sec:preliminaries}, we provide the necessary background and key definitions to establish the scope of the paper. In \autoref{sec:isotopic_moves}, we define isotopy preserving moves, up to which we provide a classification. In \autoref{sec:tree}, we introduce a binary tree notation as the foundation for a canonical representation of algebraic tangles, the central objective of this article. We further propose an algorithm to construct this representation and establish its uniqueness. In \autoref{sec:classification}, we present our tangle classification, which is further expanded in \autoref{sec:symmetries} where we study tangle symmetries.}

\section{Preliminaries}\label{sec:preliminaries}

\begin{figure}[t]
    \begin{tikzpicture}[x={(1.2cm,0)}, y={(0,1.2cm)}, z={(-.4cm, -.3cm)}, baseline=(a), scale=1.1]
        \coordinate (zero) at (0,0);
        \node[node] (a) at (-1.8,1.6) {a)};
        \node[dot, label={45:{{\small NE}}}] at (1,1,0) (NE) {};
        \node[dot, label={135:{{\small NW}}}] at (-1,1,0) (NW) {};
        \node[dot, label={-135:{{\small SW}}}] at (-1,-1,0) (SW) {};
        \node[dot, label={-45:{{\small SE}}}] at (1,-1,0) (SE) {};
        \draw[help lines, red, thick];
        \draw[very thick]
            circle[radius=1.41]
            (0,1.41) arc(90:270:.4 and 1.41)
            (-1.41,0) arc(180:360:1.41 and .3);
        \draw[very thick, loosely dashed, dash phase=2pt, cap=round] 
            (0,-1.41) arc(-90:90:.4 and 1.41);
        \draw[very thick, loosely dashed, cap=round] 
            (1.41,0) arc(0:180:1.41 and .3);
        \node[point, label=90:{$x$}] at (1.8,0,0) (xcoor) {};
        \node[point, label=-30:{$y$}] at (0,1.8,0) (ycoor) {};
        \node[point, label=-45:{$z$}] at (0,0,2) (zcoor) {};
        \draw[arrow] (0,0,0) -- (xcoor);
        \draw[arrow] (0,0,0) -- (ycoor);
        \draw[arrow] (0,0,0) -- (zcoor);
        \draw[dotted, semithick, cap=round] 
            (NW) -- (SE) node[below, pos=.05, scale=1] {$p$} ;
    \end{tikzpicture}
    \begin{tikzpicture}[baseline=(b), scale=.55]
        \node[node] (b) at (-1.7,1.3) {b)};
        \node[node, anchor=base] at (0,-1.6) {$0$};
        \node[dot] at (1,1) (NE0) {};
        \node[dot] at (-1,1) (NW0) {};
        \node[dot] at (-1,-1) (SW0) {};
        \node[dot] at (1,-1) (SE0) {};
        \draw[strand]
            (NW0) edge[out=-45, in=-135] (NE0)
            (SW0) edge[out=45, in=135] (SE0);
        \node[node, anchor=base] at (3,-1.6) {$\infty$};
        \node[dot] at (4,1) (NEinf) {};
        \node[dot] at (2,1) (NWinf) {};
        \node[dot] at (2,-1) (SWinf) {};
        \node[dot] at (4,-1) (SEinf) {};
        \draw[strand]
            (SWinf) edge[out=45, in=-45] (NWinf)
            (SEinf) edge[out=135, in=-135] (NEinf);
        \node[node, minimum size=.2cm] (cross1) at (6,0) {};
        \node[node, anchor=base] at (6,-1.6) {$1$};
        \node[dot] at (7,1) (NE1) {};
        \node[dot] at (5,1) (NW1) {};
        \node[dot] at (5,-1) (SW1) {};
        \node[dot] at (7,-1) (SE1) {};
        \draw[strand]
            (SW1) edge[out=45, in=-135] (NE1)
            (SE1) edge[out=135, in=-45] (cross1)
            (NW1) edge[out=-45, in=135] (cross1);
        \node[node, minimum size=.2cm] (cross2) at (9,0) {};
        \node[node, anchor=base] at (9,-1.6) {$\overline{1}$};
        \node[dot] at (10,1) (NE2) {};
        \node[dot] at (8,1) (NW2) {};
        \node[dot] at (8,-1) (SW2) {};
        \node[dot] at (10,-1) (SE2) {};
        \draw[strand]
            (SW2) edge[out=45, in=-135] (cross2)
            (cross2) edge[out=45, in=-135] (NE2)
            (NW2) edge[out=-45, in=135] (SE2);
        \node[node] (c) at (-1.7,-2.7) {c)};
        \coordinate (LT1) at (-1,-2.75);
        \coordinate (LB1) at (-1,-5.25);
        \coordinate (RT1) at (4,-2.75);
        \coordinate (RB1) at (4,-5.25);
        \node[box] at (0.25,-4) (R1) {R};
        \node[box] at (2.75,-4) (F1) {F};
        \draw[strand]
            (R1) edge[out=45, in=135] (F1)
            (R1) edge[out=-45, in=-135] (F1)
            (LT1) edge[out=-45, in=135] (R1)
            (LB1) edge[out=45, in=-135] (R1)
            (F1) edge[out=45, in=-135] (RT1)
            (F1) edge[out=-45, in=135] (RB1);
        \node[node, anchor=base] at (1.5,-5.8) {$R\!+\!F$};
        \coordinate (LT2) at (5,-2.75);
        \coordinate (LB2) at (5,-5.25);
        \coordinate (RT2) at (10,-2.75);
        \coordinate (RB2) at (10,-5.25);
        \node[box] at (6.25,-4) (R2) {\rotatebox{90}{\reflectbox{R}}};
        \node[box] at (8.75,-4) (F2) {F};
        \draw[strand]
            (R2) edge[out=45, in=135] (F2)
            (R2) edge[out=-45, in=-135] (F2)
            (LT2) edge[out=-45, in=135] (R2)
            (LB2) edge[out=45, in=-135] (R2)
            (F2) edge[out=45, in=-135] (RT2)
            (F2) edge[out=-45, in=135] (RB2);
        \node[node, anchor=base] at (7.5,-5.8) {$RF$};
    \end{tikzpicture}
    \caption{a) A Conway's sphere, in which the tangle is embedded. b) Basic tangles which satisfy the boundary. c) Sum and product of tangles R and F.}
    \label{fig:sphere}
\end{figure}
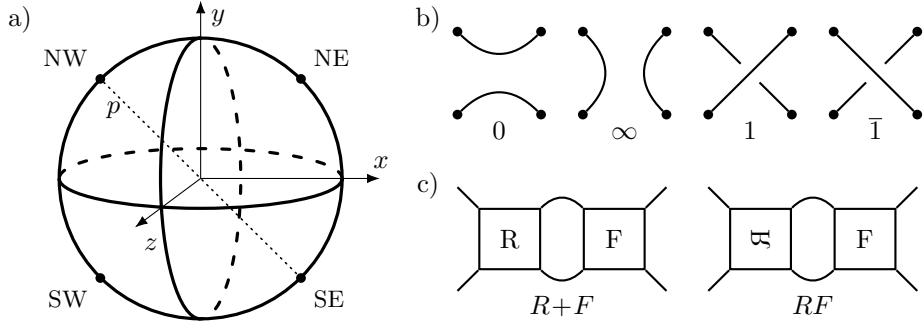

\begin{definition}[Tangle]\label{def:tangle}
    A \emph{tangle} is an embedding of two arcs (homeomorphic to the interval $[0,1]$) and circles into the unit 3-ball $B^3 = \{(x,y,z) \in \mathbb{R}^3 \mid x^2+y^2+z^2 \leq 1\}$, where only the endpoints of the two arcs lie on the boundary of the ball, $\partial B^3$, and are located at the four points
$$\textrm{NW} = \left(\tfrac{-1}{\sqrt{2}}, \tfrac{1}{\sqrt{2}}, 0\right),\;
\textrm{NE} = \left(\tfrac{1}{\sqrt{2}}, \tfrac{1}{\sqrt{2}}, 0\right),\; 
\textrm{SW} = \left(\tfrac{-1}{\sqrt{2}}, \tfrac{-1}{\sqrt{2}}, 0\right), \textrm{ and}\; 
\textrm{SE} = \left(\tfrac{1}{\sqrt{2}}, \tfrac{-1}{\sqrt{2}}, 0\right).$$
\end{definition}
The sphere $\partial B^3$, which the tangle intersects at these four points, is referred to as the \emph{Conway sphere}.
\noindent Furthermore, we define the \emph{principal diagonal} axis as the line $p = \{(x,-x,0) \in \mathbb{R}^3  \mid x \in \mathbb{R}\}$, which connects the NW-SE boundary points, see \subfigref{fig:sphere}{a}.

\begin{definition}[VHX]\label{def:vhx}
{We classify tangles based on how the open arcs connect the boundary points, grouping them into three distinct classes:}
    \begin{itemize}
    \item V-type tangle: open arcs connect points vertically (NW-SW and NE-SE points),
    \item H-type tangle: open arcs connect points horizontally (NW-NE and SW-SE points),
    \item X-type tangle: if open arcs connect points diagonally (NW-SE and SW-NE points).
    \end{itemize}
\end{definition}

\begin{definition}[Basic tangle]\label{def:tangles_basic}
    The tangles $0$, $\infty$, $1$, $\overline{1}$ $\in \basic$ depicted in \subfigref{fig:sphere}{b} are called \emph{basic tangles}.
\end{definition}
\noindent Basic tangles are building blocks for constructing more complex tangles, which can be formed by joining simpler ones through the identification of arc boundaries.

\begin{definition}[Sum and product of tangles]\label{def:operations}
The \emph{sum} $A\!+\!B$ and \emph{product} $AB$ of two tangles $A$ and $B$ are the operations depicted on \subfigref{fig:sphere}{c}.
\end{definition}

\begin{definition}[Algebraic tangle]\label{def:tangles_algebraic}
    An \emph{algebraic tangle} is a tangle $A \in \algebraic$ which can be represented as any composition of sums abd products of basic tangles. 
\end{definition}
Since every algebraic tangle can be expressed as a product of two subtangles $A=LR$, we can represent algebraic tangles as binary expression trees:
\begin{figure}[h!]
    \centering
    \begin{tikzpicture}
        \node at (0,0) (one) {$LR=$};
        \node[tree_mid={mid}, row sep={1cm}, column sep={1cm}, row sep={.45cm}] at (.65,0) (tree1)
        {
            \node (T1) {}; & \node (R1) {};\\
            \node[phantom] (mid) {}; \\
            \node (L1) {};\\
        };
        \draw (T1) edge (R1)
              (T1) edge (L1);
        \node[labelright={R1}] {$R$};
        \node[labelright={L1}] {$L$};
        \node at (3.4,0) (one) {$\quad\qquad \text{or}\qquad A=LR=$};
        \node[tree_mid={mid}, row sep={1cm}, column sep={1cm}, row sep={.45cm}] at (5.45,0) (tree2)
        {
            \node (T1) {}; & \node (R1) {};\\
            \node[phantom] (mid) {}; \\
            \node (L1) {};\\
        };
        \draw (T1) edge (R1)
              (T1) edge (L1);
        \node[labelleft={T1}] {$A$};
        \node[labelright={R1}] {$R$};
        \node[labelright={L1}] {$L$};
    \end{tikzpicture}
\end{figure}

{
As a tangle can be represented in many different ways, our first goal is to determine whether two algebraic tangles are the same or different. We define “sameness” in two distinct ways: as isotopic or equivalent, as outlined below.}
\begin{definition}[Isotopy]
    Two tangles $A$ and $B$ are \textit{isotopic} if there exists an orientation-preserving self-homeomorphism 
    $h: (B^3, A) \rightarrow (B^3, B)$ that is the identity map on the boundary ($h|_{\partial B^3} = \mathrm{Id}$). By slight abuse of notation, we will consider two isotopic tangles equal and will use the notation $A=B$.   
\end{definition}
\begin{definition}[Equivalence]
    Two tangles $A$ and $B$ are \textit{equivalent}, denoted by $A \sim B$, if there exists an self-homeomorphism 
    $h: (B^3, A) \rightarrow (B^3, B)$. 
\end{definition} 
\noindent Note that equivalence is weaker than isotopy, since equivalence does not require the four endpoints NW, NE, SE, and SW to remain fixed.

\begin{figure}[h!]
    \centering
    \begin{tabular}{c@{\quad\quad\quad}c@{\quad\quad\quad}c}  
        \includegraphics[scale=1.75, page=5]{reid.pdf} & 
        \includegraphics[scale=1.75, page=3]{reid.pdf} & 
        \includegraphics[scale=1.75, page=1]{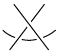} \\
        \tikz{\draw[{Latex}-{Latex}, very thin] (0,1) -- (0,0);} & 
        \tikz{\draw[{Latex}-{Latex}, very thin] (0,1) -- (0,0);} & 
        \tikz{\draw[{Latex}-{Latex}, very thin] (0,1) -- (0,0);} \\
        \includegraphics[scale=1.75, page=6]{reid.pdf} & 
        \includegraphics[scale=1.75, page=4]{reid.pdf} & 
        \includegraphics[scale=1.75, page=2]{reid.pdf} \\
        a) Type I &  b) Type II & c) Type III
    \end{tabular}
    \caption{Reidemeister moves of Type I, II, and III.}
    \label{fig:reidemeister_moves}
\end{figure}

{
\begin{definition}[Tangle diagram] A \textit{diagram} of a tangle is the regular projection of the tangle to the disk 
$D$ obtained by the intersection of $B^3$ and the plane $z=0$. By a general position argument, we will assume that the only singularities are a finite collection of double points, called \emph{crossings}, which hold also over/under information about the projection.
\end{definition}
}
{
As in the case of knots, isotopy is generated by planar isotopy of $\text{Int}\,D$ and the three Reidemeister moves depicted in \autoref{fig:reidemeister_moves}.} 
\begin{theorem}[\cite{conway1970enumeration}]\label{thm:reidemeister}
    Two tangles are isotopic if and only if their diagrams differ by planar isotopy and a finite sequence of Reidemeister moves.   
\end{theorem}

\begin{figure}[t]
    \centering
    \begin{tikzpicture}[scale=.6]
        \path [use as bounding box] (-2.4,1.6) rectangle (16.5,-6.5);
        \node[node] at (-2,1.2) {a)};
        \node[node] at (13,1.2) {b)};
        \coordinate (zero) at (0,0);
        \coordinate (NW) at (-1.25, 1.25);
        \coordinate (NE) at (1.25, 1.25);
        \coordinate (SE) at (1.25, -1.25);
        \coordinate (SW) at (-1.25, -1.25);
        \coordinate (label) at (0,-2);
        \draw[strand]
            (NW) edge[out=-45, in=135] (SE)
            (SW) edge[out=45, in=-135] (NE);
        \node[box] at (zero) {R};
        \node[node, anchor=base] at (label) {$R$};
        \begin{scope}[transform canvas={xshift=2cm}]
            \draw[strand]
                (NW) edge[out=-45, in=135] (SE)
                (SW) edge[out=45, in=-135] (NE);
            \node[box] at (zero) {\rotatebox{90}{R}};
            \node[node, anchor=base] at (label) {$\nu R$};  
        \end{scope}
        \begin{scope}[transform canvas={xshift=4cm}]
            \draw[strand]
                (NW) edge[out=-45, in=135] (SE)
                (SW) edge[out=45, in=-135] (NE);
            \node[box] at (zero) {\rotatebox{180}{R}};
            \node[node, anchor=base] at (label) {$\nu^2 \!R$};  
        \end{scope}
        \begin{scope}[transform canvas={xshift=6cm}]
            \draw[strand]
                (NW) edge[out=-45, in=135] (SE)
                (SW) edge[out=45, in=-135] (NE);
            \node[box] at (zero) {\rotatebox{-90}{R}};
            \node[node, anchor=base] at (label) {$\nu^3 \!R$};  
        \end{scope}
        \begin{scope}[transform canvas={xshift=9cm}]
            \draw[strand]
                (NW) edge[out=-45, in=135] (SE)
                (SW) edge[out=45, in=-135] (NE);
            \node[box] at (zero) {\reflectbox{\rotatebox{-90}{R}}};
            \node[node, anchor=base] at (label) {$\eta R$};  
        \end{scope}
        \begin{scope}[transform canvas={yshift=-2.5cm}]
            \draw[strand]
                (NW) edge[out=-45, in=135] (SE)
                (SW) edge[out=45, in=-135] (NE);
            \node[mirrorbox] at (zero) {\reflectbox{\textbf{R}}};
            \node[node, anchor=base] at (label) {$\mu\eta\nu R$};  
        \end{scope}
        \begin{scope}[transform canvas={xshift=2cm, yshift=-2.5cm}]
            \draw[strand]
                (NW) edge[out=-45, in=135] (SE)
                (SW) edge[out=45, in=-135] (NE);
            \node[mirrorbox] at (zero) {\reflectbox{\rotatebox{90}{\textbf{R}}}};
            \node[node, anchor=base] at (label) {$\mu\eta\nu^2 \!R$};  
        \end{scope}
        \begin{scope}[transform canvas={xshift=4cm, yshift=-2.5cm}]
            \draw[strand]
                (NW) edge[out=-45, in=135] (SE)
                (SW) edge[out=45, in=-135] (NE);
            \node[mirrorbox] at (zero) {\reflectbox{\rotatebox{180}{\textbf{R}}}};
            \node[node, anchor=base] at (label) {$\mu\nu\eta R$};  
        \end{scope}
        \begin{scope}[transform canvas={xshift=6cm, yshift=-2.5cm}]
            \draw[strand]
                (NW) edge[out=-45, in=135] (SE)
                (SW) edge[out=45, in=-135] (NE);
            \node[mirrorbox] at (zero) {\reflectbox{\rotatebox{-90}{\textbf{R}}}};
            \node[node, anchor=base] at (label) {$\mu\eta R$};
        \end{scope}
        \begin{scope}[transform canvas={xshift=9cm, yshift=-2.5cm}]
            \draw[strand]
                (NW) edge[out=-45, in=135] (SE)
                (SW) edge[out=45, in=-135] (NE);
            \node[mirrorbox] at (zero) {\textbf{R}};
            \node[node, anchor=base] at (label) {$\mu R$};  
        \end{scope}
    \end{tikzpicture}
        \caption{ Tangle $R$ transformed by rotations and reflections represented by a two-faced square with letter \qq{R}.
        a) Set of all tangles equivalent to tangle $R$ ordered by rotations: $\nu$ (horizontally), $\rho_y$ (vertically).
        b) Reflections of tangle $R$: $\eta R \equiv R0$, and $\mu R \equiv \overline{R}$. }
    \label{fig:equivalent}
\end{figure}
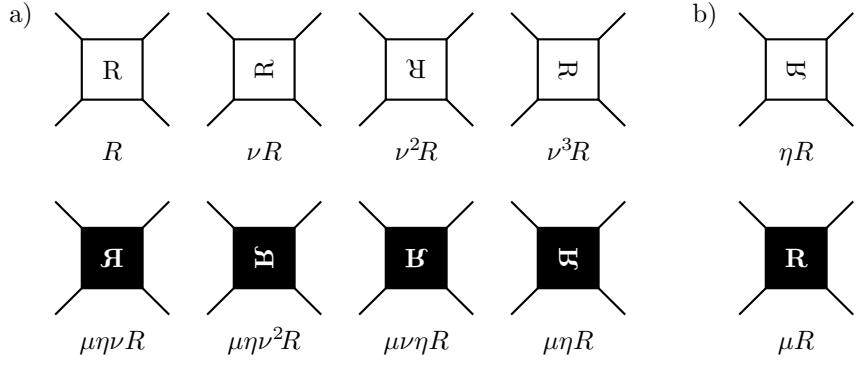

Naturally, we can generate new tangles by reflecting or rotating existing ones (see \autoref{fig:equivalent}). We will study the symmetry group of each tangle by examining its invariance under the following operations.
\begin{definition}[Rotations and reflections]\label{def:transformations}
The transformation $\mu$ is the reflection through the $xy$ plane, $\eta$ is the reflection through the $zp$ plane, $\rho_i$ are rotations around the axis $i \in \{x,y,z\}$ by angle $\pi$, and $\nu$ is the counterclockwise rotation around the $z$ axis by angle $\pi/2$. Explicitly:
\begin{align*}
    \mathmakebox[5cm][l]{\mu: \left(x,y,z\right) \rightarrow \left(x,y,-z\right),}
    \mathmakebox[5cm][l]{\rho_x = \mu\nu\eta: \left(x,y,z\right) \rightarrow \left(x,-y,-z\right),}\\
    \mathmakebox[5cm][l]{\eta: \left(x,y,z\right) \rightarrow \left(-y,-x,z\right),}
    \mathmakebox[5cm][l]{\rho_y = \mu\eta\nu: \left(x,y,z\right) \rightarrow \left(-x,y,-z\right),}\\
    \mathmakebox[5cm][l]{\nu: \left(x,y,z\right) \rightarrow \left(-y,x,z\right),} 
    \mathmakebox[5cm][l]{\rho_z =\nu^2\hspace{.7 em}: \left(x,y,z\right) \rightarrow \left(-x,-y,z\right).}
\end{align*}
\end{definition}

\section{Notation and isotopy preserving moves}\label{sec:isotopic_moves}

We use the following special notations for the reflection operators $\mu$ and $\eta$:
\begin{align}
    \mathmakebox[5cm][l]{\mu(A) \equiv \overline{A} \equiv -A,}
    \mathmakebox[5cm][l]{\eta(A) \equiv A0.} \label{eq:refl}
\end{align}
\noindent It is easy to check graphically that rotation $\rho_x$ is additive and rotation $\rho_y$ is antiadditive in the following sense:
\begin{align}\label{eq:rotations_sum}
    \mathmakebox[5cm][l]{[A\!+\!B]_x = A_x\!+\!B_x,}
    \mathmakebox[5cm][l]{[A\!+\!B]_y = B_y\!+\!A_y.}
\end{align}
It is also easy to check graphically that $\rho_x\eta = \eta\rho_y$, from where we obtain the following rules for multiplication:
\begin{align}\label{eq:rotations_multiplication}
    \mathmakebox[5cm][l]{\left[A0\right]_x = A_y0,}
    \mathmakebox[5cm][l]{\left[A0\right]_y = A_x0,} \\
    \mathmakebox[5cm][l]{\left[AB\right]_x = A_yB_x,}
    \mathmakebox[5cm][l]{\left[AB\right]_y = (B_y0)(A_x0).}
\end{align}

 Any tangle can be represented using sum and product operations (\autoref{def:operations}), and these operations can be used interchangeably, each with its own strengths and weaknesses.
These operations are related through the following equalities:
\begin{equation}\label{eq:addition_multiplication}
    \mathmakebox[5cm][r]{A \!+\! B = A0B,} 
    \quad \phantom{\iff} \quad
    \mathmakebox[5cm][l]{AB = A0 \!+\! B .}
\end{equation}

Addition can be fully expressed by multiplication, but the contrary is not true. 
The fact that 0 is the neutral tangle under addition translates in multiplicative form as follows:
\begin{equation}\label{eq:bracket_juggling}
    \mathmakebox[5cm][r]{0\!+\!A = A = A\!+\!0}
    \quad \iff \quad 
    \mathmakebox[5cm][l]{00A = A = A00} 
\end{equation}
Associativity of addition translates to the following \qq{bracket juggling} multiplication rule:
\begin{equation}\label{eq:associativity}
    \mathmakebox[5cm][r]{(A\!+\!B)\!+\!C = A\!+\!(B\!+\!C)}
    \quad \iff \quad  
    \mathmakebox[5cm][l]{(A0B)0C = A0(B0C),} 
\end{equation}
and more general rule from the multiplication representation perspective:
\begin{equation}
    \mathmakebox[5cm][r]{(A0\!+\!B)\!+\!C = A0\!+\!(B\!+\!C)}
    \quad \iff \quad  
    \mathmakebox[5cm][l]{AB0C = A(B0C),} 
\end{equation}
which are exceptions to the left-associativity of tangle multiplication:
$ABC \coloneqq((AB)C).$
\begin{definition}[Elementary moves]\label{def:elementary_moves}
    Elementary moves I and II (\autoref{fig:twist_flype}) are two tangles moves which are directly translated from I and II Reidemeister moves (\autoref{thm:reidemeister}). Algebraically, we represent them as follows (note that $00=\infty$):
    \begin{align*}
        \mathmakebox[5cm][r]{\infty\!+\!1 \reidone \infty}
        \quad \iff \quad 
        \mathmakebox[5cm][l]{01 \reidone 00} \\
        \mathmakebox[5cm][r]{1+\overline{1} \reidtwo 0 \reidtwo 1\!+\!\overline{1},}
        \quad\iff\quad 
        \mathmakebox[5cm][l]{10\overline{1} \reidtwo 0 \reidtwo \overline{1}01.} 
    \end{align*}
\end{definition}

If we have a sum of $1$ and $\overline{1}$ tangles, we can remove any pairs of neighboring 1 and $\overline{1}$ tangles due to the IInd Reidemeister move and transform it to a tangle consisting only of 1's or $\overline{1}$'s. We call such tangles integral tangles.
\begin{definition}[Integral tangle]\label{def:tangles_integral}
    An \emph{integral tangle} $n \in \integral$ represented by an integer is a tangle created by sum of $1$ and $-1$ tangles as follows: 
    \begin{equation*}
        \mathmakebox[5.8cm][r]{n = \sum_{i=1}^n 1 = \underbrace{1\!+1\!+\cdots\!+\!1}_n \stackrel{\eqref{eq:associativity}}{=} 1010\cdots\!01;} \qquad
        \mathmakebox[5.8cm][l]{-n = \sum_{i=1}^n \overline{1} = \underbrace{\overline{1}\!+\overline{1}\!+\cdots\!+\!\overline{1}}_n \stackrel{\eqref{eq:associativity}}{=} \overline{1010\cdots01}}
    \end{equation*}
\end{definition}

Next, we define three higher-order moves, namely the \emph{twist}, the \emph{flype}, and the \emph{ring move} (\autoref{fig:twist_flype}), which can be expressed as sequences of Reidemeister II and III moves. As we will see, these moves have convenient algebraic and geometric interpretations.
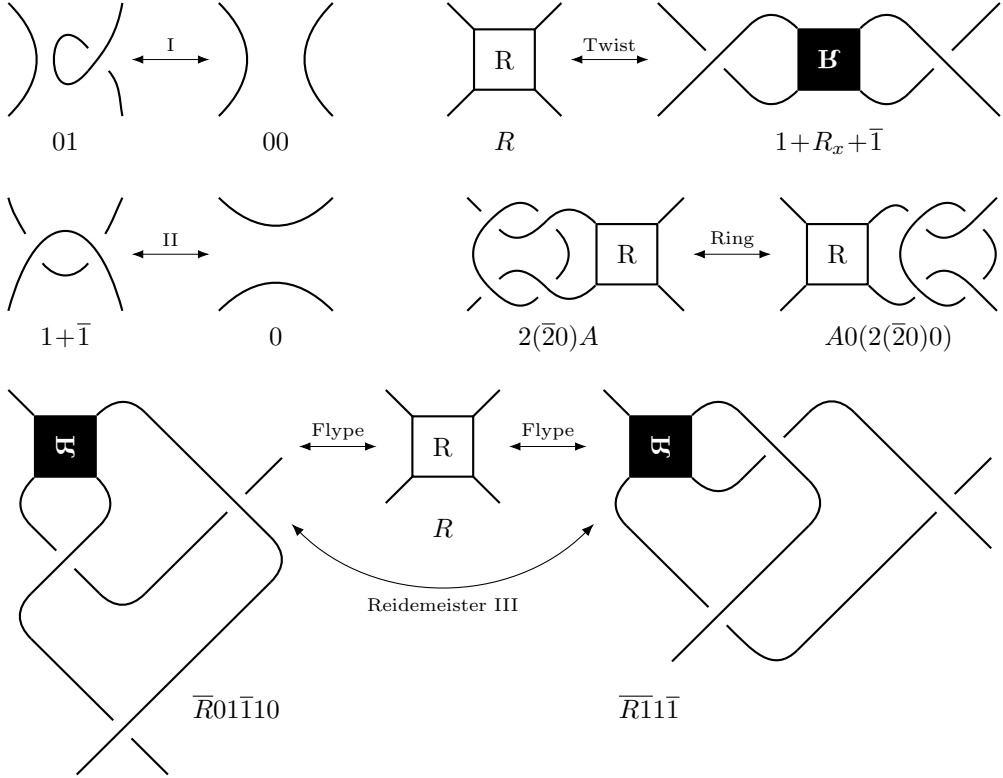
\begin{figure}
    \centering
    \begin{tikzpicture}[scale=.15, baseline=(zero)]
        \coordinate (zero) at (0,0);
        \coordinate (label) at (0,-8);
        \coordinate (A1) at (-5, 5);
        \coordinate (A2) at (-5, -5);
        \coordinate (B1) at (5, 5);
        \coordinate (B2) at (2, -1);
        \coordinate (B3) at (2, 1);
        \coordinate (B4) at (3.8,-1);
        \coordinate (B5) at (5, -5);
        \draw[strand]
            (A1) edge[out=-45, in=45] (A2);;
        \draw[strand, looseness=1]
            (B1) edge[out=-100, in=50] (B2)
            (B4) edge[out=-50, in=100] (B5)
            ;
        \draw[strand, looseness=8]
            (B2) edge[out=230, in=-230] (B3);
        \node[node, anchor=base] at (label) {$01$};
    \end{tikzpicture}
    \begin{tikzpicture}[baseline=(zero)]
        \coordinate (zero) at (0,0);
        \node at (0,.2) {\footnotesize I};
        \draw[{Latex}-{Latex}, very thin] (.5,0) -- (-.5,0); 
    \end{tikzpicture}
    \begin{tikzpicture}[scale=.15, baseline=(zero)]
        \coordinate (zero) at (0,0);
        \coordinate (label) at (0,-8);    
        \coordinate (NW) at (-5, 5);
        \coordinate (NE) at (5, 5);
        \coordinate (SE) at (5, -5);
        \coordinate (SW) at (-5, -5);
        \draw[strand]
            (SW) edge[out=45, in=-45] (NW)
            (SE) edge[out=135, in=-135] (NE);
        \node[node, anchor=base] at (label) {$00$};
        \node[node] at (0,-10) {};
    \end{tikzpicture}
    \hfill
    \begin{tikzpicture}[scale=.15, baseline=(zero)]
        \coordinate (zero) at (0,0);
        \coordinate (NW) at (-5, 5);
        \coordinate (NE) at (5, 5);
        \coordinate (SE) at (5, -5);
        \coordinate (SW) at (-5, -5);
        \coordinate (label) at (0,-8);
        \draw[strand]
            (NW) edge[out=-45, in=135] (SE)
            (SW) edge[out=45, in=-135] (NE);
        \node[box] at (zero) {R};
        \node[node, anchor=base] at (label) {$R$};
    \end{tikzpicture}
    \begin{tikzpicture}[baseline=(zero)]
        \coordinate (zero) at (0,0);
        \node at (0,.2) {\footnotesize Twist};
        \draw[{Latex}-{Latex}, very thin] (.5,0) -- (-.5,0); 
    \end{tikzpicture}
    \begin{tikzpicture}[scale=.15, baseline=(zero)]
        \coordinate (zero) at (0,0);
        \coordinate (LTend) at (-15,5);
        \coordinate (LBend) at (-15,-5);
        \coordinate (A) at (-7,3);
        \coordinate (c1) at (-10.8,.8);
        \coordinate (c2) at (-9.2,-.8);
        \coordinate (B) at (-7,-3);
        \coordinate (C) at (7,-3);
        \coordinate (c3) at (9.2,-.8);
        \coordinate (c4) at (10.8,.8);
        \coordinate (D) at (7,3);
        \coordinate (RTend) at (15,5);
        \coordinate (RBend) at (15,-5);
        \node[mirrorbox] at (zero) (R) {\reflectbox{\rotatebox{180}{\textbf{R}}}};
        \draw[strand]
            (LTend) edge[out=-45, in=135] (c1)
            (c2) edge[out=-45, in=135] (B)
            (LBend) edge[out=45, in=-135] (A)
            (A) edge[out=45, in=135] (R)
            (B) edge[out=-45, in=-135] (R)
            (R) edge[out=45, in=135] (D)
            (R) edge[out=-45, in=-135] (C)
            (C) edge[out=45, in=-135] (c3)
            (c4) edge[out=45, in=-135] (RTend)
            (D) edge[out=-45, in=135] (RBend);
        \node[node, anchor=base] at (0,-8) {$1\!+\!R_x\!+\!\overline{1}$};
    \end{tikzpicture}

    \begin{tikzpicture}[scale=.15, baseline=(zero)]
        \coordinate (zero) at (0,0);
        \coordinate (label) at (0,-8);    
        \coordinate (NW) at (-5, 5);
        \coordinate (NE) at (5, 5);
        \coordinate (SE) at (5, -5);
        \coordinate (SW) at (-5, -5);
        \coordinate (c1) at (-3.5, 1.8);
        \coordinate (c2) at (-2, -.8);
        \coordinate (c3) at (2, -.8);
        \coordinate (c4) at (3.5, 1.8);
        \draw[strand, looseness=2.5]
            (SW) edge[out=75, in=105] (SE);
        \draw[strand, looseness=1]
            (NW) edge[out=-75, in=120] (c1)
            (c4) edge[out=60, in=-115] (NE);
        \draw[strand]
            (c2) edge[out=-50, in=-130] (c3);
        \node[node, anchor=base] at (label) {$1\!+\!\overline{1}$};
        \node[node] at (0,-10) {};
    \end{tikzpicture}
    \begin{tikzpicture}[baseline=(zero)]
        \coordinate (zero) at (0,0);
        \node at (0,.2) {\footnotesize II};
        \draw[{Latex}-{Latex}, very thin] (.5,0) -- (-.5,0); 
    \end{tikzpicture}
    \begin{tikzpicture}[scale=.15, baseline=(zero)]
        \coordinate (zero) at (0,0);
        \coordinate (label) at (0,-8);    
        \coordinate (NW) at (-5, 5);
        \coordinate (NE) at (5, 5);
        \coordinate (SE) at (5, -5);
        \coordinate (SW) at (-5, -5);
        \draw[strand]
            (SW) edge[out=45, in=135] (SE)
            (NW) edge[out=-45, in=-135] (NE);
        \node[node, anchor=base] at (label) {$0$};
    \end{tikzpicture}
    \hfill
    \begin{tikzpicture}[scale=.15, baseline=(zero)]
        \coordinate (zero) at (0,0);
        \coordinate (A1) at (-7,3);
        \coordinate (A2) at (-11.2,2.2);
        \coordinate (A3) at (-12.8,3.8);
        \coordinate (A4) at (-14,5);
        \coordinate (B1) at (-7,-3);
        \coordinate (B2) at (-11.2,-2.2);
        \coordinate (B3) at (-12.8,-3.8);
        \coordinate (B4) at (-14,-5);
        \coordinate (C1) at (-7.8,3.8);
        \coordinate (C2) at (-12,3);
        \coordinate (C3) at (-12,-3);
        \coordinate (C4) at (-7.8,-3.8);
        \coordinate (C5) at (-6.2,-2.2);
        \coordinate (C6) at (-6.2,2.2);
        \coordinate (D) at (5,5);
        \coordinate (E) at (5,-5);
        \node[box] at (zero) (R) {R};
        \draw[strand]
            (R) edge[out=135, in=45] (A1)
            (A1) edge[out=-135, in=-45] (A2)
            (A3) edge[out=135, in=-45] (A4)
            (R) edge[out=-135, in=-45] (B1)
            (B1) edge[out=135, in=45] (B2)
            (B3) edge[out=-135, in=45] (B4)
            (C1) edge[out=135, in=45] (C2)
            (C2) edge[out=-135, in=135] (C3)
            (C3) edge[out=-45, in=-135] (C4)
            (C5) edge[out=45, in=-45] (C6)
            (R) edge[out=45, in=-135] (D)
            (R) edge[out=-45, in=135] (E);
        \node[node, anchor=base] at (-6,-8) {$2(\overline{2}0)A$};
    \end{tikzpicture}
    \begin{tikzpicture}[baseline=(zero)]
        \coordinate (zero) at (0,0);
        \node at (0,.2) {\footnotesize Ring};
        \draw[{Latex}-{Latex}, very thin] (.5,0) -- (-.5,0); 
    \end{tikzpicture}
    \begin{tikzpicture}[scale=.15, baseline=(zero)]
        \coordinate (zero) at (0,0);
        \coordinate (A0) at (14,5);
        \coordinate (A1) at (12,3);
        \coordinate (A2) at (7.8,2.2);
        \coordinate (A3) at (6.2,3.8);
        \coordinate (B0) at (14,-5);
        \coordinate (B1) at (12,-3);
        \coordinate (B2) at (7.8,-2.2);
        \coordinate (B3) at (6.8,-3.8);
        \coordinate (C1) at (11.2,3.8);
        \coordinate (C2) at (7,3);
        \coordinate (C3) at (7,-3);
        \coordinate (C4) at (11.2,-3.8);
        \coordinate (C5) at (12.8,-2.2);
        \coordinate (C6) at (12.8,2.2);
        \coordinate (D) at (-5,5);
        \coordinate (E) at (-5,-5);
        \node[box] at (zero) (R) {R};
        \draw[strand]
            (A0) edge[out=-135, in=45] (A1)
            (A1) edge[out=-135, in=-45] (A2)
            (A3) edge[out=135, in=45] (R)
            (B0) edge[out=135, in=-45] (B1)
            (B1) edge[out=135, in=45] (B2)
            (B3) edge[out=-135, in=-45] (R)
            (C1) edge[out=135, in=45] (C2)
            (C2) edge[out=-135, in=135] (C3)
            (C3) edge[out=-45, in=-135] (C4)
            (C5) edge[out=45, in=-45] (C6)
            (R) edge[out=135, in=-45] (D)
            (R) edge[out=-135, in=45] (E);
        \node[node, anchor=base] at (4.5,-8) {$A0(2(\overline{2}0)0)$};
    \end{tikzpicture}

    \begin{tikzpicture}[scale=.15, baseline=(zero), remember picture]
        \coordinate (zero) at (0,0);
        \coordinate (cross1) at (10,0);
        \coordinate (cross2) at (5,-15);
        \coordinate (cross3) at (25,-5);
        \coordinate (LTend) at (-5,5);
        \coordinate (A1) at (3,-7);
        \coordinate (A2) at (-3,-13);
        \coordinate (A3) at (-3,-17);
        \coordinate (A4) at (4.2,-24.2);
        \coordinate (A5) at (5.8,-25.8);
        \coordinate (A6) at (9,-29);
        \coordinate (B1) at (7,3);
        \coordinate (B2) at (18,-8);
        \coordinate (B3) at (18,-12);
        \coordinate (B4) at (1,-29);
        \coordinate (C1) at (-3,-7);
        \coordinate (C2) at (-.8,-9.2);
        \coordinate (C3) at (.8,-10.8);
        \coordinate (C4) at (3,-13);
        \coordinate (C5) at (7,-13);
        \coordinate (C6) at (14.2,-5.8);
        \coordinate (C7) at (15.8,-4.2);
        \coordinate (C8) at (19,-1);
        \node[mirrorbox] at (zero) (R) {\reflectbox{\rotatebox{-90}{\textbf{R}}}};
        \draw[strand]
            (LTend) edge[out=-45, in=135] (R)
            (R) edge[out=45, in=135] (B1)
            (B1) edge[out=-45, in=135] (B2)
            (B2) edge[out=-45, in=45] (B3)
            (B3) edge[out=-135, in=45] (B4)
            (R) edge[out=-45, in=45] (A1)
            (A1) edge[out=-135, in=45] (A2)
            (A2) edge[out=-135, in=135] (A3)
            (A3) edge[out=-45, in=135] (A4)
            (A5) edge[out=-45, in=135] (A6)
            (R) edge[out=-135, in=135] (C1)
            (C1) edge[out=-45, in=135] (C2)
            (C3) edge[out=-45, in=135] (C4)
            (C4) edge[out=-45, in=-135] (C5)
            (C5) edge[out=45, in=-135] (C6)
            (C7) edge[out=45, in=-135] (C8);
        \node[node, anchor=base] at (15,-24) {$\overline{R}01\overline{1}10$};
        \node at (19,-6) (anchor1) {};
    \end{tikzpicture}
    \begin{tikzpicture}[baseline=(zero)]
        \coordinate (zero) at (0,0);
        \node at (0,.2) {\footnotesize Flype};
        \draw[{Latex}-{Latex}, very thin] (.5,0) -- (-.5,0); 
    \end{tikzpicture}
    \begin{tikzpicture}[scale=.15, baseline=(zero), remember picture]
        \coordinate (zero) at (0,0);
        \coordinate (NW) at (-5, 5);
        \coordinate (NE) at (5, 5);
        \coordinate (SE) at (5, -5);
        \coordinate (SW) at (-5, -5);
        \coordinate (label) at (0,-8);
        \draw[strand]
            (NW) edge[out=-45, in=135] (SE)
            (SW) edge[out=45, in=-135] (NE);
        \node[box] at (zero) {R};
        \node[node, anchor=base] at (label) {$R$};
        \node[node, anchor=base] at (0,-14) (anchorlabel) {};
    \end{tikzpicture}
    \begin{tikzpicture}[baseline=(zero)]
        \coordinate (zero) at (0,0);
        \node at (0,.2) {\footnotesize Flype};
        \draw[{Latex}-{Latex}, very thin] (.5,0) -- (-.5,0); 
    \end{tikzpicture}
    \begin{tikzpicture}[scale=.15, baseline=(zero), remember picture]
        \coordinate (zero) at (0,0);
        \coordinate (cross1) at (10,0);
        \coordinate (cross2) at (5,-15);
        \coordinate (cross3) at (25,-5);
        \coordinate (LTend) at (-5,5);
        \coordinate (A1) at (7,-3);
        \coordinate (A2) at (9.2,-.8);
        \coordinate (A3) at (10.8,.8);
        \coordinate (A4) at (13,3);
        \coordinate (A5) at (17,3);
        \coordinate (A6) at (29,-9);
        \coordinate (B1) at (7,3);
        \coordinate (B2) at (13,-3);
        \coordinate (B3) at (13,-7);
        \coordinate (B4) at (1,-19);
        \coordinate (C1) at (-3,-7);
        \coordinate (C2) at (4.2,-14.2);
        \coordinate (C3) at (5.8,-15.8);
        \coordinate (C4) at (12,-18);
        \coordinate (C5) at (24.2,-5.8);
        \coordinate (C6) at (25.8,-4.2);
        \coordinate (C7) at (29,-1);
        \node[mirrorbox] at (zero) (R) {\reflectbox{\rotatebox{-90}{\textbf{R}}}};
        \draw[strand]
            (LTend) edge[out=-45, in=135] (R)
            (R) edge[out=45, in=135] (B1)
            (B1) edge[out=-45, in=135] (B2)
            (B2) edge[out=-45, in=45] (B3)
            (B3) edge[out=-135, in=45] (B4)
            (R) edge[out=-45, in=-135] (A1)
            (A1) edge[out=45, in=-135] (A2)
            (A3) edge[out=45, in=-135] (A4)
            (A4) edge[out=45, in=135] (A5)
            (A5) edge[out=-45, in=135] (A6)
            (R) edge[out=-135, in=135] (C1)
            (C1) edge[out=-45, in=135] (C2)
            (C3) edge[out=-45, in=-135] (C4)
            (C4) edge[out=45, in=-135] (C5)
            (C6) edge[out=45, in=-135] (C7)
            ;
        \node[node, anchor=base] at (-1,-24) {$\overline{R1}1\overline{1}$};
        \node at (-5,-6) (anchor2) {};
    \end{tikzpicture}
    \begin{tikzpicture}[remember picture, overlay]
        \draw[{Latex}-{Latex}] (anchor1) to [out=-45,in=-135] (anchor2);
        \node at (anchorlabel) {\footnotesize Reidemeister III};
    \end{tikzpicture}

    \caption{Tangle moves: elementary moves I and II (special cases of Reidemeister I and II moves), twist, ring and flype moves. Two presented flypes are related by Reidemeister III move.}
    \label{fig:twist_flype}
\end{figure}

\begin{definition}[Twist]\label{def:twist_move}
    The \emph{twist move} is a isotopy move obtained by rotating a subtangle around its $x$-axis by $\pi$. Algebraically, the twist can be formulated as:
    \begin{equation*}
        \mathmakebox[5cm][r]{1+A_x\!+\!\overline{1} \twist A \twist \overline{1}\!+\!A_x\!+\!1}
        \quad \iff \quad
        \mathmakebox[5cm][l]{10A_x0\overline{1} \twist  A \twist \overline{1}0A_x01.}
    \end{equation*}
\end{definition}

\noindent In practice, we will use twist move to move integral tangles from one side of the tangle to the other:
\begin{equation*}
    \mathmakebox[5cm][r]{n\!+\!A \twist A_{x^n}\!+\!n}
    \quad \iff \quad
    \mathmakebox[5cm][l]{n0A \twist A_{x^n}0n,}
\end{equation*}
where $R_{x^n}$ denotes the tangle $R$ to be rotated around the $x$-axis $n$ times, i.e. $R_{x^n} = R_x$ if $n$ is odd and $R_{x^n} = R$ if $n$ is even.

\begin{definition}[Flype]\label{def:flype_move}
    A \emph{flype move} is a tangle isotopy move obtained by the rotation of a subtangle around its principal diagonal $p$ by $\pi$. There are 4 formulas for flypes, two for each direction of rotations (right- and left-handed). In both of them, the SE arc can arrange itself in two different ways related by Reidemeister III move: 
    $$ A \flype \overline{A1}1\overline{1} \reidthree \overline{A}01\overline{1}10 
    \flype A \flype 
    \overline{A01}1\overline{1}0 \reidthree \overline{A}1\overline{1}1 = A.$$
\end{definition}
\noindent Now, using flype, we can show that $1=10$ (with a consequence that $1n=n\!+\!1$):
\begin{align}\label{eq:ambiguity_1}
    1 \flype \overline{1}01\overline{1}10 \reidtwo 0\overline{1}10 \reidone 0010 = 10
\end{align}

\begin{definition}[Ring move]\label{def:ring_move} 
    The \emph{ring move} (\autoref{fig:twist_flype}) is a tangle isotopy move obtained by pushing the ring over a neighboring subtangle: 
    \begin{equation*}
        {(20+\overline{2}0)0+A \ring A+(20\!+\!\overline{2}0)0}
        \quad \iff \quad
        {2(\overline{2}0)A  \ring A0(2(\overline{2}0)0)}.
    \end{equation*}
\end{definition}
\noindent For the tangles $2(20)\overline{1}A$ (and $\overline{2(20)}1A$) the a possibility of performing the ring move is hidden, but combination of flypes and twists reveals it:
\begin{equation}
    \overline{2(20)}1A \flype 2(2\overline{1})10A \flype 2(\overline{2}0)10A \twist 2(\overline{2}0)A_x01 \ring A_x0(2(\overline{2}0)0)01.
\end{equation}

\begin{definition}[Prime tangle]\label{def:prime_tangle}
    A tangle  $T$  in the ball $B^3$ is called prime (or non-composite) if there does not exist a 2-sphere $S^2 \subset B^3$  that intersects  $T$  transversely in exactly two points, such that the part of $T$ inside of the $S^2$ sphere is not trivial, i.e. if we close the inner part of $T$ by a geodesic on $S^2$, we do not obtain the unknot, see examples in the \autoref{fig:nonprime_tangle}.
\end{definition}
\begin{figure}
    \centering
    \includegraphics[scale=0.5, page=4]{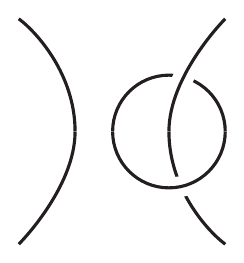} \qquad \qquad
    \includegraphics[scale=0.5, page=3]{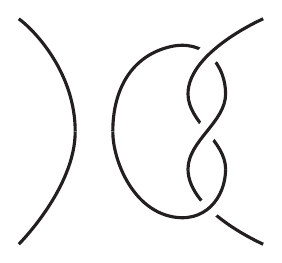}
    \caption{Examples of non-prime tangles $0(20)$ and $0(30)$ (respectively).}
    \label{fig:nonprime_tangle}
\end{figure}

A subtangle inside a 2-sphere (from definition above) can be moved anywhere along the the arc. Every tangle of the form $0A$, where $A$ is a tangle which cannot be reduced into $0$, is a composite tangle. Essentially, $0A$ is tangle $A$ with connected NW and SW ends with an additional open arc. A simple example of such non-prime tangle is a $0(2n)$ tangle, where $n$ is any integral tangle: 
\begin{align}
    0(2n) = 0(20)0n \twist 0n0(20) \reidone 0(20)
\end{align}

In the case of rational tangles, it was shown in \cite{kauffman2004classification} that the twist, flype, and IInd elementary move generate isotopy. In the case of prime algebraic tangles, we state the following conjecture:
\begin{conjecture}\label{conjecture}
Two prime algebraic tangles are isotopic if and only if they are related by a finite sequence of twist, flype, ring, and elementary moves. 
\end{conjecture}

\section{An algebraic tangle as a tree}\label{sec:tree}

\subsection{Tree notation}
Due to the fact that the multiplication is left-associative, 
left factors and right factors behave differently.
This clearly affects the behavior of tangles; however, for more complex tangles, the abundance of brackets makes the algebraic representation difficult to read.
In such cases, the expression tree representation shows its advantage. For example, every algebraic tangle $A$ can be represented as a product of two subtangles, $A = LR$, where both subtangles can similarly be decomposed into a product of sub-subtangles. This process can be repeated indefinitely (in fact, ad infinitum, since $1 = 10$):
\begin{equation}\label{eq:tree_decomposition_full}
    A = LR = L_LR_L(L_RR_R) = L_{L_L}R_{L_L}(L_{R_L}R_{R_L})\bigl(L_{L_R}R_{L_R}(L_{R_R}R_{R_R})\bigr) = \cdots
\end{equation}
For clarity, one can represent such algebraic tangle as a binary expression tree, see \autoref{tree:expression_tree}.
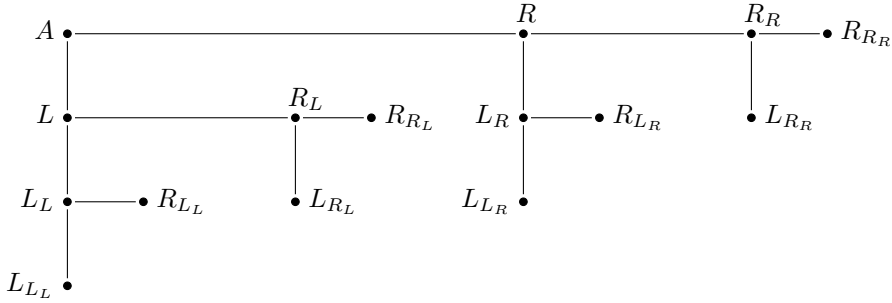
\begin{figure}[ht!]
    \centering
    \begin{tikzpicture}
        \node[tree={T}, row sep={1cm}]
        {
            \node (T) {};
            &&&&&& \node (R) {};
            &&& \node (RR) {};
            & \node (RRR) {};\\
            \node (L) {};
            &&& \node (LR) {};
            & \node (LRR) {};
            && \node (RL) {};
            & \node (RLR) {};
            && \node (RRL) {};\\
            \node (LL) {};
            & \node (LLR) {};
            && \node (LRL) {};
            &&& \node (RLL) {};\\
            \node (LLL) {};\\ 
        };
        \node[labelleft={T}] {$A$};
        \node[labelleft={L}] {$L$};
        \node[labelleft={LL}] {$L_L$};
        \node[labelleft={LLL}] {$L_{L_L}$};
        \node[labelabove={R}] {$R$};
        \node[labelabove={RR}] {$R_R$};
        \node[labelright={RRR}] {$R_{R_R}$};
        \node[labelleft={RL}] {$L_R$};
        \node[labelleft={RLL}] {$L_{L_R}$};
        \node[labelright={LLR}] {$R_{L_L}$};
        \node[labelright={LRL}] {$L_{R_L}$};
        \node[labelabove={LR}] {$R_L$};
        \node[labelright={LRR}] {$R_{R_L}$};
        \node[labelright={RLR}] {$R_{L_R}$};
        \node[labelright={RRL}] {$L_{R_R}$};
        \draw 
            (T) edge (R)
            (T) edge (L)
            (R) edge (RR)
            (R) edge (RL)
            (RR) edge (RRR)
            (RR) edge (RRL)
            (RL) edge (RLR)
            (RL) edge (RLL)
            (L) edge (LR)
            (L) edge (LL)
            (LR) edge (LRR)
            (LR) edge (LRL)
            (LL) edge (LLR)
            (LL) edge (LLL);
    \end{tikzpicture}
    \caption{Tangle $A=L_{L_L}R_{L_L}(L_{R_L}R_{R_L})\bigl(L_{L_R}R_{L_R}(L_{R_R}R_{R_R})\bigr)$ represented as an expression tree. Every node represents a tangle which is a product of its left-child by its right-child.}
    \label{tree:expression_tree}
\end{figure}

In such an expression tree, every node represents some subtangle of tangle $A$ and a product of its children.
A node with no children, resulting from stopping the division of the tree at some point, is called a \emph{leaf}.

If we expand only left children at each level, our representation stays clearer and we can write the tangle concisely in the multiplication representation:
\begin{equation}\label{eq:tree_decomposition_mul}
    A = LR = L_LR_LR = L_{L_L}R_{L_L}R_LR 
    = L_{L_{L_{L_{\resizebox{.5em}{!}{$\ddots$}}}}}\!\!\!\!
    R_{L_{L_{L_{\resizebox{.5em}{!}{$\ddots$}}}}} \!\!\!\!\cdots R_LR 
    := L_N\prod_{i=N}^1 R_i.
\end{equation}
Note that if every $R_i$ term represents an integral tangle, and $L_N$ represent a rational tangle, then the whole tangle is rational.

An analogous situation occurs when we expand only the right children at each level, allowing us to concisely represent the tangle using an additive representation:
\begin{align}\label{eq:tree_decomposition_sum}
    \begin{matrix*}[l]
    A &= LR = L(L_RR_R) = L\bigl(L_R(L_{R_R}R_{R_R})\bigr) = 
    L\biggl(L_R\Bigl(\cdots\bigl(L_{R_{R_{R_{\resizebox{.5em}{!}{$\ddots$}}}}}\!\!\!\!
    R_{R_{R_{R_{\resizebox{.5em}{!}{$\ddots$}}}}}\!\!\!\!\bigr)\cdots\Bigr)\biggr) =\\
    & =L0+\Bigl(L_R0+\Bigl(\cdots+\bigl(L_{R_{R_{R_{\resizebox{.5em}{!}{$\ddots$}}}}}\!\!\!\!0\!+\!
    R_{R_{R_{R_{\resizebox{.5em}{!}{$\ddots$}}}}}\!\!\!\!\bigr)\cdots\Bigr)\Bigr) =\\
    &=L0+L_R0+\cdots+L_{R_{R_{R_{\resizebox{.5em}{!}{$\ddots$}}}}}\!\!\!\!0\!+\!
    R_{R_{R_{R_{\resizebox{.5em}{!}{$\ddots$}}}}}\!\!\!\! := \left(\displaystyle\sum_{i=1}^N L_i0\right) + R_N.
    \end{matrix*}
\end{align}
In \cite{conway1970enumeration} Conway used special comma notation for such tangles if $R_N = n \in \integral$: $(L_1, L_2, \cdots, L_N, +++/---)$, where the number of +/- symbols is equal to the value of $n$ (\qq{$+$} for positive, \qq{$-$} for negative integrals).

Sometimes we want to represent only a part of an expression tree. In this case it may be important if the top of the subtree is a left-child or a right-child. It appears that tree-tops behave exactly as right-children. Therefore, we introduce some extra notation for edges presented in \autoref{fig:extra_tree}.
\begin{figure}[h!]
    \centering
    \begin{tikzpicture}
        \node[tree={mid}, row sep={.45cm}] at (0,0) (tree1)
        {
            \node[phantom] (T1) {}; \\
            \node[phantom] (mid) {}; \\
            \node (L1) {};\\
        };
        \node[tree={T2}] at (1,0) (tree2)
        {
            \node[phantom] (T2) {}; & \node (R2) {}; \\
        };
        \node[tree={T3}] at (3,0) (tree3)
        {
            \node[phantom] (T3) {}; & \node (R3) {}; \\
        };
        \node[tree={T4}] at (5.5,0) (tree4)
        {
            \node (T4) {};\\
        };
        \draw (T1) edge (L1);
        \draw (T2) edge (R2);
        \draw (T3) edge[dottededge] (R3);
        \node[labelright={L1}] {$A$};
        \node[labelright={R2}] {$B$};
        \node[labelright={R3}] {$C$};
        \node[labelright={T4}] {$D$};
    \end{tikzpicture}
    \caption{Subtangle $A$ is a left-child, subtangle $B$ is a right-child, subtangle $C$ is either a tree-top or a right-child, and subtangle $D$ is either a child or not.}\label{fig:extra_tree}
\end{figure}
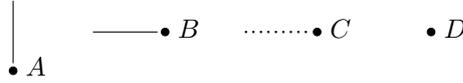

\begin{definition}[Power notation]\label{def:power_notation}
    We will use power notation $A^B$ to emphasize that $B$ is a subtangle of the tangle $A$. The subtangle $B$ may be equal to $A$, a right-child of $A$, or a right-child of the right-child of $A$, and so on.
    \begin{figure}[h!]
        \centering
        \begin{tikzpicture}[baseline=(T.base)]
            \node[smallnode] (T) {};
            \node[labelright={T}] {$A^B$};
        \end{tikzpicture}
        \begin{tikzpicture}[baseline=(A.base)]
            \node (A) {$\!=$};
        \end{tikzpicture}
        \begin{tikzpicture}[baseline=(T.base)]
            \node[tree={R}] (tree1)
            {
                \node (T) {}; & \node (R) {};\\
            };
            \node[labelleft={T}] {$A$};
            \node[labelright={R}] (B) {$B$};
            \draw (T) edge[dottededge] (R);
        \end{tikzpicture}
        \begin{tikzpicture}[baseline=(A.base)]
            \node (A) {$\in \Bigg \{$};
        \end{tikzpicture}
        \begin{tikzpicture}[baseline=(T.base)]
            \node[smallnode] (T) {};
            \node[labelright={T}] {$B$};
        \end{tikzpicture}
        \begin{tikzpicture}[baseline=(A.base)]
            \node (A) {$, \;$};
        \end{tikzpicture}
        \begin{tikzpicture}[baseline=(T.base)]
            \node[tree={T}] at (tree3)
            {
                \node (T) {}; & \node (R) {};\\
            };
            \node[labelleft={T}] {$A$};
            \node[labelright={R}] (B) {$B$};
            \draw (T) edge (R);
        \end{tikzpicture}
        \begin{tikzpicture}[baseline=(A.base)]
            \node (A) {$, \;$};
        \end{tikzpicture}
        \begin{tikzpicture}[baseline=(T.base)]
            \node[tree={T}] at (tree3)
            {
                \node (T) {}; & \node (R) {}; & \node (RR) {};\\
            };
            \node[labelleft={T}] {$A$};
            \node[labelright={RR}] (B) {$B$};
            \draw (T) edge (R)
                  (R) edge (RR);
        \end{tikzpicture}
        \begin{tikzpicture}[baseline=(A.base)]
            \node (A) {$,\; \ldots \Bigg \}$};
        \end{tikzpicture}
        \caption{Tree representation of power notation $A^B$. We use dotted-edge notation, to skip part of the tree. Note that $A$ may be equal to $B$, then $A^B=B$.}
    \end{figure}
\end{definition}
Note that $A$ may be equal to $B$, then $A^B=B$. Moreover, when $A$ is an integral tangle $A=n$, then $A^B=n^B=B$.

\subsection{Right leaves -- integral tangles}
{
Integral tangles (\subfigref{fig:rational_tangles}{a}) appear to be ideal candidates for the right leaves of algebraic tangles, since it is possible to use a twist move to change their position. Moreover, they have high symmetry, since rotations around $x,y,z$ axes do not affect them: $n = n_x = n_y = n_z$.} Indeed,
\begin{align}\label{eq:int_symmetry}
    \mathmakebox[5.5cm][l]{n_x = \left[\sum_{i=1}^n 1 \right]_x \!= \sum_{i=1}^n 1_x = \sum_{i=1}^n 1 = n;}\qquad
    \mathmakebox[5.5cm][l]{n_y = \left[\sum_{i=1}^n 1 \right]_y \!= \sum_{i=n}^1 1_y = \sum_{i=n}^1 1 = n.}
\end{align}

{
It is convenient to represent each tangle $A$ using power notation $A^n$} to emphasize its rightmost integral tangle $n$. Let us show some properties of $A^n$ notation.
\begin{lemma}\label{thm:rightmost_algebraic}
    An algebraic tangle $A^n$ can be decomposed as $A^n=A^00n$.
\end{lemma}
\begin{proof}
    If $A^n=n$, then the statement is trivial. Otherwise, let us expand the $A^n$ as in \autoref{eq:tree_decomposition_sum} and rearrange the brackets (due to associativity of addition):
    \begin{align*}
        \begin{array}{rl}
             A^n
             &=
             a_1\Bigl(a_2\bigl(\cdots(a_Nn)\cdots\bigr)\Bigr) =
             \left(\displaystyle\sum_{i=1}^N a_i0\right)\!+\!n = a_10\!+\!\Bigl(a_20\!+\!\bigl(\cdots\!+\!(a_N0)\cdots\bigr)\Bigr)\!+\!n \\
             &= a_1\Big(a_2\bigl(\cdots(a_N0)\cdots\bigr)\Bigr)0n = A^00n.
        \end{array}
    \end{align*}
\end{proof}

\begin{lemma}\label{thm:rotation_algebraic}
    A rightmost integral tangle $n$ of tangle $A^n$ can stay rightmost during rotations:
    \begin{align*}
        [A^n]_x &= [A_x]^n = A_x^n\\
        [A^n]_y &=
        \begin{cases}
            [A_y]^n\textrm{, if $n$ is even,}\\
            [A_z]^n\textrm{, if $n$ is odd,}\\
        \end{cases} \\
        [A^n]_z &=
        \begin{cases}
            [A_x]^n\textrm{, if $n$ is even,}\\
            [A_z]^n\textrm{, if $n$ is odd.}\\
        \end{cases}
    \end{align*}
\end{lemma}
\begin{proof}
    We rotate the tangle $A$ (\autoref{eq:rotations_sum},\ref{eq:rotations_multiplication}) after expanding its right children (\autoref{eq:tree_decomposition_sum}):
    \begin{align*}\label{eq:rotation_simple_algebraic}
         [A^n]_x &= \left[\left(\sum_{i=1}^N a_i0\right) + n\right]_x = \left(\sum_{i=1}^N [a_i]_y0\right) + n = [A_x]^n = A_x^n,\\
         [A^n]_y &= \left[\left(\sum_{i=1}^N a_i0\right) + n\right]_y 
         = n + \sum_{i=N}^1 [a_i]_x0 \twist\\
         &\twist
         \begin{cases}
             \left(\sum_{i=N}^1 [a_i]_x0\right) + n\textrm{, if $n$ is even,}\\
             \left(\sum_{i=N}^1 [a_i]_z0\right) + n\textrm{, if $n$ is odd,}\\
         \end{cases}
         =
         \begin{cases}
            [A_y]^n\textrm{, if $n$ is even,}\\
            [A_z]^n\textrm{, if $n$ is odd,}\\
        \end{cases} \\
         [A^n]_z &= \left[\left(\sum_{i=1}^N a_i0\right) + n\right]_z 
         = n + \sum_{i=N}^1 [a_i]_z0 \twist \\
         &\twist
         \begin{cases}
             \left(\sum_{i=N}^1 [a_i]_z0\right) + n\textrm{, if $n$ is even,}\\
             \left(\sum_{i=N}^1 [a_i]_x0\right) + n\textrm{, if $n$ is odd.}\\
         \end{cases}
         =
         \begin{cases}
            [A_x]^n\textrm{, if $n$ is even,}\\
            [A_z]^n\textrm{, if $n$ is odd.}\\
        \end{cases}
    \end{align*}
    As we see, rotation around $y$ or $z$ axis reverses order of left children, but the rightmost integral $n$ can stay rightmost, at the expense of rotating left children around $y$ axis if $n$ is odd.
\end{proof}

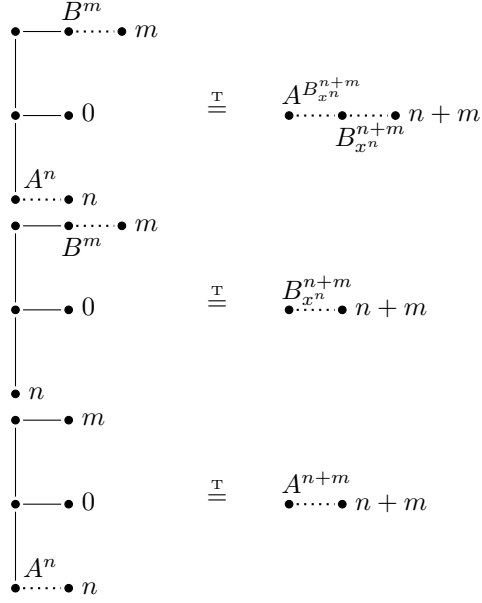
\begin{figure}
    \centering
    \begin{tabular}{lll}
        \begin{tikzpicture}[baseline=(L1.base)]
            \node[tree_mid={L1}, row sep={1cm}] (tree1)
            {
                \node (T1) {}; & \node (R1) {}; & \node (RR1) {};\\
                \node (L1) {};  & \node (LR1) {};\\
                \node (LL1) {}; & \node (LLR1) {};\\
            };
            \node[labelabove={R1}] {$B^m$};
            \node[labelright={RR1}] {$m$};
            \node[labelright={LR1}] {$0$};
            \node[labelaboveright={LL1}] {$A^n$};
            \node[labelright={LLR1}] {$n$};
            \draw 
                (T1) edge (R1)
                (T1) edge (L1)
                (L1) edge (LR1)
                (L1) edge (LL1)
                (R1) edge[dottededge] (RR1)
                (LL1) edge[dottededge] (LLR1);
        \end{tikzpicture}
        &
        \begin{tikzpicture}[baseline=(aa.base)]
            \node (aa) {$\twist$};
        \end{tikzpicture}
        &
        \begin{tikzpicture}[baseline=(tree2.base)]
            \node[tree_mid={T2}] (tree2)
            {
                \node (T2) {}; & \node (R2) {}; & \node (RR2) {};\\
            };
            \node[labelabove={T2}] {$A^{B_{x^n}^{n+m}}$};
            \node[labelbelow={R2}] {$B_{x^n}^{n+m}$};
            \node[labelright={RR2}] {$n+m$};
            \draw 
                (T2) edge[dottededge] (R2)
                (R2) edge[dottededge] (RR2);
        \end{tikzpicture}\\
        \begin{tikzpicture}[baseline=(L1.base)]
            \node[tree_mid={L1}, row sep={1cm}] (tree1)
            {
                \node (T1) {}; & \node (R1) {}; & \node (RR1) {};\\
                \node (L1) {};  & \node (LR1) {};\\
                \node (LL1) {}; \\
            };
            \node[labelbelow={R1}] {$B^m$};
            \node[labelright={RR1}] {$m$};
            \node[labelright={LR1}] {$0$};
            \node[labelright={LL1}] {$n$};
            \draw 
                (T1) edge (R1)
                (T1) edge (L1)
                (L1) edge (LR1)
                (L1) edge (LL1)
                (R1) edge[dottededge] (RR1);
        \end{tikzpicture}
        &
        \begin{tikzpicture}[baseline=(aa.base)]
            \node (aa) {$\twist$};
        \end{tikzpicture}
        &
        \begin{tikzpicture}[baseline=(tree2.base)]
            \node[tree_mid={T2}] (tree2)
            {
                \node (T2) {}; & \node (R2) {};\\
            };
            \node[labelabove={T2}] {$B_{x^n}^{n+m}$};
            \node[labelright={R2}] {$n+m$};
            \draw 
                (T2) edge[dottededge] (R2);
        \end{tikzpicture}\\
        \begin{tikzpicture}[baseline=(L1.base)]
            \node[tree_mid={L1}, row sep={1cm}] (tree1)
            {
                \node (T1) {}; & \node (R1) {}; & \node[phantom] (RR1) {};\\
                \node (L1) {};  & \node (LR1) {};\\
                \node (LL1) {}; & \node (LLR1) {};\\
            };
            \node[labelright={R1}] {$m$};
            \node[labelright={LR1}] {$0$};
            \node[labelaboveright={LL1}] {$A^n$};
            \node[labelright={LLR1}] {$n$};
            \draw 
                (T1) edge (R1)
                (T1) edge (L1)
                (L1) edge (LR1)
                (L1) edge (LL1)
                (LL1) edge[dottededge] (LLR1);
        \end{tikzpicture}
        &
        \begin{tikzpicture}[baseline=(aa.base)]
            \node (aa) {$\twist$};
        \end{tikzpicture}
        &
        \begin{tikzpicture}[baseline=(tree2.base)]
            \node[tree_mid={T2}] (tree2)
            {
                \node (T2) {}; & \node (R2) {}; & \node[phantom] (RR2) {};\\
            };
            \node[labelabove={T2}] {$A^{n+m}$};
            \node[labelright={R2}] {$n+m$};
            \draw 
                (T2) edge[dottededge] (R2);
        \end{tikzpicture}\\
    \end{tabular}
    \caption{Isotopy preserving move $A^n\!+\!B^m \twist A^{B_{x^n}^{n+m}}$ represented on a tree. Note that if $A^n=n$ or $B^m=m$ the result simplifies to $B_{x^n}^{n+m}$ or $A^{n+m}$ respectively.}\label{fig:power_add}
\end{figure}

\begin{lemma}\label{thm:generalized_twist}
    If two tangles $A^n$ and $B^m$ are added together, $B_{x^n}^{n+m}$ becomes the right subtangle of $A$ (\autoref{fig:power_add}):
    \begin{equation*}
        A^n\!+\!B^m \twist A^{B_{x^n}^{n+m}}.
    \end{equation*}
\end{lemma}
\begin{proof}
    \begin{align*}
        A^n\!+\!B^m &= A^0\!+\!n+\!B^0\!+\!m \twist A^0+\!B_{x^n}^0\!+\!n\!+\!m = A^0+\!B_{x^n}^{n+m} =\\ 
        &= \left(\sum_{i=1}^N a_i0\right) + \left[\sum_{i=1}^M b_i(n+m)\right]_{x^n}\!\!\!\!  = A^{B_{x^n}^{n+m}}.
    \end{align*}
\end{proof}

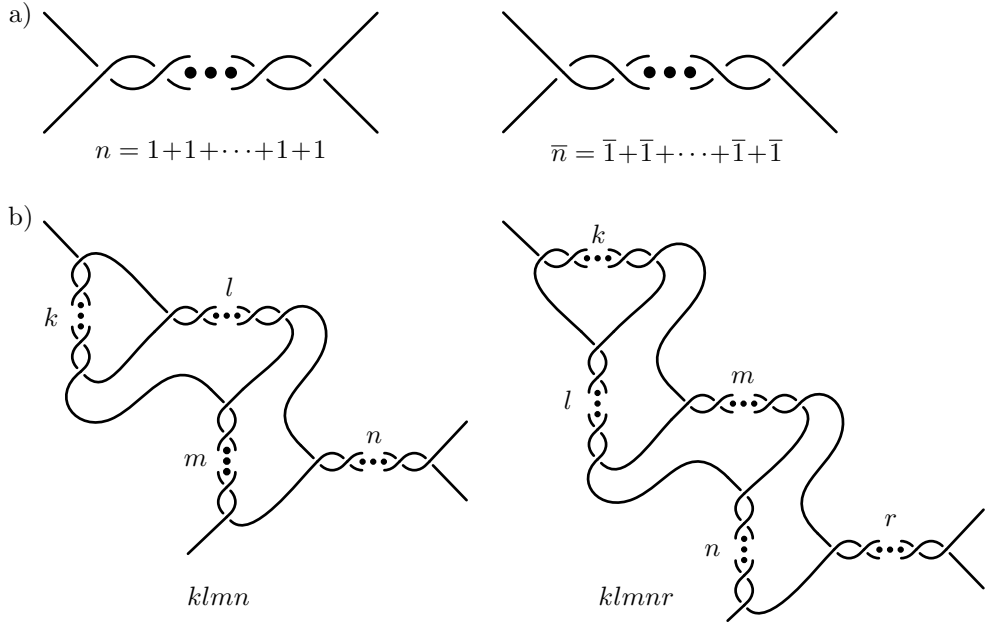
\begin{figure}[t]
    \centering
    \begin{tikzpicture}
        \node at (0,0) {\includegraphics[width=.95\textwidth]{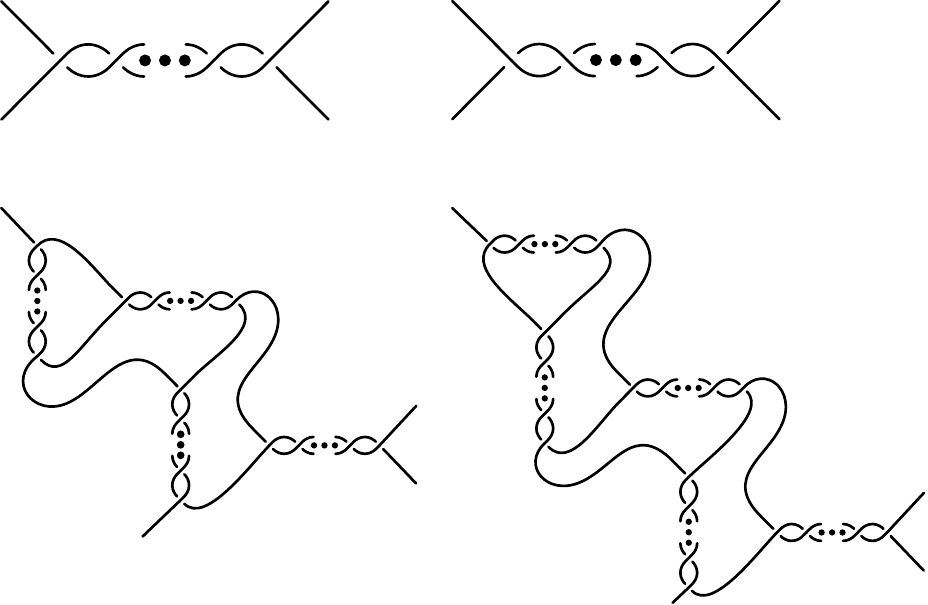}};
        \node at (-6.5,4) {a)};
        \node at (-6.5,1.3) {b)};
        \node at (-4,2.2) {$n=1\!+\!1\!+\!\cdots\!+\!1\!+\!1$};
        \node at (2,2.2) {$\overline{n}=\overline{1}\!+\!\overline{1}\!+\!\cdots\!+\!\overline{1}\!+\!\overline{1}$};
        \node at (-3.9,-3.7) {$klmn$};
        \node at (1.6,-3.7) {$klmnr$};
        \node at (-6.1,0) {$k$};
        \node at (-3.75,.4) {$l$};
        \node at (-4.2,-1.9) {$m$};
        \node at (-1.85,-1.6) {$n$};
        \node at (1.1,1.1) {$k$};
        \node at (.7,-1.1) {$l$};
        \node at (3,-.8) {$m$};
        \node at (2.6,-3.1) {$n$};
        \node at (4.95,-2.7) {$r$};
    \end{tikzpicture}
    \caption{
        a) Integral tangles $n$ as a sum of $1$ tangles, and $\overline{n}$ as sum of $\overline{1}$ tangles. 
        v) Rational tangles $klmn$ and $klmnr$ in a standard form: as a multiplication of integral tangles $k$, $l$, $m$, $n$, and $r$.}
    \label{fig:rational_tangles}
\end{figure}

\subsection{Left leaves -- rational tangles}
Now we show that rational tangles are ideal candidates for left leaves. Let us first recall some known results for rational tangles, which have been thoroughly studied in \cite{conway1970enumeration,burde1986knots,kauffman2004classification}.

\begin{definition}[Rational tangle]\label{def:tangles_rational}
    A \emph{rational tangle} (\subfigref{fig:rational_tangles}{b}) is a tangle $a \in \rational$ that can be expressed as a multiplication of integral tangles: $kl\cdots mn$. If a rational tangle is written as such a product, we say that the rational tangle is in a \emph{standard form}.
\end{definition}
Like integral tangles, rational tangles are invariant to rotations around $x,y,z$ axes, which can be shown by recurrent treatment:
\begin{align}
    [kl\cdots mn]_x &= [kl\cdots m]_yn \\
    [kl\cdots mn]_y &= n0[kl\cdots m0]_y \twist [kl\cdots m0]_{yx^n}0n = [kl\cdots m]_{xy^n}n.
\end{align}

{
More importantly, there exists a perfect invariant of rational tangles -- the continued fraction. Unlike many other knot invariants, the continued fraction uniquely determines a rational tangle, providing a simple method to distinguish them.}

\begin{definition}\label{def:fraction_rational}
The fraction $p/q$, where $p,q \in \mathbb{Z}$  (including $1/0=\infty$), of a rational tangle in a standard form $kl\cdots st$, is defined as the continued fraction:
    \begin{equation*}
     \mathrm{Frac}(kl\cdots mn) = n\!+\!\frac{1}{m+ \frac{1}{\ddots_{\;l + \frac{1}{k}} }} = p/q.
    \end{equation*}
\end{definition}

\begin{theorem}[Conway \cite{conway1970enumeration}]\label{thm:fraction}
    The fraction of a rational tangle is a perfect invariant -- two rational tangles are isotopic if and only if they have the same fraction.
\end{theorem}
The proofs can be found in \cite{montesinos1984revetements}, \cite{burde1986knots} p.196 and \cite{goldman1997rational}.

Note that all properties of rational tangles also apply to integral tangles, since integral tangles are a subset of rational tangles with an integer fraction.

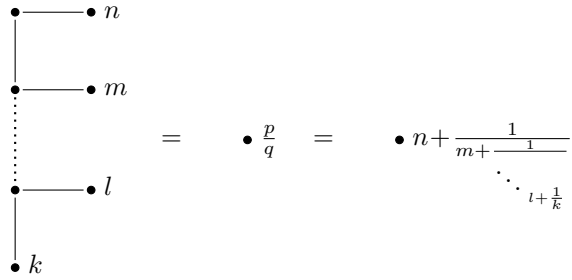
\begin{figure}
    \centering
    \begin{tikzpicture}
        \node[tree={mid}, row sep={.45cm}] (tree2) at (0,0)
            {
                \node (T2) {}; & \node (R2) {};\\
                \node[phantom] {}; \\
                \node (L2) {}; & \node (LR2) {};\\
                \node[phantom] (mid) {}; && \node[phantom] {$=$}; & \node (pq) {}; & \node[phantom] {$=$}; & \node (frac) {};\\
                \node (LL2) {}; & \node (LLR2) {};\\
                \node[phantom] {}; \\
                \node (LLL2) {}; \\
            };
            \node[labelright={R2}] (W2) {$n$};
            \node[labelright={LR2}] {$m$};
            \node[labelright={LLR2}] {$l$};
            \node[labelright={LLL2}] {$k$};
            \node[labelright={pq}] {$\frac{p}{q}$};
            \node[labelright={frac}] {$n\!+\!\frac{1}{m+ \frac{1}{\ddots_{\;l + \frac{1}{k}} }}$};
            \draw 
                (T2) edge (R2)
                (T2) edge (L2)
                (L2) edge[dottededge] (LL2)
                (L2) edge (LR2)
                (LL2) edge (LLL2)
                (LL2) edge (LLR2);
    \end{tikzpicture}
    \caption{Contraction of a rational tangle by replacing it with fraction of the rational tangle, represented by a move on a tree. On this figure dotted edge means 1 or more edges separated by nodes which right-children are integral tangles.}
    \label{fig:rational_contraction}
\end{figure}   

\begin{definition}[Canonical form]\label{def:canonical_form}
    Rational tangle is in \emph{canonical form} when all of the following conditions hold:
    \begin{itemize}
        \item it is in standard form,
        \item all its terms are all positive or all negative except the last one, which can be 0,
        \item its first term is $1$ or $\overline{1}$ only if it is the only term,
        \item its first term is 0 only if it is tangle 0 or 00.
    \end{itemize}
\end{definition}
\begin{theorem}[Kauffman \cite{kauffman2004classification}]\label{thm:canonical_form}
    Every rational tangle can be brought to a canonical form by isotopy.
\end{theorem}
\begin{theorem}[Kauffman \cite{kauffman2004classification}]\label{thm:canonical_unique}
    Every rational tangle has exactly one unique canonical form.
\end{theorem}
Note that the canonical form, defined as in \autoref{def:canonical_form}, matches the output of Euclid's algorithm (cf.\ decomposition on \autoref{fig:rational_decomposition}). Note that Kauffman used a slightly different 
definition of the canonical form but followed the same principles. The difference is that we used the relation $1 = 10$ to keep the minimal number of integral tangles, while Kauffman preferred to keep their number odd.

\begin{figure}
    \centering
    \begin{tikzpicture}
        \node[tree={T1}] (tree1) at (0,0)
        {
            \node (T1) {};\\
        };
        \node[labelright={T1}] {$\frac{p}{q}$};
        \node at (.8,0) {$=$};
        \node[tree={mid2}, row sep={.45cm}] (tree2) at (1.4,0)
        {
            \node (T2) {}; & \node (R2) {};\\
            \node[phantom] (mid2) {};\\
            \node (L2) {};\\
        };
        \node[labelright={R2}] {$\left\lfloor \frac{p}{q} \right\rfloor$};
        \node[labelright={L2}] {$\frac{q}{p \bmod{q}}$};
        \draw 
            (T2) edge (R2)
            (T2) edge (L2);
        \node at (3.6,0) {$=$};
        \node[tree={L3}, row sep={.45cm}] (tree3) at (4.2,0)
        {
            \node (T3) {}; & \node (R3) {};\\
            \node[phantom] {};\\
            \node (L3) {}; & \node (LR3) {};\\
            \node[phantom] {};\\
            \node (LL3) {};\\
        };
        \node[labelright={R3}] {$\left\lfloor \frac{p}{q} \right\rfloor$};
        \node[labelright={LR3}] {$\left\lfloor \frac{q}{p \bmod{q}} \right\rfloor$};
        \node[labelright={LL3}] {$\frac{p \bmod{q}}{q \bmod{(p \bmod{q})}}$};
        \draw 
            (T3) edge (R3)
            (T3) edge (L3)
            (L3) edge (LR3)
            (L3) edge (LL3);
        \node at (7.4,0) {$=$};
        \node[tree={mid4}, row sep={.45cm}] (tree4) at (8,0)
        {
            \node (T4) {}; & \node (R4) {};\\
            \node[phantom] {};\\
            \node (L4) {}; & \node (LR4) {};\\
            \node[phantom] (mid4) {};\\
            \node (LL4) {}; & \node (LLR4) {};\\
            \node[phantom] {};\\
            \node (LLL4) {};\\
        };
        \node[labelright={R4}] {$\left\lfloor \frac{p}{q} \right\rfloor$};
        \node[labelright={LR4}] {$\left\lfloor \frac{q}{p \bmod{q}} \right\rfloor$};
        \node[labelright={LLR4}] 
            {$\left\lfloor \frac{p \bmod{q}}{q \bmod{(p \bmod{q})}} \right\rfloor$};
        \node[labelright={LLL4}] 
            {$\frac{q \bmod{(p \bmod{q})}}{(p \bmod{q}) \bmod{(q \bmod{(p \bmod{q})})}}$};
        \draw 
            (T4) edge (R4)
            (T4) edge (L4)
            (L4) edge (LR4)
            (L4) edge (LL4)
            (LL4) edge (LLR4)
            (LL4) edge (LLL4);
    \end{tikzpicture}
    \caption{Decomposition of a positive rational tangle represented with a tangle tree. Left-leaf can be recurrently decomposed until its fraction is integral. Note that the decomposition is equivalent to Euclid's algorithm of finding $\gcd(p,q)$ -- consecutive generated integral tangles (right-children) are remainders of the operation.}
    \label{fig:rational_decomposition}
\end{figure}
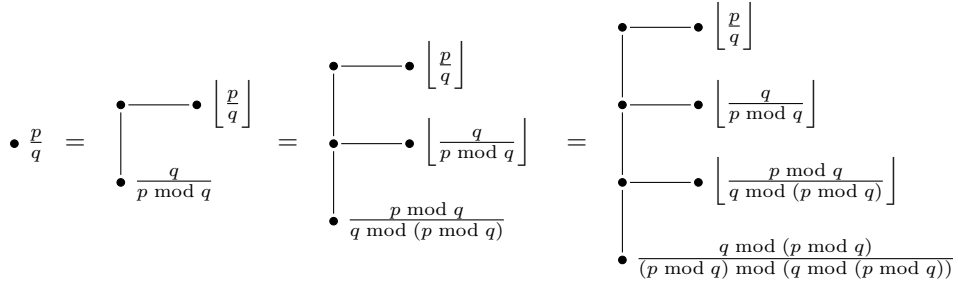


\subsection{Fraction of an algebraic tangle}\label{sec:fraction_algebraic}

{Here we generalize the fraction of rational tangles to the fraction of non-rational algebraic tangles, although, the generalized fraction loses the property of being a perfect invariant.}

\begin{definition}\label{def:fraction_algebraic}
The fraction $p/q$, where $p,q \in \mathbb{Z}$ (including $1/0=\infty$), of a tangle $A=LR$, $A$,$L$,$R\in \algebraic$ is defined as follows:
    \begin{align*}
         \mathrm{Frac}(A) &= \mathrm{Frac}(LR) = \mathrm{Frac}(R)\!+\!\frac{1}{\mathrm{Frac}(L)} = p/q.
    \end{align*}
\end{definition}
\noindent Which, for chain-multiplication of tangles (\autoref{eq:tree_decomposition_mul}), coincides with continued fraction:
\begin{equation}
     \mathrm{Frac}(A) = \mathrm{Frac}\left(\prod_{i=1}^N a_i\right) = 
    \mathrm{Frac}(a_N)\!+\!\frac{1}{\mathrm{Frac}(a_{N-1})+ \frac{1}{\ddots_{\;\mathrm{Frac}(a_2) + \frac{1}{\mathrm{Frac}(a_1)}} }} = p/q,
\end{equation}
and is linear with respect to addition (\autoref{eq:tree_decomposition_sum}):
\begin{equation}
     \mathrm{Frac}(A) = \mathrm{Frac}\left(\sum_{i=1}^N a_i \right) = 
     \sum_{i=1}^N \mathrm{Frac}(a_i) = p/q.
\end{equation}

\begin{lemma}\label{thm:fraction_rotations}
    The fraction of an algebraic tangle is invariant to rotations around $x,y,z$ axes.
\end{lemma}
\begin{proof}
    For a rotations of tangle $A=LR$ (\autoref{eq:rotations_multiplication}) we obtain:
    \begin{align*}
        \mathrm{Frac}(A_x) &= \mathrm{Frac}(L_yR_x) = \mathrm{Frac}(R_x)\!+\!\frac{1}{\mathrm{Frac}(L_y)};\\
        \mathrm{Frac}(A_y) &= \mathrm{Frac}\big((R_y0)(L_x0)\big) = \frac{1}{\mathrm{Frac}(L_x)} + \frac{1}{1/\mathrm{Frac}(R_y)} = \mathrm{Frac}(R_y) + \frac{1}{\mathrm{Frac}(L_x)}.
    \end{align*}
    Rotating a tangle around $x$ or $y$ axis does not change the expansion form of the fraction. After repeatedly expanding the tangle into rational or integral components, one can see that the structure of the rotated and non-rotated tangles remains the same; only the rotation-indices are different. Since rational tangles are invariant under rotation, the indices can be ignored.
\end{proof}

\begin{lemma}\label{thm:fraction_mirror}
    Switching all signs of an algebraic tangle changes the sign of its fraction: $\mathrm{Frac}(\overline{A}) = -\mathrm{Frac}(A)$, $A \in \algebraic$.
\end{lemma}
\begin{proof}
    By similar argument as in \autoref{thm:fraction_rotations}. {Since the statement is holds for rational tangles, it also holds for algebraic tangles.}
\end{proof}

\begin{lemma}\label{thm:fraction_flype}
    The fraction of an algebraic tangle is invariant to the elementary moves, flypes, twists, and ring moves.
\end{lemma}
\begin{proof}
{The proofs for elementary are trivial. We show proofs for a flype $A\flype \overline{A1}1\overline{1}$, a twist $A\twist \overline{1}0A_x01$ a ring move, as all other cases follow analogously:}
    \begin{align*}
        \mathrm{Frac}(\overline{A1}1\overline{1}) \;&
        = -1 + \frac{1}{1 + \frac{1}{-1+ \frac{1}{\mathrm{Frac}(\overline{A})}}} = -\mathrm{Frac}(\overline{A}) = \mathrm{Frac}(A).\\
        \mathrm{Frac}(10A_x0\overline{1}) \;&
        = -1 + \frac{1}{0+\frac{1}{\mathrm{Frac}(A_x) + \frac{1}{0+\frac{1}{1}}}} = 
        \mathrm{Frac}(A_x) = \mathrm{Frac}(A).\\
         \mathrm{Frac}(A0(2(\overline{2}0)0)) \;&
        = \frac{1}{-\frac{1}{2}+\frac{1}{2}} + \frac{1}{0 + \frac{1}{\mathrm{Frac}(A)}} = \mathrm{Frac}(A) + \frac{1}{-\frac{1}{2} + \frac{1}{2}} = \mathrm{Frac}(2(\overline{2}0)A) 
    \end{align*}
\end{proof}

Therefore, the generalized fraction is invariant to all isotopy-preserving moves which we defined in \autoref{sec:isotopic_moves}, but it is not a perfect invariant, e.g. $2(20) \neq 1$ but $\text{Fraction}(2(20)) = \text{Fraction}(1)$.


\subsection{Canonical representation}
Here we define the canonical representation of algebraic tangles and show how to obtain it. The goal is to obtain a form in which twist, flype, ring and elementary moves are ``fixed'' in a specific position, and provide an algorithm of obtaining such a form.

{
While the canonical form of a rational tangle is alternating (see \autoref{def:canonical_form} and \autoref{thm:canonical_form}), making it an ideal candidate, this property does not extend to all tangles. For instance, the ring $2(\overline{2}0)$ lacks any alternating form. Furthermore, prime non-rational tangles that do admit alternating forms have at least four such representations, with the number increasing rapidly as crossings grow. Thus, a unique representation must rely on other rules.}

\begin{definition}[Canonical representation of prime algebraic tangles]
    A prime algebraic tangle $A$ is in canonical representation when its tree $T(A)$ 
    obeys following restrictions:
    \begin{itemize}
        \item on each leaf:
            \begin{itemize}
                \item all right-leaves are integral tangles;
                \item all left-leaves are positive rational tangles (represented by fractions) with fractions larger than 1;
            \end{itemize}
        \item on each node, which is left-child:
            \begin{itemize}
                \item it cannot be a parent of two leaves;
                \item if its right-child is a leaf, then this leaf must be positive;
                \item if its right-child is a $1$ tangle, then the right-child of its left-child must be a leaf.
            \end{itemize}
        \item on each node with a ring ($2(\overline{2}0)$, or in canonical form, $2(2\overline{1})$) as a left-child:
            \begin{itemize}
                \item its right-child must be a leaf, or a parent of a ring as a left-child as well.
            \end{itemize}
    \end{itemize}
    Forbidden subtress which follow from these rules are presented in \autoref{fig:restricted}.
\end{definition}
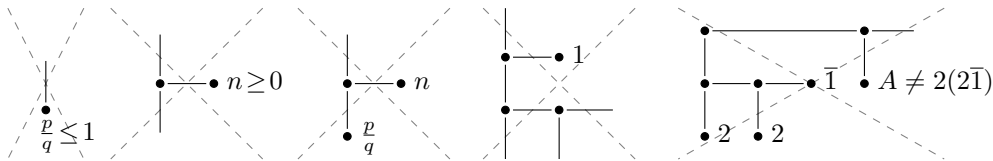
\begin{figure}[ht!]
    \begin{tikzpicture}[baseline=(mid)]
        \coordinate (mid) at (0,0);
        \node[tree_mid={T}] (tree6) at (0,.35)
        {
            \node[phantom] (T) {};\\
            \node (L) {}; \\
        };
        \coordinate (NW) at (-.5cm, 1cm);
        \coordinate (NE) at (.5cm, 1cm);
        \coordinate (SW) at (-.5cm, -1cm);
        \coordinate (SE) at (.5cm, -1cm);
        \node[labelbelow={L}] {$\frac{p}{q}\!\leq\!1$};
        \draw 
            (T) edge (L)
            (NE) edge[dashed, opacity=.5] (SW)
            (NW) edge[dashed, opacity=.5] (SE);
    \end{tikzpicture}\;
    \begin{tikzpicture}[baseline=(mid)]
        \coordinate (mid) at (0,0);
        \node[tree_mid={L}] (tree2) at (-.35,0)
            {
                \node[phantom] (T) {};\\
                \node (L) {}; & \node (LR) {};\\
                \node[phantom] (LL) {};\\
            };
        \coordinate (NW) at (-1cm, 1cm);
        \coordinate (NE) at (1cm, 1cm);
        \coordinate (SW) at (-1cm, -1cm);
        \coordinate (SE) at (1cm, -1cm);
        \node[labelright={LR}] {$n\!\geq\!0$};
        \draw 
            (T) edge (L)
            (L) edge (LR)
            (L) edge (LL)
            (NE) edge[dashed, opacity=.5] (SW)
            (NW) edge[dashed, opacity=.5] (SE);
    \end{tikzpicture}\;
    \begin{tikzpicture}[baseline=(mid)]
        \coordinate (mid) at (0,0);
        \node[tree_mid={L}] (tree2) at (-.35,0)
            {
                \node[phantom] (T) {};\\
                \node (L) {}; & \node (LR) {}; & \node[phantom] {};\\
                \node (LL) {};\\
            };
        \coordinate (NW) at (-1cm, 1cm);
        \coordinate (NE) at (1cm, 1cm);
        \coordinate (SW) at (-1cm, -1cm);
        \coordinate (SE) at (1cm, -1cm);
        \node[labelright={LR}] {$n$};
        \node[labelright={LL}] {$\frac{p}{q}$};
        \draw 
            (T) edge (L)
            (L) edge (LR)
            (L) edge (LL)
            (NE) edge[dashed, opacity=.5] (SW)
            (NW) edge[dashed, opacity=.5] (SE);
    \end{tikzpicture}\;\,
    \begin{tikzpicture}[baseline=(mid)]
        \coordinate (mid) at (0,0);
        \node[tree_mid={LR}] (tree2) at (0,.35)
            {
                \node[phantom] (T) {}; \\
                \node (L) {};  & \node (LR) {}; \\
                \node (LL) {}; & \node (LLR) {}; & \node[phantom] (LLRR) {}; \\
                \node[phantom] (LLL) {}; & \node[phantom] (LLRL) {};\\
            };
        \coordinate (NW) at (-1cm, 1cm);
        \coordinate (NE) at (1cm, 1cm);
        \coordinate (SW) at (-1cm, -1cm);
        \coordinate (SE) at (1cm, -1cm);
        \node[labelright={LR}] {$1$};
        \draw 
            (T) edge (L)
            (L) edge (LR)
            (L) edge (LL)
            (LL) edge (LLR)
            (LL) edge (LLL)
            (LLR) edge (LLRL)
            (LLR) edge (LLRR)
            (NE) edge[dashed, opacity=.5] (SW)
            (NW) edge[dashed, opacity=.5] (SE);
    \end{tikzpicture}
    \begin{tikzpicture}[baseline=(L)]
        \node[tree_mid={LRR}] (tree2) at (0,0)
            {
                & \node (T) {}; &&&\node (R) {}; & \node[phantom] (RR) {};\\
                & \node (L) {};  & \node (LR) {}; & \node (LRR) {}; 
                & \node (RL) {}; \\
                & \node (LL) {}; & \node (LRL) {}; \\
            };
        \coordinate (NW) at (-1.75cm, 1cm);
        \coordinate (NE) at (1.75cm, 1cm);
        \coordinate (SW) at (-1.75cm, -1cm);
        \coordinate (SE) at (1.75cm, -1cm);
        \node[labelright={LL}] {$2$};
        \node[labelright={LRL}] {$2$};
        \node[labelright={LRR}] {$\overline{1}$};
        \node[labelright={RL}] {$A \neq 2(2\overline{1})$};
        \draw 
            (T) edge (L)
            (T) edge (R)
            (L) edge (LR)
            (L) edge (LL)
            (LR) edge (LRR)
            (LR) edge (LRL)
            (R) edge (RR)
            (R) edge (RL)
            (NE) edge[dashed, opacity=.5] (SW)
            (NW) edge[dashed, opacity=.5] (SE);
    \end{tikzpicture}
    \caption{Tangle trees in canonical representation must have left-leaves represented as rational tangles $\frac{p}{q}$, right-leaves represented as integral tangles $n$, and cannot have any subtangles are listed above.}
    \label{fig:restricted}
\end{figure}

\begin{theorem} \label{thm:main}
    The canonical form of a prime algebraic tangle is unique up to twist, flype, ring, and elementary moves
\end{theorem}
We will proof the theorem by an algorithm (\autoref{alg:treefix}) that puts a prime tangle into its canonical form and show that the algorithm returns the same result after performing twist, flype, ring and/or elementery moves.

\begin{table}[h!]
    \caption{Possible moves for traversals of the tree-fix algorithm. Empty dot represent a node visited by an algorithm. $m$ and $n$ represent integral tangles, while $k$ represents strictly positive integral tangle.}
    \label{tab:traversal_moves}
    \centering
    \begin{tabular}{llp{3.2cm}lll}
        \toprule
        \multicolumn{2}{l}{Part 1. moves:}&&&&\\
        \cmidrule{1-2}
        \begin{tikzpicture}[scale=0.7, baseline=(T1)]
            \node[tree_mid={T1}] (tree1)
            {
                \node[phantom] (T1) {}; & \node[empty] (R1) {}; & \node (RR1) {};\\
                &\node (RL1) {};\\
            };
            \node[labelright={RR1}] {$B^m$};
            \node[labelright={RL1}] {$1/q$};
            \draw 
                (T1) edge[dottededge] (R1)
                (R1) edge (RR1)
                (R1) edge (RL1);
        \end{tikzpicture}%
        &
        \begin{tikzpicture}[baseline=(m1)]
            \node[tnode] (m1) {$\arra{\mathbb{Z}}$};
        \end{tikzpicture}%
        &
        \begin{tikzpicture}[scale=0.7, baseline=(T1)]
            \node[tree_mid={T1}] (tree1)
            {
                \node[phantom] (T1) {}; & \node (R1) {};\\
            };
            \node[labelright={R1}] {$B_{x^q}^{m+ q}$};
            \draw 
                (T1) edge[dottededge] (R1);
            \end{tikzpicture}%
        &
        \begin{tikzpicture}[scale=0.7, baseline=(L1)]
            \node[tree_mid={L1}] (tree1)
            {
                \node[phantom] (T1) {};\\
                \node[empty] (L1) {}; & \node (LR1) {};\\
                \node (LL1) {};\\
            };
            \node[labelright={LR1}] {$n$};
            \node[labelright={LL1}] {$p/q$};
            \draw 
                (T1) edge (L1)
                (L1) edge (LR1)
                (L1) edge (LL1);
        \end{tikzpicture}%
        &
        \begin{tikzpicture}[baseline=(m1)]
            \node[tnode] (m1) {$\arra{\mathbb{Q}}$};
        \end{tikzpicture}%
        &
        \begin{tikzpicture}[scale=0.7, baseline=(L1)]
            \node[tree_mid={L1}] (tree1)
            {
                \node[phantom] (T1) {};\\
                \node (L1) {};\\
            };
            \node[labelright={L1}] {$\frac{q+pn}{p}$};
            \draw 
                (T1) edge (L1);
        \end{tikzpicture}\\
        \phantom{aa}&&&&& \\
        \midrule
        \phantom{aa}&&&&& \\
        
        \multicolumn{2}{l}{Part 2. moves:}&&&&\\
        \cmidrule{1-2}
        \begin{tikzpicture}[scale=0.7, baseline=(T1)]
            \node[tree_mid={T1}] (tree1)
            {
                \node[empty] (T1) {}; & \node (R1) {};\\
                \node (L1) {};\\
            };
            \node[labelright={R1}] {$B^m$};
            \node[labelright={L1}] {$p/q$};
            \draw 
                (T1) edge (R1)
                (T1) edge (L1);
        \end{tikzpicture}%
        &
        \begin{tikzpicture}[baseline=(m1)]
            \node[tnode] (m1) {$\arrb{\mathbb{Q}}$};
        \end{tikzpicture}%
        &
        \begin{tikzpicture}[scale=0.7, baseline=(T1)]
            \node[tree_mid={T1}] (tree1)
            {
                \node (T1) {}; & \node (R1) {};\\
                \node (L1) {};\\
            };
            \node[labelright={R1}] {$B_{x^{\lfloor q/p \rfloor}}^{m+\lfloor q/p \rfloor}$};
            \node[labelright={L1}] {$\frac{p}{q \bmod{p}}$};
            \draw 
                (T1) edge (R1)
                (T1) edge (L1);
        \end{tikzpicture}%
        &
        \begin{tikzpicture}[scale=0.7, baseline=(T1)]
            \node[tree_mid={T1}] (tree1)
            {
                \node[empty] (T1) {}; & \node (R1) {};\\
                \node (L1) {}; & \node (LR1) {};\\
                \node (LL1) {}; \\
            };
            \node[labelright={R1}] {$B^m$};
            \node[labelright={LR1}] {$0$};
            \node[labelright={LL1}] {$A^n$};
            \draw 
                (T1) edge (R1)
                (T1) edge (L1)
                (L1) edge (LR1)
                (L1) edge (LL1);
        \end{tikzpicture}%
        &
        \begin{tikzpicture}[baseline=(m1)]
            \node[tnode] (m1) {$\arrb{0}$};
        \end{tikzpicture}%
        &
        \begin{tikzpicture}[scale=0.7, baseline=(T1)]
            \node[tree_mid={T1}] (tree1)
            {
                \node (R1) {};\\
            };
            \node[labelright={R1}] {$A^{B_{x^n}^{n+m}}$};
        \end{tikzpicture}\\
        \phantom{aa}&&&&&\\

        \begin{tikzpicture}[scale=0.7, baseline=(T1)]
            \node[tree_mid={T1}] (tree1)
            {
                \node (T1) {}; & \node (R1) {};\\
                \node[empty] (L1) {}; & \node (LR1) {};\\
                \node (LL1) {}; \\
            };
            \node[labelright={R1}] {$B^m$};
            \node[labelright={LR1}] {$\overline{k}$};
            \node[labelright={LL1}] {$A$};
            \draw 
                (T1) edge (R1)
                (T1) edge (L1)
                (L1) edge (LR1)
                (L1) edge (LL1);
        \end{tikzpicture}%
        &
        \begin{tikzpicture}[baseline=(m1)]
            \node[tnode] (m1) {$\arrb{-}$};
        \end{tikzpicture}%
        &
        \begin{tikzpicture}[scale=0.7, baseline=(T1)]
            \node[tree_mid={T1}] (tree1)
            {
                \node (T1) {}; & \node (R1) {};\\
                \node (L1) {}; & \node (LR1) {};\\
                \node (LL1) {}; & \node (LLR1) {};\\
                \node (LLL1) {};\\
            };
            \node[labelright={R1}] {$B_x^{m-1}$};
            \node[labelright={LR1}] {$1$};
            \node[labelright={LLR1}] {$k-1$};
            \node[labelright={LLL1}] {$\overline{A}$};
            \draw 
                (T1) edge (R1)
                (T1) edge (L1)
                (L1) edge (LR1)
                (L1) edge (LL1)
                (LL1) edge (LLR1)
                (LL1) edge (LLL1);
        \end{tikzpicture}%
        &
        \begin{tikzpicture}[scale=0.7, baseline=(T1)]
            \node[tree_mid={T1}] (tree1)
            {
                \node (T1) {}; & \node (R1) {};\\
                \node[empty] (L1) {}; & \node (LR1) {};\\
                \node (LL1) {}; & \node (LLR1) {};  & \node (LLRR1) {};\\
                \node (LLL1) {}; & \node (LLRL1) {};\\
            };
            \node[labelright={R1}] {$B^m$};
            \node[labelright={LR1}] {$1$};
            \node[labelright={LLRR1}] {$A^n$};
            \node[labelright={LLL1}] {$C$};
            \node[labelright={LLRL1}] {$D$};
            \draw 
                (T1) edge (R1)
                (T1) edge (L1)
                (L1) edge (LR1)
                (L1) edge (LL1)
                (LLR1) edge (LLRR1)
                (LLR1) edge (LLRL1)
                (LL1) edge (LLR1)
                (LL1) edge (LLL1);
        \end{tikzpicture}%
        &
        \begin{tikzpicture}[baseline=(m1)]
            \node[tnode] (m1) {$\arrb{1}$};
        \end{tikzpicture}%
        &
       \begin{tikzpicture}[scale=0.7, baseline=(T1)]
            \node[tree_mid={T1}] (tree1)
            {
                \node (T1) {}; & \node (R1) {};\\
                \node (L1) {}; & \node (LR1) {};  & \node (LRR1) {};\\
                \node (LL1) {}; & \node (LRL1) {};\\
            };
            \node[labelright={R1}] {$B_x^{m+1}$};
            \node[labelright={LRR1}] {$\overline{A^{n+1}}$};
            \node[labelright={LL1}] {$\overline{C}$};
            \node[labelright={LRL1}] {$\overline{D}$};
            \draw 
                (T1) edge (R1)
                (T1) edge (L1)
                (L1) edge (LR1)
                (L1) edge (LL1)
                (LR1) edge (LRR1)
                (LR1) edge (LRL1);
        \end{tikzpicture}\\
        \phantom{aa}&&&&&\\
        \midrule
        \phantom{aa}&&&&& \\
        
        \multicolumn{2}{l}{Part 3. move:}&&&&\\
        \cmidrule{1-2}
        \begin{tikzpicture}[baseline=(L1.base)]
            \node[tree_mid={T1}] (tree1)
            {
                \node[empty] (T1) {}; & \node (R1) {}; & \node (RR1) {};\\
                \node (L1) {};  & \node (LR1) {}; & \node (LRR1) {};\\
                \node (LL1) {}; & \node (LRL1) {};\\
            };
            \node[labelabove={R1}] {$B^m$};
            \node[labelright={RR1}] {$m$};
            \node[labelright={LL1}] {$2$};
            \node[labelright={LRL1}] {$2$};
            \node[labelright={LRR1}] {$\overline{1}$};
            \draw 
                (T1) edge (R1)
                (T1) edge (L1)
                (L1) edge (LR1)
                (L1) edge (LL1)
                (LR1) edge (LRL1)
                (LR1) edge (LRR1)
                (R1) edge[dottededge] (RR1);
        \end{tikzpicture}%
        &
        \begin{tikzpicture}[baseline=(aa.base)]
            \node (aa) {$\xrightarrow{\gamma}$};
        \end{tikzpicture}%
        &
        \begin{tikzpicture}[baseline=(RL2.base)]
            \node[tree_mid={T2}] (tree2)
            {
                \node (T2) {}; & \node (R2) {}; & \node (RR2) {};\\
                & \node (RL2) {}; & \node (RLR2) {}; & \node (RLRR2) {}; \\
                & \node (RLL2) {}; & \node (RLRL2) {}; \\
            };
            \node[labelabove={T2}] {$B^{2(2\overline{1})m}$};
            \node[labelright={RR2}] {$m$};
            \node[labelright={RLL2}] {$2$};
            \node[labelright={RLRL2}] {$2$};
            \node[labelright={RLRR2}] {$\overline{1}$};
            \draw 
                (T2) edge[dottededge] (R2)
                (R2) edge (RR2)
                (R2) edge (RL2)
                (RL2) edge (RLL2)
                (RL2) edge (RLR2)
                (RLR2) edge (RLRL2)
                (RLR2) edge (RLRR2);
        \end{tikzpicture}\\
        \phantom{aa}&&&&& \\
        \bottomrule
    \end{tabular}
\end{table}

\begin{algorithm}
    \caption{Algorithm operates on objects $N$ of a \textsc{Node} class. Every \textsc{Node} object has attributes which point to its \textit{parent}, \textit{left\_child} and \textit{right\_child} and knows if it \textit{is\_a\_left\_child} of its \textit{parent}. If a \textsc{Node} represents a rational tangle, then it has a non-\texttt{None} value of a \textit{fraction}. \textsc{FixLocal} performs one of $\alpha_\mathbb{Z}$, $\alpha_\mathbb{Q}$, $\beta_\mathbb{Q}$, $\beta_0$, $\beta_-$, $\beta_1$, $\gamma$ moves on a node $N$ if it is possible (moves are presented in \autoref{tab:traversal_moves}).} \label{alg:treefix}
    \begin{algorithmic}[1]
        \Function{GetCanonicalTree}{$T$} \Comment{$T$ is address to a tree-top \textsc{Node}}
            \State $nodes \gets$ \Call{Reversed(PreOrderTraversal}{$T$}) \Comment{1st part}
            \For{$N$ in $nodes$} 
                \State \Call{FixLocal}{$N$,[$\alpha_\mathbb{Z}$, $\alpha_\mathbb{Q}$]}
            \EndFor
            \Repeat \Comment{2nd part}
                \State $changed \gets$ \texttt{False}
                \State $nodes \gets$ \Call{PreOrderTraversal}{$T$}
                \For{$N$ in $nodes$}
                    \State $changed \gets$ \Call{FixLocal}{$N$,[$\beta_\mathbb{Q}$, $\beta_0$, $\beta_-$, $\beta_1$]}
                    \If {$changed=$ \texttt{True}}
                        \State \textbf{break the for loop}
                    \EndIf
                \EndFor
            \Until {$changed=$ \texttt{False}}
            \State $nodes \gets$ \Call{Reversed(PreOrderTraversal}{$T$}) \Comment{3rd part}
            \For{$N$ in $nodes$} 
                \State \Call{FixLocal}{$N$,[$\gamma$]}
            \EndFor
            \State \Return $T$
        \EndFunction
        \Statex
        \Function{FixLocal}{$N$, $moves$} 
            \For{$move$ in $moves$}
                \If {$move$ on node $N$ is possible}
                    \State \textbf{perform} $move$ on a node $N$
                    \State \Return \texttt{True}
                \EndIf
            \EndFor
            \State \Return \texttt{False}
        \EndFunction
        \Statex
        \Function{PreOrderTraversal}{$N$}
            \If {$N.fraction =$ \texttt{None}}
                \State \Return [$N$]
            \Else 
                \State $l \gets$ \Call{PreOrderTraversal}{$N.left\_child$}
                \State $r \gets$ \Call{PreOrderTraversal}{$N.right\_child$}
                \State \Return joined lists $l$ and $r$
            \EndIf
        \EndFunction
    \end{algorithmic}
\end{algorithm}

To introduce the algorithm, we need to define ordering of tree elements, and the theory of binary trees already gives us such tools. We will define both a top-down ordering and a bottom-up ordering. The top-down ordering results from a pre-order traversal, e.g. for tree from \autoref{tree:expression_tree}, the top-down ordering is $[A, L, L_L, L_{L_L}, R_{L_L}, R_L, L_{R_L}, R_{R_L}, R, L_R, L_{L_R}, R_{L_R}, R_R, L_{R_R}, R_{R_R}]$. The bottom-up ordering is the top-down ordering reversed.

\autoref{alg:treefix} transforms an algebraic tangle, represented by a binary tree, into its canonical representation (\autoref{thm:canonical_algorithm}). {In general, when the tree is traversed, and when a node is visited, possible isotopy preserving moves are performed on the node.} The algorithm is divided into three parts, which differ how the tree is traversed and which moves $\alpha_\mathbb{Z}$, $\alpha_\mathbb{Q}$, $\beta_\mathbb{Q}$, $\beta_0$, $\beta_-$, $\beta_1$, $\gamma$, presented in \autoref{tab:traversal_moves}, are possible to perform: 
{
\begin{enumerate}
    \item In the first part, we represent rational tangles in a possibly concise way and ensure that e.g., a subtangle $0$ is represented as $0$, not as $2(2\overline{1}0)1\overline{21}$.
    This prepares the algorithm for the second part to let it correctly recognize subtree patterns. At the same time, the twists are performed to keep twistable integral tangles as right leaves. The tree is traversed bottom-up once, and each node $\alpha_\mathbb{Z}$ (twist: $q\!+\!B \twist B_{x^q}\!+\!q$) and $\alpha_\mathbb{Q}$ (contraction: \autoref{fig:rational_contraction}) moves are performed if possible.
    \item In the second part, we perform the flypes required to follow the canonical representation. The tree is traversed top-down, and each time a move is performed, the traversal starts again from the top. 
    The move $\beta_0$ (\autoref{thm:generalized_twist}) is equivalent to bracket juggling and twisting integral tangles, to remove $0$ subtangles when possible. The moves $\beta_\mathbb{Q}$, $\beta_-$, $\beta_1$ are flypes followed by twists. 
    Sometimes both $\beta_0$ and $\beta_-$ could be performed, but $\beta_0$ has higher priority, to ensure that $A^n0\overline{k}B^m$ is twisted into $A^{n-k}B^m$, instead of being flyped into $\overline{A}1B_x^{m-1}$. The algorithm follows \qq{twists before flypes} rule.
    The priority couldn't be ensured by adding $\beta_0$ to the first part of the algorithm, since $0$-right leaf can appear as a result of both $\beta_-$ and $\beta_1$ moves.

    \item After the second part, the only missing thing is correct positioning of rings $2(\overline{2}0)$ (which in canonical representation are represented by $2(2\overline{1})$). Their position is fixed in the third part, by bottom-up traversal and the $\gamma$ move.
\end{enumerate}
}

\begin{lemma}\label{thm:algorithm_invariant_other}
    \autoref{alg:treefix} gives the same output for tangles in a tree form that differ by a twist, ring, elementary moves, or bracket juggling.
\end{lemma}
\begin{proof}
    The ring move switches horizontally two neighboring subtrees (one of which is a ring ($2(2\overline{1})$)). The third part of the algorithm \qq{pushes} all rings to the right, fixing positions of all $2(2\overline{1})$ subtangles.
    
    The first part of the algorithm collects bottom-up rational tangle subtangles into a single fraction, which ensures that every rational tangle is in its simplest form before the second part of the algorithm.
    \begin{center}
        \begin{tabular}{ccccc}
            \begin{tikzpicture}[scale=0.7, baseline=(T1), transform shape]
                \node[tree_small={T1}] (tree1)
                {
                    \node (T1) {}; & \node (R1) {}; \\
                    \node (L1)[empty] {}; & \node (LR1) {}; \\
                    \node (LL1) {}; \\
                };
                \draw 
                    (T1) edge (R1)
                    (T1) edge (L1)
                    (L1) edge (LR1)
                    (L1) edge (LL1);
                \node[labelright={R1}] {$m$};
                \node[labelright={LR1}] {$0$};
                \node[labelright={LL1}] {$n$};
            \end{tikzpicture}%
            &
            \begin{tikzpicture}[baseline=(m1)]
                \node[tnode] (m1) {$\arra{\mathbb{Q}}$};
            \end{tikzpicture}%
            &
            \begin{tikzpicture}[scale=0.7, baseline=(T1), transform shape]
                \node[tree_small={T1}] (tree1)
                {
                    \node (T1)[empty] {}; & \node (R1) {}; \\
                    \node (L1) {}; \\
                };
                \draw 
                    (T1) edge (R1)
                    (T1) edge (L1);
                \node[labelright={R1}] {$m$};
                \node[labelright={L1}] {$\frac{1}{n}$};
            \end{tikzpicture}%
            &
            \begin{tikzpicture}[baseline=(m1)]
                \node[tnode] (m1) {$\arra{\mathbb{Q}}$};
            \end{tikzpicture}%
            &
            \begin{tikzpicture}[scale=0.7, baseline=(T2), transform shape]
                \node[tree_small={T2}] (tree2)
                {
                    \node (T2) {}; \\
                };
                \node[labelright={T2}] {$n+m$};
            \end{tikzpicture}
        \end{tabular}
    \end{center}

    The second part of the algorithm returns the same output for both $A0B0C$ and $A0(B0C)$ tangles. Anything that happens to tangle $A^{B0C}$ before the traversal reaches $B0C$, has the same impact on the $A^{B^{C}}$ tangle. Note that if the tangle $A$ is integral, then both $A0B0C$ and $A0(B0C)$ tangles become identical earlier after the first traversal. Twist moves, bracket juggling, and the II elementary move are special cases of this example.
    \begin{center}
        \begin{tabular}{ccccccccc}
            \begin{tikzpicture}[scale=0.7, baseline=(T1), transform shape]
                \node[tree_small={T1}] (tree1)
                {
                    \node (T1)[empty] {}; & \node (R1) {}; \\
                    \node (L1) {}; & \node (LR1) {}; \\
                    \node (LL1) {}; & \node (LLR1) {}; \\
                    \node (LLL1) {}; & \node (LLLR1) {}; \\
                    \node (LLLL1) {}; \\
                };
                \draw 
                    (T1) edge (R1)
                    (T1) edge (L1)
                    (L1) edge (LR1)
                    (L1) edge (LL1)
                    (LL1) edge (LLR1)
                    (LL1) edge (LLL1)
                    (LLL1) edge (LLLR1)
                    (LLL1) edge (LLLL1);
                \node[labelright={R1}] {$C$};
                \node[labelright={LR1}] {$0$};
                \node[labelright={LLR1}] {$B$};
                \node[labelright={LLLR1}] {$0$};
                \node[labelright={LLLL1}] {$A$};
            \end{tikzpicture}%
            &
            \begin{tikzpicture}[baseline=(m1)]
                \node[tnode] (m1) {$\arrb{0}$};
            \end{tikzpicture}%
            &
            \begin{tikzpicture}[scale=0.7, baseline=(T1), transform shape]
                \node[tree_small={T1}] (tree1)
                {
                    \node (T1)[empty] {}; & \node (R1) {}; \\
                    \node (L1) {}; & \node (LR1) {}; \\
                    \node (LL1) {}; \\
                };
                \draw 
                    (T1) edge (R1)
                    (T1) edge (L1)
                    (L1) edge (LR1)
                    (L1) edge (LL1);
                \node[labelright={R1}] {$B^{C}$};
                \node[labelright={LR1}] {$0$};
                \node[labelright={LL1}] {$A$};
            \end{tikzpicture}%
            &
            \begin{tikzpicture}[baseline=(m1)]
                \node[tnode] (m1) {$\arrb{0}$};
            \end{tikzpicture}%
            &
            \begin{tikzpicture}[scale=0.7, baseline=(T1), transform shape]
                \node[tree_small={T1}] (tree1)
                {
                    \node (T1) {}; & \node (R1) {}; \\
                };
                \draw 
                    (T1) edge[dottededge] (R1);
                \node[labelabove={T1}] {$A^{B^{C}}$};
                \node[labelright={R1}] {$B^{C}$};
            \end{tikzpicture}
            &
            \begin{tikzpicture}[baseline=(m1)]
                \node[tnode] (m1) {$\xleftarrow{\mathbb{\beta}_0}$};
            \end{tikzpicture}%
            &
            \begin{tikzpicture}[scale=0.7, baseline=(T1), transform shape]
                \node[tree_small={T1}] (tree1)
                {
                    \node (T1) {}; & \node[empty] (R1) {}; & \node (RR1) {}; \\
                    & \node (RL1) {}; & \node (RLR1) {};\\
                    & \node (RLL1) {};\\
                };
                \draw 
                    (T1) edge[dottededge] (R1)
                    (R1) edge (RR1)
                    (R1) edge (RL1)
                    (RL1) edge (RLR1)
                    (RL1) edge (RLL1);
                \node[labelabove={T1}] {$A^{B0C}$};
                \node[labelright={RR1}] {$C$};
                \node[labelright={RLR1}] {$0$};
                \node[labelright={RLL1}] {$B$};
            \end{tikzpicture}%
            &
            \begin{tikzpicture}[baseline=(m1)]
                \node[tnode] (m1) {$\xleftarrow{\mathbb{\beta}_0}$};
            \end{tikzpicture}%
            &
            \begin{tikzpicture}[scale=0.7, baseline=(T1), transform shape]
                \node[tree_small={T1}] (tree1)
                {
                    \node (T1)[empty] {}; && \node (R1) {}; & \node (RR1) {}; \\
                    \node (L1) {}; & \node (LR1) {}; & \node (RL1) {}; & \node (RLR1) {};\\
                    \node (LL1) {}; && \node (RLL1) {};\\
                };
                \draw 
                    (T1) edge (R1)
                    (T1) edge (L1)
                    (L1) edge (LR1)
                    (L1) edge (LL1)
                    (R1) edge (RR1)
                    (R1) edge (RL1)
                    (RL1) edge (RLR1)
                    (RL1) edge (RLL1);
                \node[labelright={RR1}] {$C$};
                \node[labelright={RLR1}] {$0$};
                \node[labelright={RLL1}] {$B$};
                \node[labelright={LR1}] {$0$};
                \node[labelright={LL1}] {$A$};
            \end{tikzpicture}\\
            \multicolumn{9}{l}{\footnotesize Rotations were omitted for a clarity of the representation.}
        \end{tabular}
    \end{center}
\end{proof}

\begin{lemma}\label{thm:algorithm_invariant_flype}
    \autoref{alg:treefix} gives the same output for tangles in a tree form that differ by flype moves.
\end{lemma}
\begin{proof}
We will check all possible positions of a tangle $A$ in a tree, perform a flype at each position, and show that algorithm always returns the same output. These positions can be divided into 4 cases:
\begin{enumerate}
    \item $A$ is a right-child/tree-top;
    \item $A$ is a left-child of a right-child/tree-top;
    \item $A$ is a left-child of a left-child of a right-child/tree-top;
    \item $A$ is a left-child of a left-child of a left-child, which collects all other cases.
\end{enumerate}
In all cases, $A$ is a non-rational tangle. If $A$ is a rational tangle, then it is represented by a fraction after a the first part of the algorithm, which is invariant to flypes.

There are 4 types of flypes; we will provide a proof for only one of them ($A\flype \overline{A}01\overline{1}10$), as proofs for other 3 flypes are analogous or simpler.

\medskip
\noindent {\bf Case 1.} $A$ is a right-child/tree top.
    \begin{flushleft}
        \begin{tabular}{cccccccccc}
            \begin{tikzpicture}[scale=0.7, baseline=(T1), transform shape]
                \node[tree_small={T1}] (tree1)
                {
                    \node[phantom] (P1) {}; & \node (T1) {}; \\
                };
                \draw 
                    (P1) edge[dottededge] (T1);
                \node[labelright={T1}] {$A^n$};
            \end{tikzpicture}%
            &
            \begin{tikzpicture}[baseline=(m1)]
                \node[tnode] (m1) {$\flype$};
            \end{tikzpicture}%
            &
            \begin{tikzpicture}[scale=0.7, baseline=(T1), transform shape]
                \node[tree_small={T1}] (tree1)
                {
                    \node[phantom] (P1) {}; & \node (T1) {}; & \node (R1) {}; \\
                    & \node (L1) {}; & \node (LR1) {}; \\
                    & \node[empty] (LL1) {}; & \node (LLR1) {}; \\
                    & \node (LLL1) {}; & \node (LLLR1) {}; \\
                    & \node (LLLL1) {}; & \node (LLLLR1) {}; \\
                    & \node (LLLLL1) {}; \\
                };
                \draw 
                    (P1) edge[dottededge] (T1)
                    (T1) edge (R1)
                    (T1) edge (L1)
                    (L1) edge (LR1)
                    (L1) edge (LL1)
                    (LL1) edge (LLR1)
                    (LL1) edge (LLL1)
                    (LLL1) edge (LLLR1)
                    (LLL1) edge (LLLL1)
                    (LLLL1) edge (LLLLR1)
                    (LLLL1) edge (LLLLL1);
                \node[labelright={R1}] {$0$};
                \node[labelright={LR1}] {$1$};
                \node[labelright={LLR1}] {$\overline{1}$};
                \node[labelright={LLLR1}] {$1$};
                \node[labelright={LLLLR1}] {$0$};
                \node[labelright={LLLLL1}] {$\overline{A^n}$};
            \end{tikzpicture}%
            &
            \begin{tikzpicture}[baseline=(m1)]
                \node[tnode] (m1) {$\arrb{-}$};
            \end{tikzpicture}%
            &
            \begin{tikzpicture}[scale=0.7, baseline=(T1), transform shape]
                \node[tree_small={T1}] (tree1)
                {
                    \node[phantom] (P1) {}; & \node (T1) {}; & \node (R1) {}; \\
                    & \node (L1) {}; & \node (LR1) {}; \\
                    & \node[empty] (LL1) {}; & \node (LLR1) {}; \\
                    & \node (LLL1) {}; & \node (LLLR1) {}; \\
                    & \node (LLLL1) {}; & \node (LLLLR1) {}; \\
                    & \node (LLLLL1) {}; & \node (LLLLLR1) {}; \\
                    & \node (LLLLLL1) {}; \\
                };
                \draw 
                    (P1) edge[dottededge] (T1)
                    (T1) edge (R1)
                    (T1) edge (L1)
                    (L1) edge (LR1)
                    (L1) edge (LL1)
                    (LL1) edge (LLR1)
                    (LL1) edge (LLL1)
                    (LLL1) edge (LLLR1)
                    (LLL1) edge (LLLL1)
                    (LLLL1) edge (LLLLR1)
                    (LLLL1) edge (LLLLL1)
                    (LLLLL1) edge (LLLLLR1)
                    (LLLLL1) edge (LLLLLL1);
                \node[labelright={R1}] {$0$};
                \node[labelright={LR1}] {$0$};
                \node[labelright={LLR1}] {$1$};
                \node[labelright={LLLR1}] {$0$};
                \node[labelright={LLLLR1}] {$\overline{1}$};
                \node[labelright={LLLLLR1}] {$0$};
                \node[labelright={LLLLLL1}] {$A^n$};
            \end{tikzpicture}%
            &
            \begin{tikzpicture}[baseline=(m1)]
                \coordinate (m1) at (0,.1);
                \node[tnode] at (0,0) {$\xrightarrow[\times 3]{\beta_0}$};
            \end{tikzpicture}%
            &
            \begin{tikzpicture}[scale=0.7, baseline=(T1), transform shape]
                \node[tree_small={T1}] (tree1)
                {
                    \node[phantom] (P1) {}; & \node (T1) {}; \\
                };
                \draw 
                    (P1) edge[dottededge] (T1);
                \node[labelright={T1}] {$A^n$};
            \end{tikzpicture}%
        \end{tabular}
    \end{flushleft}

\medskip
\noindent {\bf Case 2.} $A$ is a left-child of a right-child/tree top ($B$).
    \begin{flushleft}
        \begin{tabular}{cccccccc}
            \begin{tikzpicture}[scale=0.7, baseline=(T1), transform shape]
                \node[tree_small={T1}] (tree1)
                {
                    \node[phantom] (P1) {}; & \node (T1) {}; & \node (R1) {}; \\
                    & \node (L1) {}; \\
                };
                \draw 
                    (P1) edge[dottededge] (T1)
                    (T1) edge (R1)
                    (T1) edge (L1);
                \node[labelright={R1}] {$B^m$};
                \node[labelright={L1}] {$A^n$};
            \end{tikzpicture}%
            &
            \begin{tikzpicture}[baseline=(m1)]
                \node[tnode] (m1) {$\flype$};
            \end{tikzpicture}%
            &
            \begin{tikzpicture}[scale=0.7, baseline=(T1), transform shape]
                \node[tree_small={T1}] (tree1)
                {
                    \node[phantom] (P1) {}; & \node[empty] (T1) {}; & \node (R1) {}; \\
                    & \node (L1) {}; & \node (LR1) {}; \\
                    & \node (LL1) {}; & \node (LLR1) {}; \\
                    & \node (LLL1) {}; & \node (LLLR1) {}; \\
                    & \node (LLLL1) {}; & \node (LLLLR1) {}; \\
                    & \node (LLLLL1) {}; & \node (LLLLLR1) {}; \\
                    & \node (LLLLLL1) {}; \\
                };
                \draw 
                    (P1) edge[dottededge] (T1)
                    (T1) edge (R1)
                    (T1) edge (L1)
                    (L1) edge (LR1)
                    (L1) edge (LL1)
                    (LL1) edge (LLR1)
                    (LL1) edge (LLL1)
                    (LLL1) edge (LLLR1)
                    (LLL1) edge (LLLL1)
                    (LLLL1) edge (LLLLR1)
                    (LLLL1) edge (LLLLL1)
                    (LLLLL1) edge (LLLLLR1)
                    (LLLLL1) edge (LLLLLL1);
                \node[labelright={R1}] {$B^m$};
                \node[labelright={LR1}] {$0$};
                \node[labelright={LLR1}] {$1$};
                \node[labelright={LLLR1}] {$\overline{1}$};
                \node[labelright={LLLLR1}] {$1$};
                \node[labelright={LLLLLR1}] {$0$};
                \node[labelright={LLLLLL1}] {$\overline{A^n}$};
            \end{tikzpicture}%
            &
            \begin{tikzpicture}[baseline=(m1)]
                \node[tnode] (m1) {$\arrb{0}$};
            \end{tikzpicture}%
            &
            \begin{tikzpicture}[scale=0.7, baseline=(T1), transform shape]
                \node[tree_small={T1}] (tree1)
                {
                    \node[phantom] (P1) {}; & \node (T1) {}; & \node (R1) {}; \\
                    & \node[empty] (L1) {}; & \node (LR1) {}; \\
                    & \node (LL1) {}; & \node (LLR1) {}; \\
                    & \node (LLL1) {}; & \node (LLLR1) {}; \\
                    & \node (LLLL1) {}; \\
                };
                \draw 
                    (P1) edge[dottededge] (T1)
                    (T1) edge (R1)
                    (T1) edge (L1)
                    (L1) edge (LR1)
                    (L1) edge (LL1)
                    (LL1) edge (LLR1)
                    (LL1) edge (LLL1)
                    (LLL1) edge (LLLR1)
                    (LLL1) edge (LLLL1);
                \node[labelright={R1}] {$B_x^{m+1}$};
                \node[labelright={LR1}] {$\overline{1}$};
                \node[labelright={LLR1}] {$1$};
                \node[labelright={LLLR1}] {$0$};
                \node[labelright={LLLL1}] {$\overline{A^n}$};
            \end{tikzpicture}%
            &
            \begin{tikzpicture}[baseline=(m1)]
                \node[tnode] (m1) {$\arrb{-}$};
            \end{tikzpicture}%
            &
            \begin{tikzpicture}[scale=0.7, baseline=(T1), transform shape]
                \node[tree_small={T1}] (tree1)
                {
                    \node[phantom] (P1) {}; & \node (T1) {}; & \node (R1) {}; \\
                    & \node[empty] (L1) {}; & \node (LR1) {}; \\
                    & \node (LL1) {}; & \node (LLR1) {}; \\
                    & \node (LLL1) {}; & \node (LLLR1) {}; \\
                    & \node (LLLL1) {}; & \node (LLLLR1) {}; \\
                    & \node (LLLLL1) {}; \\
                };
                \draw 
                    (P1) edge[dottededge] (T1)
                    (T1) edge (R1)
                    (T1) edge (L1)
                    (L1) edge (LR1)
                    (L1) edge (LL1)
                    (LL1) edge (LLR1)
                    (LL1) edge (LLL1)
                    (LLL1) edge (LLLR1)
                    (LLL1) edge (LLLL1)
                    (LLLL1) edge (LLLLR1)
                    (LLLL1) edge (LLLLL1);
                \node[labelright={R1}] {$B^m$};
                \node[labelright={LR1}] {$1$};
                \node[labelright={LLR1}] {$0$};
                \node[labelright={LLLR1}] {$\overline{1}$};
                \node[labelright={LLLLR1}] {$0$};
                \node[labelright={LLLLL1}] {$A^n$};
            \end{tikzpicture}%
            &
            \begin{tikzpicture}[baseline=(m1)]
                \coordinate (m1) at (0,.1);
                \node[tnode] at (0,0) {$\xrightarrow[\times 3]{\beta_0}$};
            \end{tikzpicture}\\
            &
            \begin{tikzpicture}[baseline=(m1)]
                \coordinate (m1) at (0,.1);
                \node[tnode] at (0,0) {$\xrightarrow[\times 3]{\beta_0}$};
            \end{tikzpicture}%
            &
            \begin{tikzpicture}[scale=0.7, baseline=(T1), transform shape]
                \node[tree_small={T1}] (tree1)
                {
                    \node[phantom] (P1) {}; & \node (T1) {}; & \node (R1) {}; \\
                    & \node (L1) {}; \\
                };
                \draw 
                    (P1) edge[dottededge] (T1)
                    (T1) edge (R1)
                    (T1) edge (L1);
                \node[labelright={R1}] {$B^m$};
                \node[labelright={L1}] {$A^n$};
            \end{tikzpicture}
        \end{tabular}
    \end{flushleft}

\medskip
\noindent {\bf Case 3.} $A$ is a left child of a left-child ($B$) of a right-child/tree-top ($C$). There are five distinct subcases:
    \begin{enumerate}[label=3.{\arabic*})]
        \item $B \not\in \{0,-1,-2,\cdots\}$;
        \item $B=\overline{m}; \; m\geq 3$;
        \begin{enumerate}[label=3.2.{\arabic*})]
            \item $A = Xn$; $n \in \{1,2,\cdots \}$; $X\not\in \rational$;
            \item $A = X\overline{n}$; $n \in \{0,1,2,\cdots \}$; $X\not\in \rational$;
            \item $A = X(YZ)$; $YZ\not\in \integral$;
            \item $A = X0n$; $n \in \integral$; $n \not\in \integral$;
        \end{enumerate}
        \item $B=\overline{2}$;
        \item $B=\overline{1}$;
        \item $B=0$.
    \end{enumerate}
    We will prove subcases 3.1 and 3.2. Subcase 3.3 is analogous to 3.2, and subcases 3.4, 3.5 are simpler. In subcase 3.2.4, $\beta_0$ move is performed 0 from $A = X0n$, and we obtain one of other 3.2 cases.
    
    \noindent {\bf Case 3.1.} $B \not\in \{0,-1,-2,\cdots\}$:
    \begin{flushleft}

    \end{flushleft}

    \noindent {\bf Case 3.2.} $B=\overline{m}; \; m\geq 3$. In this case, the results do not match, and we need to consider further specific cases of $A$ subtangle (Cases 3.2.1-3).
    \begin{flushleft}
        \begin{tabular}{cccccccc}
            \begin{tikzpicture}[scale=0.7, baseline=(T1), transform shape]
                \node[tree_small={T1}] (tree1)
                {
                    \node[phantom] (P1) {}; & \node (T1) {}; & \node (R1) {}; \\
                    & \node (L1) {}; & \node (LR1) {}; \\
                    & \node (LL1) {}; & \node (LLR1) {}; \\
                    & \node (LLL1) {}; \\
                };
                \draw 
                    (P1) edge[dottededge] (T1)
                    (T1) edge (R1)
                    (T1) edge (L1)
                    (L1) edge (LR1)
                    (L1) edge (LL1)
                    (LL1) edge (LLR1)
                    (LL1) edge (LLL1);
                \node[labelright={R1}] {$C_x^{k-1}$};
                \node[labelright={LR1}] {$1$};
                \node[labelright={LLR1}] {$m-1$};
                \node[labelright={LLL1}] {$\overline{A^n}$};
            \end{tikzpicture}%
            &
            \begin{tikzpicture}[baseline=(m1)]
                \node[tnode] (m1) {$\xleftarrow{\beta_-}$};
            \end{tikzpicture}%
            &
            \begin{tikzpicture}[scale=0.7, baseline=(T1), transform shape]
                \node[tree_small={T1}] (tree1)
                {
                    \node[phantom] (P1) {}; & \node (T1) {}; & \node (R1) {}; \\
                    & \node[empty] (L1) {}; & \node (LR1) {}; \\
                    & \node (LL1) {}; \\
                };
                \draw 
                    (P1) edge[dottededge] (T1)
                    (T1) edge (R1)
                    (T1) edge (L1)
                    (L1) edge (LR1)
                    (L1) edge (LL1);
                \node[labelright={R1}] {$C^k$};
                \node[labelright={LR1}] {$\overline{m}$};
                \node[labelright={LL1}] {$A^n$};
            \end{tikzpicture}%
            &
            \begin{tikzpicture}[baseline=(m1)]
                \node[tnode] (m1) {$\flype$};
            \end{tikzpicture}%
            &
            \begin{tikzpicture}[scale=0.7, baseline=(T1), transform shape]
                \node[tree_small={T1}] (tree1)
                {
                    \node[phantom] (P1) {}; & \node (T1) {}; & \node (R1) {}; \\
                    & \node[empty] (L1) {}; & \node (LR1) {}; \\
                    & \node (LL1) {}; & \node (LLR1) {}; \\
                    & \node (LLL1) {}; & \node (LLLR1) {}; \\
                    & \node (LLLL1) {}; & \node (LLLLR1) {}; \\
                    & \node (LLLLL1) {}; & \node (LLLLLR1) {}; \\
                    & \node (LLLLLL1) {}; & \node (LLLLLLR1) {}; \\
                    & \node (LLLLLLL1) {}; \\
                };
                \draw 
                    (P1) edge[dottededge] (T1)
                    (T1) edge (R1)
                    (T1) edge (L1)
                    (L1) edge (LR1)
                    (L1) edge (LL1)
                    (LL1) edge (LLR1)
                    (LL1) edge (LLL1)
                    (LLL1) edge (LLLR1)
                    (LLL1) edge (LLLL1)
                    (LLLL1) edge (LLLLR1)
                    (LLLL1) edge (LLLLL1)
                    (LLLLL1) edge (LLLLLR1)
                    (LLLLL1) edge (LLLLLL1)
                    (LLLLLL1) edge (LLLLLLR1)
                    (LLLLLL1) edge (LLLLLLL1);
                \node[labelright={R1}] {$C^k$};
                \node[labelright={LR1}] {$\overline{m}$};
                \node[labelright={LLR1}] {$0$};
                \node[labelright={LLLR1}] {$1$};
                \node[labelright={LLLLR1}] {$\overline{1}$};
                \node[labelright={LLLLLR1}] {$1$};
                \node[labelright={LLLLLLR1}] {$0$};
                \node[labelright={LLLLLLL1}] {$\overline{A^n}$};
            \end{tikzpicture}%
            &
            \begin{tikzpicture}[baseline=(m1)]
                \node[tnode] (m1) {$\arrb{0}$};
            \end{tikzpicture}%
            &
            \begin{tikzpicture}[scale=0.7, baseline=(T1), transform shape]
                \node[tree_small={T1}] (tree1)
                {
                    \node[phantom] (P1) {}; & \node (T1) {}; & \node (R1) {}; \\
                    & \node[empty] (L1) {}; & \node (LR1) {}; \\
                    & \node (LL1) {}; & \node (LLR1) {}; \\
                    & \node (LLL1) {}; & \node (LLLR1) {}; \\
                    & \node (LLLL1) {}; & \node (LLLLR1) {}; \\
                    & \node (LLLLL1) {}; \\
                };
                \draw 
                    (P1) edge[dottededge] (T1)
                    (T1) edge (R1)
                    (T1) edge (L1)
                    (L1) edge (LR1)
                    (L1) edge (LL1)
                    (LL1) edge (LLR1)
                    (LL1) edge (LLL1)
                    (LLL1) edge (LLLR1)
                    (LLL1) edge (LLLL1)
                    (LLLL1) edge (LLLLR1)
                    (LLLL1) edge (LLLLL1);
                \node[labelright={R1}] {$C^k$};
                \node[labelright={LR1}] {$\overline{m-1}$};
                \node[labelright={LLR1}] {$\overline{1}$};
                \node[labelright={LLLR1}] {$1$};
                \node[labelright={LLLLR1}] {$0$};
                \node[labelright={LLLLL1}] {$\overline{A^n}$};
            \end{tikzpicture}%
            &
            \begin{tikzpicture}[baseline=(m1)]
                \node[tnode] (m1) {$\arrb{-}$};
            \end{tikzpicture}\\
\end{tabular}
    \end{flushleft}

    \noindent {\bf Case 3.2.1.} 
    {Let the right-child of $A$ be a positive integral tangle $A^n=Xn$.}
    We show that algorithm acting on a non-flyped result of case 3.2 gives the same result as acting on the flyped result.
    \begin{flushleft}
        \begin{tabular}{ccccccc}
            \begin{tikzpicture}[scale=0.7, baseline=(A), transform shape]
                \node[tree_small={A}] (tree1)
                {
                    \node[phantom] (A) {}; & \node (B) {}; & \node (C) {}; \\
                    & \node (T1) {}; & \node (R1) {}; \\
                    & \node (L1) {}; & \node (LR1) {}; \\
                    &\node (LL1) {}; \\
                };
                \draw 
                    (A) edge[dottededge] (B)
                    (B) edge (C)
                    (B) edge (T1)
                    (T1) edge (R1)
                    (T1) edge (L1)
                    (L1) edge (LR1)
                    (L1) edge (LL1);
                \node[labelright={C}] {$C_x^{k-1}$};
                \node[labelright={R1}] {$1$};
                \node[labelright={LR1}] {$m-1$};
                \node[labelright={LL1}] {$\overline{A^n}$};
            \end{tikzpicture}%
            &
            \begin{tikzpicture}[baseline=(m1)]
                \node[tnode] (m1) {$=$};
            \end{tikzpicture}%
            &
            \begin{tikzpicture}[scale=0.7, baseline=(A), transform shape]
                \node[tree_small={A}] (tree1)
                {
                    \node[phantom] (A) {}; & \node (B) {}; & \node (C) {}; \\
                    & \node (T1) {}; & \node (R1) {}; \\
                    & \node (L1) {}; & \node (LR1) {}; \\
                    & \node[empty] (LL1) {}; & \node (LLR1) {}; \\
                    & \node (LLL1) {}; \\
                };
                \draw 
                    (A) edge[dottededge] (B)
                    (B) edge (C)
                    (B) edge (T1)
                    (T1) edge (R1)
                    (T1) edge (L1)
                    (L1) edge (LR1)
                    (L1) edge (LL1)
                    (LL1) edge (LLR1)
                    (LL1) edge (LLL1);
                \node[labelright={C}] {$C_x^{k-1}$};
                \node[labelright={R1}] {$1$};
                \node[labelright={LR1}] {$m-1$};
                \node[labelright={LLR1}] {$\overline{n}$};
                \node[labelright={LLL1}] {$\overline{X}$};
            \end{tikzpicture}%
            &
            \begin{tikzpicture}[baseline=(m1)]
                \node[tnode] (m1) {$\arrb{-}$};
            \end{tikzpicture}%
            &
            \begin{tikzpicture}[scale=0.7, baseline=(A), transform shape]
                \node[tree_small={A}] (tree1)
                {
                    \node[phantom] (A) {}; & \node (B) {}; & \node (C) {}; \\
                    & \node (T1) {}; & \node (R1) {}; \\
                    & \node (L1) {}; & \node (LR1) {}; \\
                    & \node (LL1) {}; & \node (LLR1) {}; \\
                    & \node (LLL1) {}; & \node (LLLR1) {}; \\
                    & \node (LLLL1) {}; \\
                };
                \draw 
                    (A) edge[dottededge] (B)
                    (B) edge (C)
                    (B) edge (T1)
                    (T1) edge (R1)
                    (T1) edge (L1)
                    (L1) edge (LR1)
                    (L1) edge (LL1)
                    (LL1) edge (LLR1)
                    (LL1) edge (LLL1)
                    (LLL1) edge (LLLR1)
                    (LLL1) edge (LLLL1);
                \node[labelright={C}] {$C_x^{k-1}$};
                \node[labelright={R1}] {$1$};
                \node[labelright={LR1}] {$m-2$};
                \node[labelright={LLR1}] {$1$};
                \node[labelright={LLLR1}] {$n-1$};
                \node[labelright={LLLL1}] {$X$};
            \end{tikzpicture}%
            &
            \begin{tikzpicture}[baseline=(m1)]
                \node[tnode] (m1) {$=$};
            \end{tikzpicture}%
            &
            \begin{tikzpicture}[scale=0.7, baseline=(A), transform shape]
                \node[tree_small={A}] (tree1)
                {
                    \node[phantom] (A) {}; & \node (B) {}; & \node (C) {}; \\
                    & \node (T1) {}; & \node (R1) {}; \\
                    & \node (L1) {}; & \node (LR1) {}; \\
                    & \node (LL1) {}; & \node (LLR1) {}; \\
                    & \node (LLL1) {}; \\
                };
                \draw 
                    (A) edge[dottededge] (B)
                    (B) edge (C)
                    (B) edge (T1)
                    (T1) edge (R1)
                    (T1) edge (L1)
                    (L1) edge (LR1)
                    (L1) edge (LL1)
                    (LL1) edge (LLR1)
                    (LL1) edge (LLL1);
                \node[labelright={C}] {$C_x^{k-1}$};
                \node[labelright={R1}] {$1$};
                \node[labelright={LR1}] {$m-2$};
                \node[labelright={LLR1}] {$1$};
                \node[labelright={LLL1}] {$A^{n-1}$};
            \end{tikzpicture}\\
\end{tabular}
    \end{flushleft}

    \noindent {\bf Case 3.2.2.} 
    {Let the right-child of $A$ be a non-positive integral tangle $A^{\overline{n}}=X\overline{n}$.}
    We show that the algorithm acting on a flyped result of case 3.2 gives the same result as acting on the non-flyped result.
    \begin{flushleft}
        \begin{tabular}{ccccccccc}
            \begin{tikzpicture}[scale=0.7, baseline=(A), transform shape]
                \node[tree_small={A}] (tree1)
                {
                    \node[phantom] (A) {}; & \node (B) {}; & \node (C) {}; \\
                    & \node (T1) {}; & \node (R1) {}; \\
                    & \node (L1) {}; & \node (LR1) {}; \\
                    & \node (LL1) {}; & \node (LLR1) {}; \\
                    & \node (LLL1) {}; \\
                };
                \draw 
                    (A) edge[dottededge] (B)
                    (B) edge (C)
                    (B) edge (T1)
                    (T1) edge (R1)
                    (T1) edge (L1)
                    (L1) edge (LR1)
                    (L1) edge (LL1)
                    (LL1) edge (LLR1)
                    (LL1) edge (LLL1);
                \node[labelright={C}] {$C_x^{k-1}$};
                \node[labelright={R1}] {$1$};
                \node[labelright={LR1}] {$m-2$};
                \node[labelright={LLR1}] {$1$};
                \node[labelright={LLL1}] {$A^{n-1}$};
            \end{tikzpicture}
            &
            \begin{tikzpicture}[baseline=(m1)]
                \node[tnode] (m1) {$=$};
            \end{tikzpicture}%
            &
            \begin{tikzpicture}[scale=0.7, baseline=(A), transform shape]
                \node[tree_small={A}] (tree1)
                {
                    \node[phantom] (A) {}; & \node (B) {}; & \node (C) {}; \\
                    & \node (T1) {}; & \node (R1) {}; \\
                    & \node (L1) {}; & \node (LR1) {}; \\
                    & \node (LL1) {}; & \node (LLR1) {}; \\
                    & \node[empty] (LLL1) {}; & \node (LLLR1) {}; \\
                    & \node (LLLL1) {}; \\
                };
                \draw 
                    (A) edge[dottededge] (B)
                    (B) edge (C)
                    (B) edge (T1)
                    (T1) edge (R1)
                    (T1) edge (L1)
                    (L1) edge (LR1)
                    (L1) edge (LL1)
                    (LL1) edge (LLR1)
                    (LL1) edge (LLL1)
                    (LLL1) edge (LLLR1)
                    (LLL1) edge (LLLL1);
                \node[labelright={C}] {$C_x^{k-1}$};
                \node[labelright={R1}] {$1$};
                \node[labelright={LR1}] {$m-2$};
                \node[labelright={LLR1}] {$1$};
                \node[labelright={LLLR1}] {$\overline{n+1}$};
                \node[labelright={LLLL1}] {$X$};
            \end{tikzpicture}%
            &
            \begin{tikzpicture}[baseline=(m1)]
                \node[tnode] (m1) {$\arrb{-}$};
            \end{tikzpicture}%
            &
            \begin{tikzpicture}[scale=0.7, baseline=(A), transform shape]
                \node[tree_small={A}] (tree1)
                {
                    \node[phantom] (A) {}; & \node (B) {}; & \node (C) {}; \\
                    & \node (T1) {}; & \node (R1) {}; \\
                    & \node[empty] (L1) {}; & \node (LR1) {}; \\
                    & \node (LL1) {}; & \node (LLR1) {}; \\
                    & \node (LLL1) {}; & \node (LLLR1) {}; \\
                    & \node (LLLL1) {}; & \node (LLLLR1) {}; \\
                    & \node (LLLLL1) {}; \\
                };
                \draw 
                    (A) edge[dottededge] (B)
                    (B) edge (C)
                    (B) edge (T1)
                    (T1) edge (R1)
                    (T1) edge (L1)
                    (L1) edge (LR1)
                    (L1) edge (LL1)
                    (LL1) edge (LLR1)
                    (LL1) edge (LLL1)
                    (LLL1) edge (LLLR1)
                    (LLL1) edge (LLLL1)
                    (LLLL1) edge (LLLLR1)
                    (LLLL1) edge (LLLLL1);
                \node[labelright={C}] {$C_x^{k-1}$};
                \node[labelright={R1}] {$1$};
                \node[labelright={LR1}] {$m-2$};
                \node[labelright={LLR1}] {$0$};
                \node[labelright={LLLR1}] {$1$};
                \node[labelright={LLLLR1}] {$n$};
                \node[labelright={LLLLL1}] {$\overline{X}$};
            \end{tikzpicture}%
            &
            \begin{tikzpicture}[baseline=(m1)]
                \node[tnode] (m1) {$\arrb{0}$};
            \end{tikzpicture}%
            &
            \begin{tikzpicture}[scale=0.7, baseline=(A), transform shape]
                \node[tree_small={A}] (tree1)
                {
                    \node[phantom] (A) {}; & \node (B) {}; & \node (C) {}; \\
                    & \node (T1) {}; & \node (R1) {}; \\
                    & \node (L1) {}; & \node (LR1) {}; \\
                    & \node (LL1) {}; & \node (LLR1) {}; \\
                    & \node (LLL1) {}; \\
                };
                \draw 
                    (A) edge[dottededge] (B)
                    (B) edge (C)
                    (B) edge (T1)
                    (T1) edge (R1)
                    (T1) edge (L1)
                    (L1) edge (LR1)
                    (L1) edge (LL1)
                    (LL1) edge (LLR1)
                    (LL1) edge (LLL1);
                \node[labelright={C}] {$C_x^{k-1}$};
                \node[labelright={R1}] {$1$};
                \node[labelright={LR1}] {$m-1$};
                \node[labelright={LLR1}] {$n$};
                \node[labelright={LLL1}] {$\overline{X}$};
            \end{tikzpicture}%
            &
            \begin{tikzpicture}[baseline=(m1)]
                \node[tnode] (m1) {$=$};
            \end{tikzpicture}\\
\end{tabular}
    \end{flushleft}

    \noindent {\bf Case 3.2.3.} 
    {Let the right-child of $A$ be a non-rational tangle $A^{n}=X(YZ^n)$.}
    We show that the algorithm acting on the non-flyped result of case 3.2 gives the same result as acting on the flyped result.
    \begin{flushleft}
        \begin{tabular}{cccccccc}
            \begin{tikzpicture}[scale=0.7, baseline=(T1), transform shape]
                \node[tree_small={T1}] (tree1)
                {
                    \node[phantom] (A) {}; & \node (B) {}; & \node (C) {}; \\
                    & \node (T1) {}; & \node (R1) {}; \\
                    & \node (L1) {}; & \node (LR1) {}; \\
                    & \node (LL1) {}; & \node (LLR1) {}; \\
                    & \node (LLL1) {}; \\
                };
                \draw 
                    (A) edge[dottededge] (B)
                    (B) edge (C)
                    (B) edge (T1)
                    (T1) edge (R1)
                    (T1) edge (L1)
                    (L1) edge (LR1)
                    (L1) edge (LL1)
                    (LL1) edge (LLR1)
                    (LL1) edge (LLL1);
                \node[labelright={C}] {$C_x^{k-1}$};
                \node[labelright={R1}] {$1$};
                \node[labelright={LR1}] {$m-2$};
                \node[labelright={LLR1}] {$1$};
                \node[labelright={LLL1}] {$A^{n-1}$};
            \end{tikzpicture}
            &
            \begin{tikzpicture}[baseline=(m1)]
                \node[tnode] (m1) {$=$};
            \end{tikzpicture}%
            &
            \begin{tikzpicture}[scale=0.7, baseline=(T1), transform shape]
                \node[tree_small={T1}] (tree1)
                {
                    \node[phantom] (A) {}; & \node (B) {}; & \node (C) {}; \\
                    & \node (T1) {}; & \node (R1) {}; \\
                    & \node (L1) {}; & \node (LR1) {}; \\
                    & \node[empty] (LL1) {}; & \node (LLR1) {}; \\
                    & \node (LLL1) {}; & \node (LLLR1) {}; & \node (LLLRR1) {}; \\
                    & \node (LLLL1) {};& \node (LLLRL1) {}; \\
                };
                \draw 
                    (A) edge[dottededge] (B)
                    (B) edge (C)
                    (B) edge (T1)
                    (T1) edge (R1)
                    (T1) edge (L1)
                    (L1) edge (LR1)
                    (L1) edge (LL1)
                    (LL1) edge (LLR1)
                    (LL1) edge (LLL1)
                    (LLL1) edge (LLLR1)
                    (LLL1) edge (LLLL1)
                    (LLLR1) edge (LLLRR1)
                    (LLLR1) edge (LLLRL1);
                \node[labelright={C}] {$C_x^{k-1}$};
                \node[labelright={R1}] {$1$};
                \node[labelright={LR1}] {$m-2$};
                \node[labelright={LLR1}] {$1$};
                \node[labelright={LLLRR1}] {$Z^{n-1}$};
                \node[labelright={LLLRL1}] {$Y$};
                \node[labelright={LLLL1}] {$X$};
            \end{tikzpicture}%
            &
            \begin{tikzpicture}[baseline=(m1)]
                \node[tnode] (m1) {$\arrb{1}$};
            \end{tikzpicture}%
            &
            \begin{tikzpicture}[scale=0.7, baseline=(T1), transform shape]
                \node[tree_small={T1}] (tree1)
                {
                    \node[phantom] (A) {}; & \node (B) {}; & \node (C) {}; \\
                    & \node (T1) {}; & \node (R1) {}; \\
                    & \node (L1) {}; & \node (LR1) {}; \\
                    & \node (LL1) {}; & \node (LLR1) {};  & \node (LLRR1) {};\\
                    & \node (LLL1) {}; & \node (LLRL1) {}; \\
                };
                \draw 
                    (A) edge[dottededge] (B)
                    (B) edge (C)
                    (B) edge (T1)
                    (T1) edge (R1)
                    (T1) edge (L1)
                    (L1) edge (LR1)
                    (L1) edge (LL1)
                    (LL1) edge (LLR1)
                    (LL1) edge (LLL1)
                    (LLR1) edge (LLRR1)
                    (LLR1) edge (LLRL1);
                \node[labelright={C}] {$C_x^{k-1}$};
                \node[labelright={R1}] {$1$};
                \node[labelright={LR1}] {$m-1$};
                \node[labelright={LLRR1}] {$\overline{Z^n}$};
                \node[labelright={LLRL1}] {$\overline{Y}$};
                \node[labelright={LLL1}] {$\overline{X}$};
            \end{tikzpicture}%
            &
            \begin{tikzpicture}[baseline=(m1)]
                \node[tnode] (m1) {$=$};
            \end{tikzpicture}%
            &
            \begin{tikzpicture}[scale=0.7, baseline=(T1), transform shape]
                \node[tree_small={T1}] (tree1)
                {
                    \node[phantom] (A) {}; & \node (B) {}; & \node (C) {}; \\
                    & \node (T1) {}; & \node (R1) {}; \\
                    & \node (L1) {}; & \node (LR1) {}; \\
                    & \node (LL1) {}; \\
                };
                \draw 
                    (A) edge[dottededge] (B)
                    (B) edge (C)
                    (B) edge (T1)
                    (T1) edge (R1)
                    (T1) edge (L1)
                    (L1) edge (LR1)
                    (L1) edge (LL1);
                \node[labelright={C}] {$C_x^{k-1}$};
                \node[labelright={R1}] {$1$};
                \node[labelright={LR1}] {$m-1$};
                \node[labelright={LL1}] {$\overline{A^n}$};
            \end{tikzpicture}\\
\end{tabular}
    \end{flushleft}

\medskip
\noindent {\bf Case 4.} 
{
In this case the algorithm already reaches a $C^k$ subtangle, meaning that both $C$ and $B$ are not the 0 tangle, and $C$ cannot be a negative integral tangle. 
In most cases the output of algorithm is the same as in case 3), with the exception  when $C=1$ and $B=\overline{m}$, $m>0$, since $C^k$ will become $C_x^{k-1}=0$ after the $\beta_-$ move will act on $B$. In this case, the algorithm  performs a $\beta_0$ move on $C$'s parent ($D$).}

    \begin{flushleft}
        %
    \end{flushleft}
If we flype subtangle $A$, there are 3 subcases to consider:
    \begin{enumerate}[label=4.{\arabic*})]
        \item $m=1$;
        \item $m=2$;
        \item $m\geq3$;
    \end{enumerate}
{We do not evaluate subcase 4.2, as it is analogous to 4.3.}

    \noindent {\bf Case 4.1.} $C=1$, $B=\overline{1}$.
    \begin{flushleft}
        %
    \end{flushleft}
    {}
Results of subcases 4.2 and 4.3 after the flype differ from the result of case 4 (without flyping), but in the same way as in the subcases 3.2. and 3.3, which were shown to finally give the same result regardless if the flype was performed or not.
\end{proof}

\begin{lemma}\label{thm:canonical_algorithm}
{
\autoref{alg:treefix} transforms an algebraic tangle into its canonical representation while preserving its isotopy class.}
\end{lemma}
\begin{proof}
    All restricted subtangle states (\autoref{fig:restricted}) are handeled by corresponding algorithm moves (\autoref{tab:traversal_moves}) which correct the subtree from the restricted state. The algorithm terminates when the tangle reaches its canonical form. {It is finite: the first and third parts both have only one iteration; the second part of algorithm can require multiple traversals, however, with (nearly) every iteration the traversal goes further, the only exception (case 4 in \autoref{thm:algorithm_invariant_flype}) can occur only a finite number of times, since the tree is finite. Therefore, the algorithm terminates.}
\end{proof}

\noindent {We are now ready to prove \autoref{thm:main}.}

{
\begin{proof}[Proof of the \autoref{thm:main}]
    \autoref{thm:canonical_algorithm} shows \autoref{alg:treefix} transforms a prime tangle into its canonical form. \autoref{thm:algorithm_invariant_other} and \autoref{thm:algorithm_invariant_flype} state that the algorithm is invariant to twists, flypes, ring and elementary moves .
\end{proof}
}


\section{Classification of prime tangles}\label{sec:classification}

{In addition to classifying tangles up to isotopy and equivalence, we also classify them by their orbit class.}

\begin{definition}[Tangle orbit]
    The \emph{orbit of a tangle} is a set of tangles generated by all possible compositions of $\mu$, $\nu$, $\eta$ acting on the tangle.
\end{definition}
\noindent Note that $\mu$, $\eta$, $\nu$ are generators of the $\mathbf{D}_8\!\times\!\mathbf{Z}_2$ group. This topic is extend in \autoref{sec:symmetries}.

The classifications of tangles up to 14 crossings and up to isotopy, equivalence, and orbit classes,
is presented in \autoref{tab:tangles_counted}. The classification of tangles up to 10 crossings up to the orbit classes (and their diagrams) can be found in the supplementary material \cite{supplement}.

\subsection{Generating tangles and the classification methods} 

{
In this subsection we describe how we generate tangles and what methods we used to classify them (generate the tangle tables).
}

\begin{table}[ht]
    \caption{Number of all distinct tangle orbits, divided into subgroups by number of crossings ("\#cross") and number of extra closed components ("\#closed components"). Tangles in each column/row are added up ("Orbit" column/row). Additionally tangles are counted and added up to equivalence ("Equiv.") and isotopy ("Isotopy").}
    \label{tab:tangles_counted}
    \begin{tabular}{rr@{\enspace}r@{\enspace}r@{\enspace}r@{\enspace}r@{\enspace}r@{\enspace}rr@{\enspace}r@{\enspace}r}
        \toprule
        \multirow{2}{*}{\#cross} & \multicolumn{7}{c}{\#closed components} & \multicolumn{3}{c}{Total} \\
        \cmidrule(rl){2-8} \cmidrule(rl){9-11}
        & 0 & 1 & 2 & 3 & 4 & 5 & 6 & Orbit & Equiv. & Isotopy \\
        \midrule
        0 & 1 & - & - & - & - & - & - & 1 & 1 & 2 \\
        1 & 1 & - & - & - & - & - & - & 1 & 1 & 2 \\
        2 & 1 & - & - & - & - & - & - & 1 & 2 & 4 \\
        3 & 2 & - & - & - & - & - & - & 2 & 4 & 8 \\
        4 & 4 & 2 & - & - & - & - & - & 6 & 11 & 22 \\
        5 & 11 & 3 & - & - & - & - & - & 14 & 28 & 68 \\
        6 & 30 & 12 & 2 & - & - & - & - & 44 & 87 & 236 \\
        7 & 86 & 46 & 4 & - & - & - & - & 136 & 270 & 880 \\
        8 & 267 & 152 & 37 & 3 & - & - & - & 459 & 912 & 3442 \\
        9 & 844 & 608 & 129 & 4 & - & - & - & 1585 & 3168 & 13900 \\
        10 & 2910 & 2202 & 611 & 81 & 3 & - & - & 5807 & 11595 & 57488 \\
        11 & 10102 & 8742 & 2803 & 286 & 5 & - & - & 21938 & 43864 & 242206 \\
        12 & 37115 & 34853 & 11913 & 2021 & 178 & 4 & - & 86084 & 172091 & 1035696 \\
        13 & 139778 & 141102 & 56073 & 9831 & 573 & 5 & - & 347362 & 694690 & 4482956 \\
        14 & 539872 & 592627 & 249371 & 50979 & 5961 & 328 & 4 & 1439142 & 2877981 & 19602964 \\
        \midrule
        Orbit & 731024 & 780349 & 320943 & 63205 & 6720 & 337 & 4 & 1902582 & & \\
        Equiv. & 1461922 & 1560597 & 641753 & 126347 & 13408 & 670 & 8 & & 3804705 & \\
        Isotopy & 10192496 & 10480208 & 3994874 & 705590 & 64048 & 2642 & 16 & & & 25439874 \\
        \bottomrule
    \end{tabular}
\end{table}

\paragraph{Generation of alternating tangles}
{
The initial set of all alternating algebraic tangles up to $N$ crossings is generated using a standard integer partition algorithm, that is, generate sequences $(p_i \in \mathbb{N}_0)_i$, satisfying $\sum_{i}p_i = N$, without consecutive zeroes, and nest the partitions in all possible ways.
}
Using integer partitions, our goal is to generate all tangles up to $N$ crossings and up to $\mu$, $\eta$ reflections. These reflected tangles can be obtained by changing signs and/or by adding $0$ at the end of existing tangles (see \eqref{eq:refl}). Such generated non-negative partitions represent alternating tangles in multiplication form. Not all non-negative partitions are needed, since many representations are redundant:
\begin{itemize}
    \item zeroes can appear only at the end of the bracket -- zeroes at the beginning of bracket generate non-prime tangles, which we omit in the classification, and zeroes inside the bracket are redundant: $n0m = (n\!+\!m)$.
    \item consecutive opening brackets are redundant (left associativity of multiplication): $((A \cdots B) C \cdots) = (A \cdots B C \cdots)$
    \item first integral tangle after the opening bracket is ambiguous: $(1 \cdots) = (1 0 \cdots) \implies (1 n \cdots) = \left((n\!+\!1) \cdots\right)$.
\end{itemize}
Next, we remove tangles that end with 0. At this stage, 1,527,810 tangles up to 14 crossings are obtained.

\FloatBarrier

\paragraph{Generation of non-alternating tangles}
Tangles containing $0$ at the end of any bracket are used as a template for generation of non-alternating tangles. The $0$'s inside tangles are then replaced with negative integers $-1, \ldots, -k$, where $k$ is the level of nestedness, defined as the number of closing brackets to the right of the $0$.
At this stage, 25,267,083 tangles are obtained.

\paragraph{Determination of tangles' minimal set and symmetries}
The main algorithm translates each tangle into a binary tree. We perform the following set of operations on each tangle in a binary tree form: 
\begin{itemize}
    \item generation of the remaining 15 tangles in each the tangle's orbit, 
    \item transforming all 16 tangles into their canonical forms, and determination of a tangle symmetry group (by equality of canonical forms),
    \item all 16 tangles are minimized to determine the number of crossings in their minimal representations, {here the minimization only of one of the orbits' representative is sufficient}.
\end{itemize}
\medskip

If a tangle is found in the orbit of any other tangle, then it is rejected.
{In each orbit, we choose the lexicographically minimal tangle as the representative.}
Finally, we obtain 1,902,582 unique tangle representatives and their symmetry groups, which translates into 3,804,705 unique tangles up to equivalence, and 25,439,874 unique tangles up to isotopy.

\paragraph{Minimization algorithm}
{In order to find the tangle form with a minimal number of crossings, we first modify the tangle in its canonical form by twisting the negative rightmost tangles, to incorporate them into rational tangles (this way the number of crossings of each rational tangle does not change, only their sign, but the number of crossings of the rightmost integral tangle is lowered). 
Next, we search for possible flypes that could lower the number of crossings (between two right leaves with opposite signs, which are neighbors or separated by one $0$ right-leaf) and perform these flypes on separate copies of the tangle. Then we recursively repeat this process until no possibilities are left.
}

\begin{figure}[ht]
    \centering
    \includegraphics[scale=0.5,page=1]{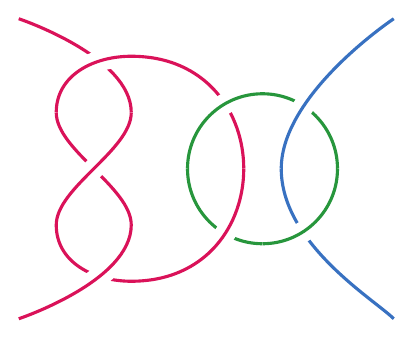} \qquad \qquad
    \includegraphics[scale=0.5,page=2]{tanglesx.pdf}
    \caption{Example of mutant tangles: left $3(2(20))$, right: $2(3(20))$.}
    \label{fig:mutants}
\end{figure}

\subsection{Distinguishing mutant tangles}
{
Identifying and distinguishing mutant knots  presents significant challenges in knot theory, as they share several of the same invariants (including the hyperbolic volume and the HOMFLYPT polynomial), similarly, this difficulty extends to tangles. Notable examples are tangles with 7-crossings, which are missing in the table of algebraic tangles in \cite{moriuchi2008enumeration}, where mutants are present at the very end of the classification. The methods used in \cite{moriuchi2008enumeration} could not distinguish these 4 pairs, since after closures (numerator, denominator, double) they form the same links.

In \autoref{tab:mutants}, we list pairs of mutants, and in \autoref{fig:mutants} we show diagrams of one of these pairs. Moreover, missing tangles have higher symmetry and belong to $z \rho$ group, when their mutants belong to one of $\rho$ groups (see next section).
}

\begin{table}[t]
    \caption{4 tangles missing from classification \cite{moriuchi2008enumeration} are mutants of other classified tangles. Moriuchi used Conway's comma notation $(A,B,C) = A0\!+\!B0\!+\!C0 = A(B(C0))$.}\label{tab:mutants}
    \centering
    \begin{tabular}{cccc}
        \toprule
        \multicolumn{2}{c}{In Moriuchi's table} & \multicolumn{2}{c}{Missing mutant tangle}  \\
        \cmidrule(rl){1-2} \cmidrule(rl){3-4}
        our notation & Conway's notation & our notation & Conway's notation  \\
        \midrule
        $3(2(20))$ & $(3,2,2)$ & $2(3(20))$ & $(2,3,2)$\\
        $3(2(\overline{2}0))$ & $(3,2,\overline{2})$ & $2(3(\overline{2}0))$ & $(2,3,\overline{2})$ \\
        $3(\overline{2}(\overline{2}0))$ & $(3,\overline{2},\overline{2})$ & $\overline{2}(3(\overline{2}0))$ & $(\overline{2}, 3,\overline{2})$ \\
        $21(2(20))$ & $(21,2,2)$ & $2(21(20))$ & $(2,21,2)$ \\
         \bottomrule
    \end{tabular}
\end{table}


\section{Tangle symmetry groups}\label{sec:symmetries}
Operators $\nu, \eta, \mu$ from \autoref{def:transformations} are generators of $\mathbf{D}_8\!\times\!\mathbf{Z}_2$ group:
\begin{align}
    \mathbf{D}_8\!\times\!\mathbf{Z}_2 = \langle \mu, \nu, \eta \mid \mu^2\!=\!\nu^4\!=\!\eta^2\!=\!e,  [\nu,\eta]\!=\!\nu^2, [\nu,\mu]\!=\![\eta,\mu]\!=\!0 \rangle
\end{align}
which is the symmetry group of the Conway sphere. \autoref{fig:D8_Z2_representation} presents tangles related by $\nu, \eta, \mu$ operations (orbits), which in general generate the $\mathbf{D}_8\!\times\!\mathbf{Z}_2$ group. Furthermore, we study symmetry groups of tangles in terms of subgroups of $\mathbf{D}_8\!\times\!\mathbf{Z}_2$.
\autoref{fig:subgroups} presents the lattice diagram of subgroups of $\mathbf{D}_8\!\times\!\mathbf{Z}_2$.

\begin{figure}[ht!]
    \centering
    \begin{tikzpicture}[scale=1.4, transform shape]
        \node[box, minimum size=1cm] at (0,0) {\rotatebox{180}{R}};
        \node[mirrorbox, minimum size=1cm] at (1,0) {\reflectbox{\rotatebox{180}{\textbf{R}}}};
        \node[mirrorbox, minimum size=1cm] at (0,-1) {\reflectbox{\rotatebox{0}{\textbf{R}}}};
        \node[box, minimum size=1cm] at (1,-1) {R};
        
        \node[box, minimum size=1cm] at (0,-3) {{\rotatebox{-90}{{R}}}};
        \node[mirrorbox, minimum size=1cm] at (1,-3) {\reflectbox{\rotatebox{-90}{\textbf{R}}}};
        \node[mirrorbox, minimum size=1cm] at (0,-4) {\reflectbox{\rotatebox{90}{\textbf{R}}}};
        \node[box, minimum size=1cm] at (1,-4) {{\rotatebox{90}{{R}}}};

        \node[mirrorbox, minimum size=1cm] at (3,0) {{\rotatebox{90}{\textbf{R}}}};
        \node[box, minimum size=1cm] at (4,0) {\reflectbox{\rotatebox{90}{{R}}}};
        \node[box, minimum size=1cm] at (3,-1) {\reflectbox{\rotatebox{-90}{{R}}}};
        \node[mirrorbox, minimum size=1cm] at (4,-1) {{\rotatebox{-90}{\textbf{R}}}};

        \node[mirrorbox, minimum size=1cm] at (3,-3) {{\textbf{R}}};
        \node[box, minimum size=1cm] at (4,-3) {\reflectbox{\rotatebox{0}{{R}}}};
        \node[box, minimum size=1cm] at (3,-4) {\reflectbox{\rotatebox{180}{{R}}}};
        \node[mirrorbox, minimum size=1cm] at (4,-4) {{\rotatebox{180}{\textbf{R}}}};

        \node at (2,-1.8) {\footnotesize $\mu$};
        \node at (2,-.87) {\footnotesize $\eta$};
        \node at (2,-3.15) {\footnotesize $\eta$};
        \node at (1.2,-2) {\footnotesize $\mu \eta$};
        \node at (2.8,-2) {\footnotesize $\mu \eta$};
        \node at (.5,.72) {\footnotesize $\rho_y$};
        \node at (3.5,.72) {\footnotesize $\rho_y$};
        \node at (.5,-4.72) {\footnotesize $\rho_y$};
        \node at (3.5,-4.72) {\footnotesize $\rho_y$};
        \node at (-.76,-.5) {\footnotesize $\rho_x$};
        \node at (-.76,-3.5) {\footnotesize $\rho_x$};
        \node at (4.78,-.5) {\footnotesize $\rho_x$};
        \node at (4.78,-3.5) {\footnotesize $\rho_x$};
        \draw[{Latex}-{Latex}] (1.5,-1.5) -- (2.5,-2.5);
        \draw[{Latex}-{Latex}] (1.5,-2.5) -- (2.5,-1.5);
        \draw[{Latex}-{Latex}] (1.5,-1) -- (2.5,-1);
        \draw[{Latex}-{Latex}] (1.5,-3) -- (2.5,-3);
        \draw[{Latex}-{Latex}] (1,-1.5) -- (1,-2.5);
        \draw[{Latex}-{Latex}] (3,-1.5) -- (3,-2.5);
        \draw[{Latex}-{Latex}] (.2,.6) -- (.8,.6);
        \draw[{Latex}-{Latex}] (.2,-4.6) -- (.8,-4.6);
        \draw[{Latex}-{Latex}] (3.2,.6) -- (3.8,.6);
        \draw[{Latex}-{Latex}] (3.2,-4.6) -- (3.8,-4.6);
        \draw[{Latex}-{Latex}] (-.6,-.8) -- (-.6,-.2);
        \draw[{Latex}-{Latex}] (-.6,-3.8) -- (-.6,-3.2);
        \draw[{Latex}-{Latex}] (4.6,-.8) -- (4.6,-.2);
        \draw[{Latex}-{Latex}] (4.6,-3.8) -- (4.6,-3.2);
        \node[anchor=north] at (2,2) {\small Orbit};
        \draw (-2.2,2) -- (6.2,2) -- (6.2,-5.6) -- (-2.2,-5.6) -- cycle;
        \node[anchor=west] at (-2,-2) {\rotatebox{90}{\footnotesize Equivalence}};
        \node[anchor=east] at (6,-2) {\rotatebox{-90}{\footnotesize Equivalence}};
        \draw[dashed] (-2,1.4) -- (1.7,1.4) -- (1.7,-5.4) -- (-2,-5.4) -- cycle;
        \draw[dashed] (2.3,1.4) -- (6,1.4) -- (6,-5.4) -- (2.3,-5.4) -- cycle;
        \node[anchor=north west] at (-1.2,1.2) {$\frac{p}{q}$};
        \draw[dotted] (1.5,1.2) -- (-1.2,1.2) -- (-1.2,-1.5) -- (1.5,-1.5) -- cycle;
        \node[anchor=south west] at (-1.2,-5.2) {$-\frac{q}{p}$};
        \draw[dotted] (1.5,-5.2) -- (-1.2,-5.2) -- (-1.2,-2.5) -- (1.5,-2.5) -- cycle;
        \node[anchor=north east] at (5.2,1.2) {$\frac{q}{p}$};
        \draw[dotted] (2.5,1.2) -- (5.2,1.2) -- (5.2,-1.5) -- (2.5,-1.5) -- cycle;
        \node[anchor=south east] at (5.2,-5.2) {$-\frac{p}{q}$};
        \draw[dotted] (2.5,-5.2) -- (5.2,-5.2) -- (5.2,-2.5) -- (2.5,-2.5) -- cycle;
        \draw[-{Latex}] (.2,-2.5) to [out=90 ,in=-90] (.6,-1.5);
        \draw[-{Latex}] (.2,-1.5) to [out=-90 ,in=90] (.6,-2.5);
        \node at (.2,-2) {\footnotesize $\nu$};
        \draw[-{Latex}] (3.8,-2.5) to [out=90 ,in=-90] (3.4,-1.5);
        \draw[-{Latex}] (3.8,-1.5) to [out=-90 ,in=90] (3.4,-2.5);
        \node at (3.8,-2) {\footnotesize $\nu$};
    \end{tikzpicture}
    \caption{{The orbit of a tangle $R$. The tangle $R$ and its transformations are represented by two-faced square tiles. The orbit of a tangle collects all tangles obtainable by $\mu$, $\nu$ and $\eta$ operations. Tangles in the orbits have up to 4 distinct fraction values, which are their own inverses and/or opposites: $\frac{p}{q}$, $\frac{q}{p}$, $-\frac{p}{q}$, $-\frac{q}{p}$. \newline
    Equivalent tangles are related by rotations, $\nu$ and $\mu\eta$, and their compositions, $\rho_x=\mu\eta\nu^3$, $\rho_y=\mu\eta\nu$, $\rho_z=\nu^2$. If a tangle has no mirror symmetries, then its orbit splits into two sets of equivalent tangles.}} \label{fig:D8_Z2_representation}
\end{figure}
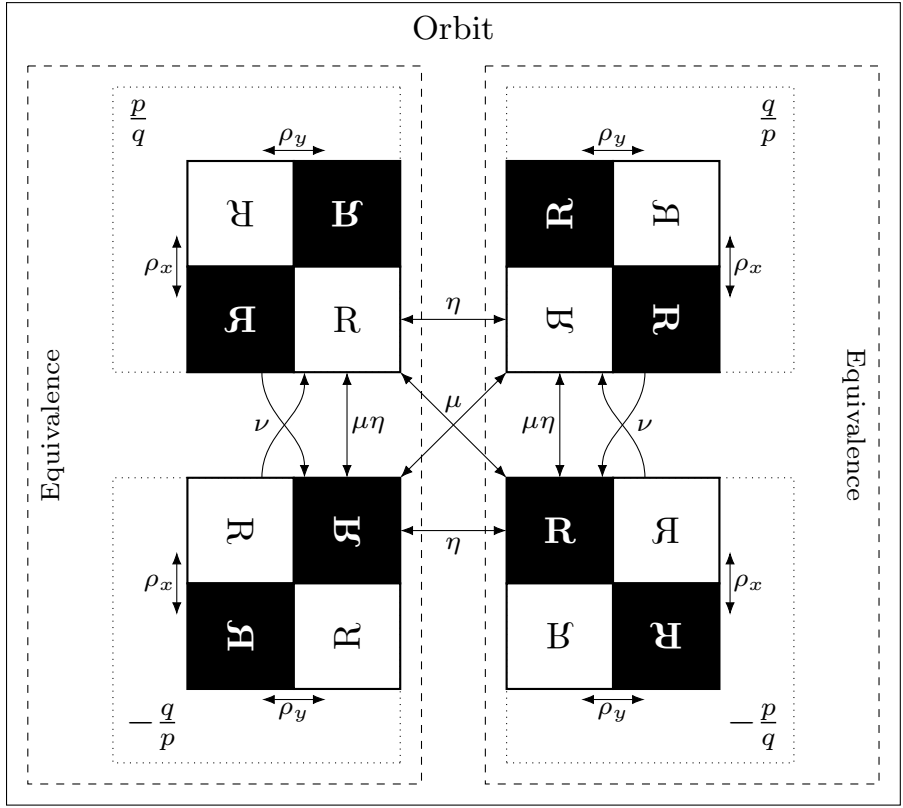

Some subgroups of $\mathbf{D}_8\!\times\!\mathbf{Z}_2$ cannot be the symmetry group of a tangle (e.g. $\mathbf{D}_8\!\times\!\mathbf{Z}_2$ itself), since no tangle can be invariant to $\nu$, $\nu^3$, $\mu\eta$, and $\mu\eta\nu^2$ operations, which is stated in \autoref{thm:symmetries}. All other symmetry groups have been observed and they are listed in \autoref{tab:symmetries} together with operations they are invariant to. 
The remaining \tables{tab:tangles_sym_rot} -- \ref{tab:tangles_sym_mir_0} count tangles up to 14 crossings split into subgroups based on the numbers of crossings and tangle symmetry groups. \tables{tab:tangles_sym_rot_0} and \ref{tab:tangles_sym_mir_0} count only tangles with no additional closed components.

\autoref{tab:tangles_sym_mir} shows that tangles with mirror symmetries are very rare. Note that, while in general tangles generated by $\nu$, $\mu$ and $\eta$ create two groups of equivalent tangles, if any mirror symmetry is present, all tangles generated by $\nu$, $\mu$ and $\eta$, are equivalent.
{Interestingly, mirror symmetry groups are split between X-tangles and H/V-tangles (\tables{tab:symmetries}, \ref{tab:tangles_sym_mir}, and \ref{tab:tangles_sym_mir_0}).} Moreover, presence of mirror symmetries for a specific number of crossings (in minimal representation) is restricted to either X-tangles or H/V-tangle (\autoref{tab:tangles_sym_mir}) -- at least for tangles up to 14 crossings. {It remains an open question if this is generally true for all algebraic tangles -- with any number of crossings.}

\begin{figure}[ht!]
    \centering
    \begin{tikzpicture}[scale=6]
        \node[node] at (0,0,0) (D8Z2) {$\mathbf{D}_8\!\times\!\mathbf{Z}_2$};
        \node[node] at (1,0,0) (D8) {$\mathbf{D}_8$};
        \node[node] at (0,0,1) (Z4Z2) {$\mathbf{Z}_4\!\times\!\mathbf{Z}_2$};
        \node[node] at (1,0,1) (Z4) {$\mathbf{Z}_4$};
        \node[node] at (0,-.6,0) (Z2Z2Z2) {$\mathbf{Z}_2\!\times\!\mathbf{Z}_2\!\times\!\mathbf{Z}_2$};
        \node[node] at (1,-.6,0) (Z2Z2r) {$\mathbf{Z}_2\!\times\!\mathbf{Z}_2$};
        \node[node] at (0,-.6,1) (Z2Z2l) {$\mathbf{Z}_2\!\times\!\mathbf{Z}_2$};
        \node[node] at (1,-.6,1) (Z2c) {$\mathbf{Z}_2$};
        \node[node] at (0,-1.2,0) (Z2Z2c) {$\mathbf{Z}_2\!\times\!\mathbf{Z}_2$};
        \node[node] at (1,-1.2,0) (Z2r) {$\mathbf{Z}_2$};
        \node[node] at (0,-1.2,1) (Z2l) {$\mathbf{Z}_2$};
        \node[node] at (1,-1.2,1) (Z1) {$\mathbf{Z}_1$};
        \begin{scope}
            \tikzstyle{gen} = [scale=.75,align=left,inner sep=0, outer sep=0, anchor=south]
            \node[gen, below right = .3 and .3 of D8Z2.center] {
                \st{$\langle \mu, \nu, \eta \rangle$}};
            \node[gen, below right = .3 and .3 of D8.center] {
                $\langle \eta, \mu\nu \rangle$\\
                \st{$\langle \eta, \nu \rangle$}\\
                \st{$\langle \mu\eta, \nu \rangle$}\\
                \st{$\langle \mu\eta, \mu\nu \rangle$}};
            \node[gen, below left = .3 and .3 of Z4Z2.center] {
                \st{$\langle \mu, \nu \rangle$}};
            \node[gen, below right = -.12 and .3 of Z4.center] {
                {$\langle \mu\nu \rangle$}\\
                \st{$\langle \nu \rangle$}}; 
            \node[gen, below right = .3 and .3 of Z2Z2Z2.center] {
                {$\langle \mu, \nu^2, \eta\nu \rangle$}\\
                \st{$\langle \mu, \nu^2, \eta \rangle$}};
            \node[gen, below right = .3 and .3 of Z2Z2r.center] {
                {$\langle \nu^2, \eta\nu \rangle$}\\
                {$\langle \nu^2, \eta \rangle$}\\
                {$\langle \nu^2, \mu\eta\nu \rangle$}\\
                \st{$\langle \nu^2, \mu\eta \rangle$}};
            \node[gen, below left = .3 and .3 of Z2Z2l.center] {
                $\langle \mu, \nu^2 \rangle$};
            \node[gen, below right = -.12 and .3 cm of Z2c.center] {
                $\langle \nu^2 \rangle$}; 
            \node[gen, below right = .3 and .3 of Z2Z2c.center] {
                {$\langle \mu, \eta\nu \rangle$}
                {$\langle \mu, \eta\nu^3 \rangle$}\\
                {$\langle \mu\nu^2, \eta\nu \rangle$}
                {$\langle \mu\nu^2, \eta\nu^3 \rangle$}\\
                \st{$\langle \mu, \eta \rangle$}
                \st{$\langle \mu, \eta\nu^2 \rangle$}\\
                \st{$\langle \mu\nu^2, \eta \rangle$}
                \st{$\langle \mu\nu^2, \eta\nu^2 \rangle$}\\};
            \node[gen, below right = -.1 and .3 of Z2r.center] {
                {$\langle \eta \rangle$}\\
                {$\langle \eta\nu^2 \rangle$}\\
                {$\langle \eta\nu^3 \rangle$}\\
                {$\langle \eta\nu \rangle$}\\
                {$\langle \mu\eta\nu^3 \rangle$}\\
                {$\langle \mu\eta\nu \rangle$}\\
                \st{$\langle \mu\eta \rangle$}\\
                \st{$\langle \mu\eta\nu^2 \rangle$}};
            \node[gen, below left = -.44 and .3 of Z2l.center] {
                $\langle \mu \rangle$\\
                $\langle \mu\nu^2 \rangle$};
            \node[gen, below right = -.12 and .3 cm of Z1.center] {
                $\langle e \rangle$};
        \end{scope}
        \draw (Z1) -- (Z2r);
        \draw (Z1) -- (Z2l);
        \draw (Z1) -- (Z2c);
        \draw (Z2c) -- (Z2Z2r);
        \draw (Z2c) -- (Z2Z2l);
        \draw (Z2c) -- (Z4);
        \draw (Z4) -- (Z4Z2);
        \draw (Z4) -- (D8);
        \draw (Z2r) -- (Z2Z2r);
        \draw (Z2r) -- (Z2Z2c);
        \draw (Z2Z2r) -- (Z2Z2Z2);
        \draw (Z2Z2r) -- (D8);
        \draw (D8) -- (D8Z2);
        \draw (Z2l) -- (Z2Z2c);
        \draw (Z2l) -- (Z2Z2l);
        \draw (Z2Z2l) -- (Z2Z2Z2);
        \draw (Z2Z2l) -- (Z4Z2);
        \draw (Z4Z2) -- (D8Z2);
        \draw (Z2Z2c) -- (Z2Z2Z2);
        \draw (Z2Z2Z2) -- (D8Z2);
    \end{tikzpicture}
    \caption{The lattice of subgroups of $\mathbf{D}_8\!\times\!\mathbf{Z}_2$. Next to each group there are listed possible sets of generators. Sets of generators which are \st{crossed out} cannot generate the tangle symmetry group, since they include at least one of $\nu, \nu^3, \mu\eta, \mu\eta\nu^2$ (\autoref{thm:symmetries}). All remaining generating sets create symmetry groups, which were all observed in the tangles.}
    \label{fig:subgroups}
\end{figure}
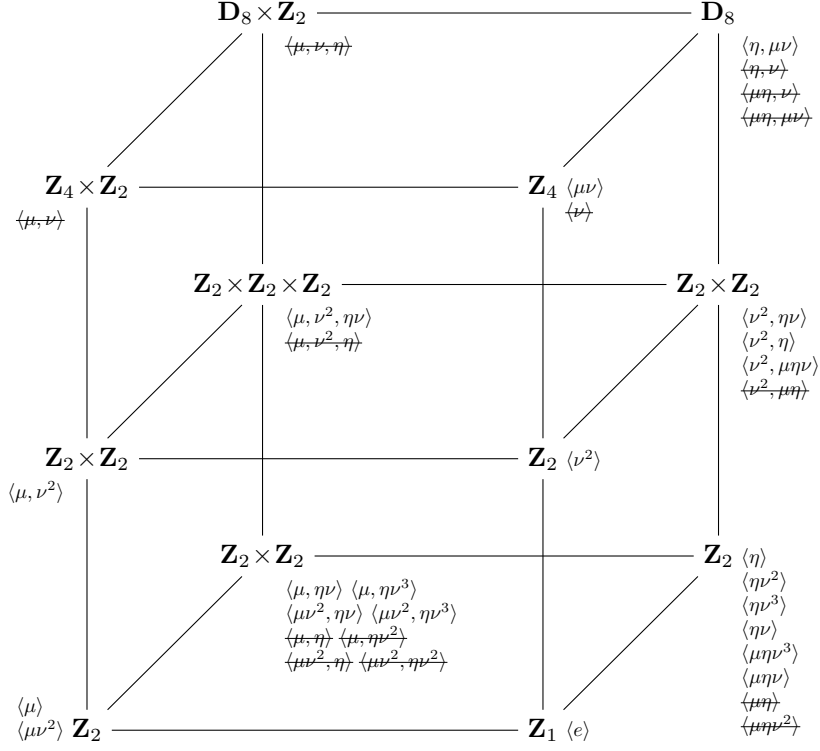

\newpage

\begin{theorem}\label{thm:symmetries}
    No tangle is isotopic to itself after $\nu$, $\nu^3$, $\mu\eta$ or $\mu\eta\nu^2$ transformations.
\end{theorem}
\begin{proof}
    The proof follows from the properties of the fraction (\autoref{def:fraction_algebraic} and \autoref{thm:fraction_rotations},\ref{thm:fraction_mirror}). Let us assume that the fraction of tangle $A$ is $w$. From $\rho_x\!\defeq\! \mu\nu\eta$, $\rho_z\!\defeq\!\nu^2$, $\eta R\!\defeq\!R0$, $\mu R \!\defeq\! \overline{R}$ we obtain:
    \begin{enumerate}
        \item[1)] $\mathrm{Frac}(\mu\nu\eta R) = 
            \mathrm{Frac}(\nu^2 R) = w$,
        \item[2)] $\mathrm{Frac}(\mu R) = -w$,
        \item[3)] $\mathrm{Frac}(\eta R) = 1/w$.
    \end{enumerate}
    1)+2)+3) imply that $\mathrm{Frac}(\nu R) = \mathrm{Frac}(\nu^3 R) = \mathrm{Frac}(\mu\eta R) = \mathrm{Frac}(\mu\eta\nu^2 R) = -1/w$.
    {
    Invariance to  $\nu, \nu^3, \mu\eta,$ or $\mu\eta\nu^2$ implies that $w = -1/w$, which is a contradiction in $\mathbb{Q} \cup \{ \infty \}$.}
\end{proof}

\FloatBarrier

\begin{table}[ht]
\caption{Possible symmetry groups of tangles. For each symmetry group, the invariant transformations are listed. Each group has 1, 2, 4 or 8 invariant operations (including $\text{Id}$, which is not listed). Invariance to $\mu\nu^{3}$ is not listed, since it is an inverse of $\mu\nu$, and other listed transformations are their own inverses. Some equivalence groups of tangles collect tangles from two symmetry groups -- such conjugated groups are marked with curly braces. {In the first column an abbreviation is assigned to each group (or to a pair of conjugated groups).}}\label{tab:symmetries}
\begin{tabular}{l@{$\,$}l@{$\,$}l@{$\quad$}l@{$\quad$}lc@{$\quad$}c@{$\quad$}c@{$\quad$}c@{$\quad$}c@{$\quad$}c@{$\quad$}c@{$\quad$}c@{$\quad$}c@{$\quad$}cc}
\toprule
\multicolumn{5}{c}{Symmetry group} & \multicolumn{10}{c}{Is invariant to?} & Type \\
\cmidrule(rl){1-5} \cmidrule(rl){6-15} \cmidrule(rl){16-16}
\multicolumn{3}{l}{Abb.} & Type & Generators & 
$\rho_z$ & $\rho_y$ & $\rho_x$ & $\mu\rho_z$ & $\mu\rho_y$ & $\mu\rho_x$ & $\mu$ & $\eta$ & $\eta\rho_z$ & $\mu\nu$ & VHX\\
\bottomrule
\toprule
$\mathfrak{1}$ & & & $\mathbf{D}_8$ & $\langle\eta,\mu\nu\rangle$ &
$\checkmark$ & $\checkmark$ & $\checkmark$ & & & & & $\checkmark$ & $\checkmark$ & $\checkmark$ & \phantom{VH}X\\
$\overline{\nu}$ & & & $\mathbf{Z}_4$ & $\langle\mu\nu\rangle$ & 
$\checkmark$ & & & & & & & & & $\checkmark$ & \phantom{VH}X\\
$\eta z$ & & & $\mathbf{Z}_2\!\times\!\mathbf{Z}_2$ & $\langle\eta,\nu^2\rangle$ & 
$\checkmark$ & & & & & & & $\checkmark$ & $\checkmark$ & & \phantom{VH}X\\
\multirow{2}{*}{$\eta$} & \multirow{2}{*}{$\biggl\{$} & $\diagdown$ & $\mathbf{Z}_2$ & $\langle\eta\rangle$ &
& & & & & & & $\checkmark$ & & & \phantom{VH}X\\
& & $\diagup$ & $\mathbf{Z}_2$ & $\langle\eta\nu^2\rangle$ &
& & & & & & & & $\checkmark$ & & \phantom{VH}X\\
\midrule

$\mathfrak{0}$ & & & $\mathbf{Z}_2\!\times\!\mathbf{Z}_2\!\times\!\mathbf{Z}_2$ & $\langle\mu,\nu^2,\eta\nu\rangle$ & 
$\checkmark$ & $\checkmark$ & $\checkmark$ & $\checkmark$ & $\checkmark$ & $\checkmark$ & $\checkmark$ & & & & VH\phantom{X}\\
$\mu z$ & & & $\mathbf{Z}_2\!\times\!\mathbf{Z}_2$ & $\langle\mu,\nu^2\rangle$ & $\checkmark$ & & & $\checkmark$ & & & $\checkmark$ & & & & VH\phantom{X}\\
\multirow{2}{*}{$\mu\rho$} & \multirow{2}{*}{$\biggl\{$} & $\mu y$ & $\mathbf{Z}_2\!\times\!\mathbf{Z}_2$ & $\langle\mu,\eta\nu\rangle$ & 
& $\checkmark$ & & & $\checkmark$ & & $\checkmark$ & & & & VH\phantom{X}\\
& & $\mu x$ & $\mathbf{Z}_2\!\times\!\mathbf{Z}_2$ & $\langle\mu,\eta\nu^3\rangle$ & 
& & $\checkmark$ & & & $\checkmark$ & $\checkmark$ & & & & VH\phantom{X}\\
$\overline{xy}$ & & & $\mathbf{Z}_2\!\times\!\mathbf{Z}_2$ & $\langle\eta\nu,\nu^2\rangle$ & 
$\checkmark$ & & & & $\checkmark$ & $\checkmark$ & & & & & VH\phantom{X}\\
\multirow{2}{*}{$\overline{z\rho}$} & \multirow{2}{*}{$\biggl\{$} & $\overline{zx}$ & $\mathbf{Z}_2\!\times\!\mathbf{Z}_2$ & $\langle\mu\nu^2,\eta\nu^3\rangle$ & 
& $\checkmark$ & & $\checkmark$ & & $\checkmark$ & & & & & VH\phantom{X}\\
& & $\overline{zy}$ & $\mathbf{Z}_2\!\times\!\mathbf{Z}_2$ & $\langle\mu\nu^2,\eta\nu\rangle$ & 
& & $\checkmark$ & $\checkmark$ & $\checkmark$ & & & & & & VH\phantom{X}\\
$\overline{z}$ & & & $\mathbf{Z}_2$ & $\langle\mu\nu^2\rangle$ & 
& & & $\checkmark$ & & & & & & & VH\phantom{X}\\
\multirow{2}{*}{$\overline{\rho}$} & \multirow{2}{*}{$\biggl\{$} & $\overline{y}$ & $\mathbf{Z}_2$ & $\langle\eta\nu\rangle$ & 
& & & & $\checkmark$ & & & & & & VH\phantom{X}\\
& & $\overline{x}$ & $\mathbf{Z}_2$ & $\langle\eta\nu^3\rangle$ & 
& & & & & $\checkmark$ & & & & & VH\phantom{X}\\
$\mu$ & & & $\mathbf{Z}_2$ & $\langle\mu\rangle$ & 
& & & & & & $\checkmark$ & & & & VH\phantom{X}\\
\midrule

$z\rho$ & & & $\mathbf{Z}_2\!\times\!\mathbf{Z}_2$ & $\langle\nu^2,\mu\eta\nu\rangle$ & 
$\checkmark$ & $\checkmark$ & $\checkmark$ & & & & & & & & VHX\\
$z$ & & & $\mathbf{Z}_2$ & $\langle\nu^2\rangle$ & 
$\checkmark$ & & & & & & & & & & VHX\\
\multirow{2}{*}{$\rho$} & \multirow{2}{*}{$\biggl\{$} & $y$ & $\mathbf{Z}_2$ & $\langle\mu\eta\nu\rangle$ & 
& $\checkmark$ & & & & & & & & & VHX\\
& & $x$ & $\mathbf{Z}_2$ & $\langle\mu\eta\nu^3\rangle$ & 
& & $\checkmark$ & & & & & & & & VHX\\
$e$ & & & $\mathbf{Z}_1$ & $\langle e\rangle$ & 
& & & & & & & & & & VHX\\
\bottomrule
\end{tabular}
\end{table}

\begin{table}[]
    \caption{Number of all tangle orbits with non-mirror symmetries split by the number of crossings (\qq{\#cross}). The properties of each symmetry group are listed in \autoref{tab:symmetries}. Tangles in these group can be of any V-, H-, or X-type. Tangles in each column/row are summed up (\qq{Orbit} column/row), additionally, tangles are counted and summed up to equivalence (”Equiv.”) and isotopy (”Isotopy”).}
    \label{tab:tangles_sym_rot}
    \begin{tabular}{rr@{\quad}r@{\quad}r@{\quad}rr@{\quad}r@{\quad}r}
        \toprule
        \multirow{2}{*}{\#cross} & \multicolumn{4}{c}{Non-mirror symmetry group} & \multicolumn{3}{c}{Total} \\
        \cmidrule(rl){2-5} \cmidrule(rl){6-8}
        & $z\rho$ & $z$ & $\rho$ & $e$ & Orbit & Equiv. & Isotopy  \\
        \midrule
        0 & - & - & - & - & - & - & - \\
        1 & - & - & - & - & - & - & - \\
        2 & 1 & - & - & - & 1 & 2 & 4 \\
        3 & 2 & - & - & - & 2 & 4 & 8 \\
        4 & 5 & - & - & - & 5 & 10 & 20 \\
        5 & 11 & - & 3 & - & 14 & 28 & 68 \\
        6 & 28 & 4 & 11 & - & 43 & 86 & 232 \\
        7 & 61 & 15 & 52 & 6 & 134 & 268 & 876 \\
        8 & 154 & 64 & 183 & 52 & 453 & 906 & 3424 \\
        9 & 352 & 227 & 675 & 329 & 1583 & 3166 & 13888 \\
        10 & 858 & 827 & 2289 & 1814 & 5788 & 11576 & 57384 \\
        11 & 1968 & 2798 & 7833 & 9327 & 21926 & 43852 & 242152 \\
        12 & 4754 & 9510 & 25961 & 45782 & 86007 & 172014 & 1035296 \\
        13 & 11049 & 31482 & 86257 & 218540 & 347328 & 694656 & 4482748 \\
        14 & 26467 & 104385 & 283466 & 1024521 & 1438839 & 2877678 & 19601012 \\
        \midrule
        Orbit & 45710 & 149312 & 406730 & 1300371 & 1902123 & & \\
        Equiv. & 91420 & 298624 & 813460 & 2600742 & & 3804246 & \\
        Isotopy & 182840 & 1194496 & 3253840 & 20805936 & & & 25437112 \\
        \bottomrule
    \end{tabular}
\end{table}

\begin{table}[]
    \caption{Same as in \autoref{tab:tangles_sym_rot}, but only for tangles with no extra closed components.}
    \label{tab:tangles_sym_rot_0}
    \begin{tabular}{rrrrrrrr}
        \toprule
        \multirow{2}{*}{\#cross} & \multicolumn{4}{c}{Non-mirror symmetry group} & \multicolumn{3}{c}{Total} \\
        \cmidrule(rl){2-5} \cmidrule(rl){6-8}
        & $z\rho$ & $z$ & $\rho$ & $e$ & Orbit & Equiv. & Isotopy  \\
        \midrule
        0 & - & - & - & - & - & - & - \\
        1 & - & - & - & - & - & - & - \\
        2 & 1 & - & - & - & 1 & 2 & 4 \\
        3 & 2 & - & - & - & 2 & 4 & 8 \\
        4 & 4 & - & - & - & 4 & 8 & 16 \\
        5 & 8 & - & 3 & - & 11 & 22 & 56 \\
        6 & 19 & 4 & 6 & - & 29 & 58 & 156 \\
        7 & 38 & 8 & 34 & 6 & 86 & 172 & 584 \\
        8 & 83 & 41 & 101 & 41 & 266 & 532 & 2124 \\
        9 & 173 & 120 & 332 & 218 & 843 & 1686 & 7796 \\
        10 & 373 & 391 & 1027 & 1110 & 2901 & 5802 & 30596 \\
        11 & 781 & 1201 & 3075 & 5044 & 10101 & 20202 & 118036 \\
        12 & 1688 & 3562 & 9156 & 22688 & 37094 & 74188 & 471504 \\
        13 & 3551 & 10551 & 26999 & 98674 & 139775 & 279550 & 1893388 \\
        14 & 7621 & 31091 & 78626 & 422447 & 539785 & 1079570 & 7667372 \\
        \midrule
        Orbit & 14342 & 46969 & 119359 & 550228 & 730898 & & \\
        Equiv. & 28684 & 93938 & 238718 & 1100456 & & 1461796 & \\
        Isotopy & 57368 & 375752 & 954872 & 8803648 & & & 10191640 \\
        \bottomrule
    \end{tabular}
\end{table}

\begin{table}[]
    \caption{Number of all tangle orbits with mirror symmetries divided by number of crossings ("\#cross"). Mirror symmetries are divided into two groups, available either only for X-type or VH-type tangles.\\
    Properties of each symmetry group are listed in \autoref{tab:symmetries}. Tangles in each column/row are added up (”Orbit” column/row). Additionally tangles are counted and added up to equivalence (”Equiv.”) and isotopy (”Isotopy”).}
    \label{tab:tangles_sym_mir}
    \begin{tabular}{rr@{\quad}r@{\quad}r@{\quad}rr@{\quad}r@{\quad}r@{\quad}r@{\quad}r@{\quad}r@{\quad}r@{\quad}rr@{\quad}r@{\quad}r}
        \toprule
        \multirow{2}{*}{\#cross} & \multicolumn{4}{c}{X-only} & 
        \multicolumn{8}{c}{VH-only} & \multicolumn{3}{c}{Total} \\
        \cmidrule(rl){2-5} \cmidrule(rl){6-13} \cmidrule(rl){14-16}
        & $\mathfrak{1}$ & $\overline{\nu}$ &  
        $\eta z$ & $\eta$ & 
        $\mathfrak{0}$ & $\mu z$ & $\mu \rho$ & 
        $\overline{xy}$ & $\overline{z\rho}$ & $\overline{z}$ & $\overline{\rho}$ & $\mu$ & Orbit & Equiv. & Isotopy \\
        \midrule
        0 & - & - & - & - & 1 & - & - & - & - & - & - & - & 1 & 1 & 2 \\
        1 & 1 & - & - & - & - & - & - & - & - & - & - & - & 1 & 1 & 2 \\
        2 & - & - & - & - & - & - & - & - & - & - & - & - & - & - & - \\
        3 & - & - & - & - & - & - & - & - & - & - & - & - & - & - & - \\
        4 & - & - & - & - & 1 & - & - & - & - & - & - & - & 1 & 1 & 2 \\
        5 & - & - & - & - & - & - & - & - & - & - & - & - & - & - & - \\
        6 & - & - & - & - & - & - & - & - & 1 & - & - & - & 1 & 1 & 4 \\
        7 & 2 & - & - & - & - & - & - & - & - & - & - & - & 2 & 2 & 4 \\
        8 & - & - & - & - & 3 & - & - & - & 3 & - & - & - & 6 & 6 & 18 \\
        9 & - & 1 & - & 1 & - & - & - & - & - & - & - & - & 2 & 2 & 12 \\
        10 & - & - & - & - & - & 1 & 2 & - & 9 & 3 & 4 & - & 19 & 19 & 104 \\
        11 & 3 & 3 & 3 & 3 & - & - & - & - & - & - & - & - & 12 & 12 & 54 \\
        12 & - & - & - & - & 10 & 3 & 3 & 1 & 32 & 12 & 15 & 1 & 77 & 77 & 400 \\
        13 & 4 & 8 & 2 & 20 & - & - & - & - & - & - & - & - & 34 & 34 & 208 \\
        14 & - & - & - & - & - & 10 & 17 & 2 & 89 & 77 & 98 & 10 & 303 & 303 & 1952 \\
        \midrule
        Orbit & 10 & 12 & 5 & 24 & 15 & 14 & 22 & 3 & 134 & 92 & 117 & 11 & 459 & & \\
        Equiv. & 10 & 12 & 5 & 24 & 15 & 14 & 22 & 3 & 134 & 92 & 117 & 11 & & 459 & \\
        Isotopy & 20 & 48 & 20 & 192 & 30 & 56 & 88 & 12 & 536 & 736 & 936 & 88 & & & 2762 \\
        \bottomrule
    \end{tabular}
\end{table}

\begin{table}[]
    \caption{Same as in \autoref{tab:tangles_sym_mir}, but only for tangles with no extra closed components.}
    \label{tab:tangles_sym_mir_0}
    \begin{tabular}{rr@{\quad}rr@{\quad}r@{\quad}r@{\quad}r@{\quad}r@{\quad}r@{\quad}r@{\quad}r@{\quad}r@{\quad}rr@{\quad}r@{\quad}r}
        \toprule
        \multirow{2}{*}{\#cross} & \multicolumn{2}{c}{X-only} & 
        \multicolumn{4}{c}{VH-only} & \multicolumn{3}{c}{Total} \\
        \cmidrule(rl){2-3} \cmidrule(rl){4-7} \cmidrule(rl){8-10}
        & $\mathfrak{1}$ & $\mu\nu$ & $\mathfrak{0}$ & $\overline{z}\rho$ & $\overline{z}$ & $\overline{\rho}$ & Orbit & Equiv. & Isotopy \\
        \midrule
        0 & - & - & 1 & - & - & - & 1 & 1 & 2 \\
        1 & 1 & - & - & - & - & - & 1 & 1 & 2 \\
        2 & - & - & - & - & - & - & - & - & - \\
        3 & - & - & - & - & - & - & - & - & - \\
        4 & - & - & - & - & - & - & - & - & - \\
        5 & - & - & - & - & - & - & - & - & - \\
        6 & - & - & - & 1 & - & - & 1 & 1 & 4 \\
        7 & - & - & - & - & - & - & - & - & - \\
        8 & - & - & - & 1 & - & - & 1 & 1 & 4 \\
        9 & - & 1 & - & - & - & - & 1 & 1 & 4 \\
        10 & - & - & - & 3 & 3 & 3 & 9 & 9 & 60 \\
        11 & - & 1 & - & - & - & - & 1 & 1 & 4 \\
        12 & - & - & - & 8 & 7 & 6 & 21 & 21 & 136 \\
        13 & - & 3 & - & - & - & - & 3 & 3 & 12 \\
        14 & - & - & - & 17 & 36 & 34 & 87 & 87 & 628 \\
        \midrule
        Orbit & 1 & 5 & 1 & 30 & 46 & 43 & 126 & & \\
        Equiv. & 1 & 5 & 1 & 30 & 46 & 43 & & 126 & \\
        Isotopy & 2 & 20 & 2 & 120 & 368 & 344 & & & 856 \\
        \bottomrule
    \end{tabular}
\end{table}

\FloatBarrier

\section*{Declarations}

\paragraph{Funding}
This work was supported by the National Science Centre (\#2021/43/I/NZ1/03341 to JIS), University of Warsaw IDUB grant (BOB-IDUB-622-823/2023 to BAG) and the Slovenian Research and Innovation Agency grant N1-0278 to BG. 

\paragraph{Competing interests}
The authors have no relevant financial or non-financial interests to disclose.

\paragraph{Data and code availability}
Database of tangles is available on the website: \url{https://tangleinfo.cent.uw.edu.pl} \cite{tangleinfo}. Programs that run the classification and draw tangles, as well as supplementary material with classification of tangles up to 10 crossings, are available in the GitHub repository: \url{https://github.com/baaagr/Classification_of_algebraic_tangles} \cite{supplement}.

\paragraph{Author contribution}
All authors: conceptualization, funding, writing. 
Bartosz A. Gren: data curation, formal analysis, methodology, software, validation, visualisation.
Bo\v{s}tjan Gabrov\v{s}ek: methodology, software, supervision, validation, visualisation.
Joanna I. Sulkowska: project administration, supervision.




\begin{thebibliography}{33}
\ifx \bisbn   \undefined \def \bisbn  #1{ISBN #1}\fi
\ifx \binits  \undefined \def \binits#1{#1}\fi
\ifx \bauthor  \undefined \def \bauthor#1{#1}\fi
\ifx \batitle  \undefined \def \batitle#1{#1}\fi
\ifx \bjtitle  \undefined \def \bjtitle#1{#1}\fi
\ifx \bvolume  \undefined \def \bvolume#1{\textbf{#1}}\fi
\ifx \byear  \undefined \def \byear#1{#1}\fi
\ifx \bissue  \undefined \def \bissue#1{#1}\fi
\ifx \bfpage  \undefined \def \bfpage#1{#1}\fi
\ifx \blpage  \undefined \def \blpage #1{#1}\fi
\ifx \burl  \undefined \def \burl#1{\textsf{#1}}\fi
\ifx \doiurl  \undefined \def \doiurl#1{\url{https://doi.org/#1}}\fi
\ifx \betal  \undefined \def \betal{\textit{et al.}}\fi
\ifx \binstitute  \undefined \def \binstitute#1{#1}\fi
\ifx \binstitutionaled  \undefined \def \binstitutionaled#1{#1}\fi
\ifx \bctitle  \undefined \def \bctitle#1{#1}\fi
\ifx \beditor  \undefined \def \beditor#1{#1}\fi
\ifx \bpublisher  \undefined \def \bpublisher#1{#1}\fi
\ifx \bbtitle  \undefined \def \bbtitle#1{#1}\fi
\ifx \bedition  \undefined \def \bedition#1{#1}\fi
\ifx \bseriesno  \undefined \def \bseriesno#1{#1}\fi
\ifx \blocation  \undefined \def \blocation#1{#1}\fi
\ifx \bsertitle  \undefined \def \bsertitle#1{#1}\fi
\ifx \bsnm \undefined \def \bsnm#1{#1}\fi
\ifx \bsuffix \undefined \def \bsuffix#1{#1}\fi
\ifx \bparticle \undefined \def \bparticle#1{#1}\fi
\ifx \barticle \undefined \def \barticle#1{#1}\fi
\bibcommenthead
\ifx \bconfdate \undefined \def \bconfdate #1{#1}\fi
\ifx \botherref \undefined \def \botherref #1{#1}\fi
\ifx \url \undefined \def \url#1{\textsf{#1}}\fi
\ifx \bchapter \undefined \def \bchapter#1{#1}\fi
\ifx \bbook \undefined \def \bbook#1{#1}\fi
\ifx \bcomment \undefined \def \bcomment#1{#1}\fi
\ifx \oauthor \undefined \def \oauthor#1{#1}\fi
\ifx \citeauthoryear \undefined \def \citeauthoryear#1{#1}\fi
\ifx \endbibitem  \undefined \def \endbibitem {}\fi
\ifx \bconflocation  \undefined \def \bconflocation#1{#1}\fi
\ifx \arxivurl  \undefined \def \arxivurl#1{\textsf{#1}}\fi
\csname PreBibitemsHook\endcsname

\bibitem[\protect\citeauthoryear{Conway}{1970}]{conway1970enumeration}
\begin{bchapter}
\bauthor{\bsnm{Conway}, \binits{J.H.}}:
\bctitle{An enumeration of knots and links, and some of their algebraic properties}.
In: \bbtitle{Computational Problems in Abstract Algebra},
pp. \bfpage{329}--\blpage{358}
(\byear{1970}).
\bcomment{Elsevier}
\end{bchapter}
\endbibitem

\bibitem[\protect\citeauthoryear{Bonahon and Siebenmann}{2010}]{bonahon2010new}
\begin{botherref}
\oauthor{\bsnm{Bonahon}, \binits{F.}},
\oauthor{\bsnm{Siebenmann}, \binits{L.}}:
New geometric splittings of classical knots and the classification and symmetries of arborescent knots.
preprint
(2010)
\end{botherref}
\endbibitem

\bibitem[\protect\citeauthoryear{Dowker and Thistlethwaite}{1983}]{dowker1983classification}
\begin{barticle}
\bauthor{\bsnm{Dowker}, \binits{C.H.}},
\bauthor{\bsnm{Thistlethwaite}, \binits{M.B.}}:
\batitle{Classification of knot projections}.
\bjtitle{Topology and its Applications}
\bvolume{16}(\bissue{1}),
\bfpage{19}--\blpage{31}
(\byear{1983})
\end{barticle}
\endbibitem

\bibitem[\protect\citeauthoryear{Hoste et~al.}{1998}]{hoste1998first}
\begin{barticle}
\bauthor{\bsnm{Hoste}, \binits{J.}},
\bauthor{\bsnm{Thistlethwaite}, \binits{M.}},
\bauthor{\bsnm{Weeks}, \binits{J.}}:
\batitle{The first 1,701,936 knots}.
\bjtitle{Math. Intelligencer}
\bvolume{20}(\bissue{4}),
\bfpage{33}--\blpage{48}
(\byear{1998})
\end{barticle}
\endbibitem

\bibitem[\protect\citeauthoryear{Burton}{2020}]{burton2020next}
\begin{bchapter}
\bauthor{\bsnm{Burton}, \binits{B.A.}}:
\bctitle{The next 350 million knots}.
In: \bbtitle{36th International Symposium on Computational Geometry (SoCG 2020)}
(\byear{2020}).
\bcomment{Schloss-Dagstuhl-Leibniz Zentrum f{\"u}r Informatik}
\end{bchapter}
\endbibitem

\bibitem[\protect\citeauthoryear{Bar-Natan}{2007}]{bar2007fast}
\begin{barticle}
\bauthor{\bsnm{Bar-Natan}, \binits{D.}}:
\batitle{Fast khovanov homology computations}.
\bjtitle{Journal of Knot Theory and Its Ramifications}
\bvolume{16}(\bissue{03}),
\bfpage{243}--\blpage{255}
(\byear{2007})
\end{barticle}
\endbibitem

\bibitem[\protect\citeauthoryear{Mansfield}{1994}]{mansfield1994there}
\begin{barticle}
\bauthor{\bsnm{Mansfield}, \binits{M.L.}}:
\batitle{Are there knots in proteins?}
\bjtitle{Nature structural biology}
\bvolume{1}(\bissue{4}),
\bfpage{213}
(\byear{1994})
\end{barticle}
\endbibitem

\bibitem[\protect\citeauthoryear{Taylor}{2000}]{taylor2000deeply}
\begin{barticle}
\bauthor{\bsnm{Taylor}, \binits{W.R.}}:
\batitle{A deeply knotted protein structure and how it might fold}.
\bjtitle{Nature}
\bvolume{406}(\bissue{6798}),
\bfpage{916}
(\byear{2000})
\end{barticle}
\endbibitem

\bibitem[\protect\citeauthoryear{Sulkowska et~al.}{2012}]{sulkowska2012conservation}
\begin{barticle}
\bauthor{\bsnm{Sulkowska}, \binits{J.I.}},
\bauthor{\bsnm{Rawdon}, \binits{E.J.}},
\bauthor{\bsnm{Millet}, \binits{K.C.}},
\bauthor{\bsnm{Onuchic}, \binits{J.N.}},
\bauthor{\bsnm{Stasiak}, \binits{A.}}:
\batitle{Conservation of complex knotting and slipknotting patterns in proteins}.
\bjtitle{Biophysical Journal}
\bvolume{102}(\bissue{3}),
\bfpage{253}
(\byear{2012})
\end{barticle}
\endbibitem

\bibitem[\protect\citeauthoryear{Niemyska et~al.}{2016}]{niemyska2016complex}
\begin{barticle}
\bauthor{\bsnm{Niemyska}, \binits{W.}},
\bauthor{\bsnm{Dabrowski-Tumanski}, \binits{P.}},
\bauthor{\bsnm{Kadlof}, \binits{M.}},
\bauthor{\bsnm{Haglund}, \binits{E.}},
\bauthor{\bsnm{Su{\l}kowski}, \binits{P.}},
\bauthor{\bsnm{Sulkowska}, \binits{J.I.}}:
\batitle{Complex lasso: new entangled motifs in proteins}.
\bjtitle{Scientific reports}
\bvolume{6},
\bfpage{36895}
(\byear{2016})
\end{barticle}
\endbibitem

\bibitem[\protect\citeauthoryear{Dabrowski-Tumanski et~al.}{2016}]{dabrowski2016lassoprot}
\begin{barticle}
\bauthor{\bsnm{Dabrowski-Tumanski}, \binits{P.}},
\bauthor{\bsnm{Niemyska}, \binits{W.}},
\bauthor{\bsnm{Pasznik}, \binits{P.}},
\bauthor{\bsnm{Sulkowska}, \binits{J.I.}}:
\batitle{Lassoprot: server to analyze biopolymers with lassos}.
\bjtitle{Nucleic acids research}
\bvolume{44}(\bissue{W1}),
\bfpage{383}--\blpage{389}
(\byear{2016})
\end{barticle}
\endbibitem

\bibitem[\protect\citeauthoryear{Gre{\'n} et~al.}{2021}]{gren2021lasso}
\begin{barticle}
\bauthor{\bsnm{Gre{\'n}}, \binits{B.A.}},
\bauthor{\bsnm{Dabrowski-Tumanski}, \binits{P.}},
\bauthor{\bsnm{Niemyska}, \binits{W.}},
\bauthor{\bsnm{Sulkowska}, \binits{J.I.}}:
\batitle{Lasso proteins—unifying cysteine knots and miniproteins}.
\bjtitle{Polymers}
\bvolume{13}(\bissue{22}),
\bfpage{3988}
(\byear{2021})
\end{barticle}
\endbibitem

\bibitem[\protect\citeauthoryear{Moriuchi}{2008}]{moriuchi2008enumeration}
\begin{barticle}
\bauthor{\bsnm{Moriuchi}, \binits{H.}}:
\batitle{Enumeration of algebraic tangles with applications to theta-curves and handcuff graphs}.
\bjtitle{Kyungpook Mathematical Journal}
\bvolume{48}(\bissue{3}),
\bfpage{337}--\blpage{357}
(\byear{2008})
\end{barticle}
\endbibitem

\bibitem[\protect\citeauthoryear{Dabrowski-Tumanski et~al.}{2024}]{dabrowski2024theta}
\begin{barticle}
\bauthor{\bsnm{Dabrowski-Tumanski}, \binits{P.}},
\bauthor{\bsnm{Goundaroulis}, \binits{D.}},
\bauthor{\bsnm{Stasiak}, \binits{A.}},
\bauthor{\bsnm{Rawdon}, \binits{E.J.}},
\bauthor{\bsnm{Sulkowska}, \binits{J.I.}}:
\batitle{Theta-curves in proteins}.
\bjtitle{Protein science}
\bvolume{33}(\bissue{9}),
\bfpage{5133}
(\byear{2024})
\end{barticle}
\endbibitem

\bibitem[\protect\citeauthoryear{Bruno~da Silva et~al.}{2024}]{bruno2024knots}
\begin{barticle}
\bauthor{\bsnm{Silva}, \binits{F.}},
\bauthor{\bsnm{Gabrovšek}, \binits{B.}},
\bauthor{\bsnm{Korpacz}, \binits{M.}},
\bauthor{\bsnm{Luczkiewicz}, \binits{K.}},
\bauthor{\bsnm{Niewieczerzal}, \binits{S.}},
\bauthor{\bsnm{Sikora}, \binits{M.}},
\bauthor{\bsnm{Sulkowska}, \binits{J.I.}}:
\batitle{Knots and $\theta$-curves identification in polymeric chains and native proteins using neural networks}.
\bjtitle{Macromolecules}
\bvolume{57}(\bissue{9}),
\bfpage{4599}--\blpage{4608}
(\byear{2024})
\end{barticle}
\endbibitem

\bibitem[\protect\citeauthoryear{Sulkowska}{2020}]{sulkowska2020folding}
\begin{barticle}
\bauthor{\bsnm{Sulkowska}, \binits{J.I.}}:
\batitle{On folding of entangled proteins: knots, lassos, links and $\theta$-curves}.
\bjtitle{Current opinion in structural biology}
\bvolume{60},
\bfpage{131}--\blpage{141}
(\byear{2020})
\end{barticle}
\endbibitem

\bibitem[\protect\citeauthoryear{Gabrov{\v{s}}ek}{2021}]{gabrovvsek2021invariant}
\begin{barticle}
\bauthor{\bsnm{Gabrov{\v{s}}ek}, \binits{B.}}:
\batitle{An invariant for colored bonded knots}.
\bjtitle{Studies in Applied Mathematics}
\bvolume{146}(\bissue{3}),
\bfpage{586}--\blpage{604}
(\byear{2021})
\end{barticle}
\endbibitem

\bibitem[\protect\citeauthoryear{Adams et~al.}{2020}]{adams2020knot}
\begin{barticle}
\bauthor{\bsnm{Adams}, \binits{C.}},
\bauthor{\bsnm{Devadoss}, \binits{J.}},
\bauthor{\bsnm{Elhamdadi}, \binits{M.}},
\bauthor{\bsnm{Mashaghi}, \binits{A.}}:
\batitle{Knot theory for proteins: Gauss codes, quandles and bondles}.
\bjtitle{Journal of Mathematical Chemistry}
\bvolume{58},
\bfpage{1711}--\blpage{1736}
(\byear{2020})
\end{barticle}
\endbibitem

\bibitem[\protect\citeauthoryear{Dabrowski-Tumanski et~al.}{2019}]{dabrowski2019knotprot}
\begin{barticle}
\bauthor{\bsnm{Dabrowski-Tumanski}, \binits{P.}},
\bauthor{\bsnm{Rubach}, \binits{P.}},
\bauthor{\bsnm{Goundaroulis}, \binits{D.}},
\bauthor{\bsnm{Dorier}, \binits{J.}},
\bauthor{\bsnm{Su{\l}kowski}, \binits{P.}},
\bauthor{\bsnm{Millett}, \binits{K.C.}},
\bauthor{\bsnm{Rawdon}, \binits{E.J.}},
\bauthor{\bsnm{Stasiak}, \binits{A.}},
\bauthor{\bsnm{Sulkowska}, \binits{J.I.}}:
\batitle{Knotprot 2.0: a database of proteins with knots and other entangled structures}.
\bjtitle{Nucleic acids research}
\bvolume{47}(\bissue{D1}),
\bfpage{367}--\blpage{375}
(\byear{2019})
\end{barticle}
\endbibitem

\bibitem[\protect\citeauthoryear{Dabrowski-Tumanski et~al.}{2016}]{dabrowski2016linkprot}
\begin{botherref}
\oauthor{\bsnm{Dabrowski-Tumanski}, \binits{P.}},
\oauthor{\bsnm{Jarmolinska}, \binits{A.I.}},
\oauthor{\bsnm{Niemyska}, \binits{W.}},
\oauthor{\bsnm{Rawdon}, \binits{E.J.}},
\oauthor{\bsnm{Millett}, \binits{K.C.}},
\oauthor{\bsnm{Sulkowska}, \binits{J.I.}}:
Linkprot: a database collecting information about biological links.
Nucleic acids research,
976
(2016)
\end{botherref}
\endbibitem

\bibitem[\protect\citeauthoryear{Dabrowski-Tumanski and Sulkowska}{2017}]{dabrowski2017topological}
\begin{barticle}
\bauthor{\bsnm{Dabrowski-Tumanski}, \binits{P.}},
\bauthor{\bsnm{Sulkowska}, \binits{J.I.}}:
\batitle{Topological knots and links in proteins}.
\bjtitle{Proceedings of the National Academy of Sciences}
\bvolume{114}(\bissue{13}),
\bfpage{3415}--\blpage{3420}
(\byear{2017})
\end{barticle}
\endbibitem

\bibitem[\protect\citeauthoryear{Goundaroulis et~al.}{2020}]{goundaroulis2020knotoids}
\begin{barticle}
\bauthor{\bsnm{Goundaroulis}, \binits{D.}},
\bauthor{\bsnm{Dorier}, \binits{J.}},
\bauthor{\bsnm{Stasiak}, \binits{A.}}:
\batitle{Knotoids and protein structure}.
\bjtitle{Topol. Geom. Biopolym}
\bvolume{746},
\bfpage{185}
(\byear{2020})
\end{barticle}
\endbibitem

\bibitem[\protect\citeauthoryear{Gabrov{\v{s}}ek and G{\"u}g{\"u}mc{\"u}}{2023}]{gabrovvsek2023invariants}
\begin{barticle}
\bauthor{\bsnm{Gabrov{\v{s}}ek}, \binits{B.}},
\bauthor{\bsnm{G{\"u}g{\"u}mc{\"u}}, \binits{N.}}:
\batitle{Invariants of multi-linkoids}.
\bjtitle{Mediterranean journal of mathematics}
\bvolume{20}(\bissue{3}),
\bfpage{165}
(\byear{2023})
\end{barticle}
\endbibitem

\bibitem[\protect\citeauthoryear{G{\"u}g{\"u}mc{\"u} et~al.}{2022}]{gugumcu2022invariants}
\begin{barticle}
\bauthor{\bsnm{G{\"u}g{\"u}mc{\"u}}, \binits{N.}},
\bauthor{\bsnm{Gabrovsek}, \binits{B.}},
\bauthor{\bsnm{Kauffman}, \binits{L.H.}}:
\batitle{Invariants of bonded knotoids and applications to protein folding}.
\bjtitle{Symmetry}
\bvolume{14}(\bissue{8}),
\bfpage{1724}
(\byear{2022})
\end{barticle}
\endbibitem

\bibitem[\protect\citeauthoryear{Cameron et~al.}{2023}]{cameron2023coinhibition}
\begin{barticle}
\bauthor{\bsnm{Cameron}, \binits{D.P.}},
\bauthor{\bsnm{Grosser}, \binits{J.}},
\bauthor{\bsnm{Ladigan}, \binits{S.}},
\bauthor{\bsnm{Kuzin}, \binits{V.}},
\bauthor{\bsnm{Iliopoulou}, \binits{E.}},
\bauthor{\bsnm{Wiegard}, \binits{A.}},
\bauthor{\bsnm{Benredjem}, \binits{H.}},
\bauthor{\bsnm{Jackson}, \binits{K.}},
\bauthor{\bsnm{Liffers}, \binits{S.T.}},
\bauthor{\bsnm{Lueong}, \binits{S.}}, \betal:
\batitle{Coinhibition of topoisomerase 1 and brd4-mediated pause release selectively kills pancreatic cancer via readthrough transcription}.
\bjtitle{Science Advances}
\bvolume{9}(\bissue{41}),
\bfpage{5109}
(\byear{2023})
\end{barticle}
\endbibitem

\bibitem[\protect\citeauthoryear{Ernst and Sumners}{1990}]{Ernst_Sumners_1990}
\begin{barticle}
\bauthor{\bsnm{Ernst}, \binits{C.}},
\bauthor{\bsnm{Sumners}, \binits{D.W.}}:
\batitle{A calculus for rational tangles: applications to dna recombination}.
\bjtitle{Mathematical Proceedings of the Cambridge Philosophical Society}
\bvolume{108}(\bissue{3}),
\bfpage{489}--\blpage{515}
(\byear{1990})
\doiurl{10.1017/S0305004100069383}
\end{barticle}
\endbibitem

\bibitem[\protect\citeauthoryear{Sumners}{2011}]{sumners2011dna}
\begin{bchapter}
\bauthor{\bsnm{Sumners}, \binits{D.W.}}:
\bctitle{Dna, knots and tangles}.
In: \bbtitle{The Mathematics of Knots: Theory and Application},
pp. \bfpage{327}--\blpage{353}.
\bpublisher{Springer},
\blocation{Berlin, Heidelberg}
(\byear{2011})
\end{bchapter}
\endbibitem

\bibitem[\protect\citeauthoryear{Kauffman and Lambropoulou}{2004}]{kauffman2004classification}
\begin{barticle}
\bauthor{\bsnm{Kauffman}, \binits{L.H.}},
\bauthor{\bsnm{Lambropoulou}, \binits{S.}}:
\batitle{On the classification of rational tangles}.
\bjtitle{Advances in Applied Mathematics}
\bvolume{33}(\bissue{2}),
\bfpage{199}--\blpage{237}
(\byear{2004})
\end{barticle}
\endbibitem

\bibitem[\protect\citeauthoryear{Burde and Zieschang}{1986}]{burde1986knots}
\begin{bbook}
\bauthor{\bsnm{Burde}, \binits{G.}},
\bauthor{\bsnm{Zieschang}, \binits{H.}}:
\bbtitle{Knots}
vol. \bseriesno{5}.
\bpublisher{de Gruyter Stud. Math.},
\blocation{Berlin, New York}
(\byear{1986})
\end{bbook}
\endbibitem

\bibitem[\protect\citeauthoryear{Montesinos}{1984}]{montesinos1984revetements}
\begin{botherref}
\oauthor{\bsnm{Montesinos}, \binits{J.M.}}:
Rev{\^e}tements ramifi{\'e}s de noeuds, espaces fibr{\'e}s de Seifert et scindements de Heegaard.
Universidad de Zaragoza. Seminario Matematico Garcia de Galdeano
(1984)
\end{botherref}
\endbibitem

\bibitem[\protect\citeauthoryear{Goldman and Kauffman}{1997}]{goldman1997rational}
\begin{barticle}
\bauthor{\bsnm{Goldman}, \binits{J.R.}},
\bauthor{\bsnm{Kauffman}, \binits{L.H.}}:
\batitle{Rational tangles}.
\bjtitle{Advances in Applied Mathematics}
\bvolume{18}(\bissue{3}),
\bfpage{300}--\blpage{332}
(\byear{1997})
\end{barticle}
\endbibitem

\bibitem[\protect\citeauthoryear{}{2025}]{supplement}
\begin{botherref}
Classification of algebraic tangles -- code and supplement
(2025).
\url{https://github.com/baaagr/Classification of algebraic tangles}
\end{botherref}
\endbibitem

\bibitem[\protect\citeauthoryear{}{2025}]{tangleinfo}
\begin{botherref}
TangleInfo
(2025).
\url{https://tangleinfo.cent.uw.edu.pl}
\end{botherref}
\endbibitem

\end{thebibliography}

\end{document}